\documentclass[12pt,francais]{article}
\usepackage{amsmath}\usepackage{epsf,amsfonts,amsthm}\usepackage{amscd,amssymb}
\usepackage{xcolor,epic,eepic}\usepackage{epsfig}
\usepackage{indentfirst, delarray}
\usepackage{rotating}
\usepackage{mathdots, mathtools}
\usepackage[normalem]{ulem}
\usepackage[matrix,arrow,curve]{xy}
\usepackage{graphicx}
\usepackage{fancyhdr}
\usepackage{dsfont,texdraw}
\usepackage{mathrsfs}
\usepackage{latexsym}
\usepackage{hyperref}
\usepackage{tikz-cd}
\usepackage{tikz, graphicx}
\newsavebox{\pullback}
\sbox\pullback{%
\begin{tikzpicture}%
\draw (0,0) -- (1ex,0ex);%
\draw (1ex,0ex) -- (1ex,1ex);%
\end{tikzpicture}}
\theoremstyle{plain}

\usepackage{tikz}
\usetikzlibrary{matrix,arrows,decorations.pathmorphing}
\usepackage{scalerel,stackengine}
\usetikzlibrary{arrows,chains,matrix,positioning,scopes,snakes}
\makeatletter
\tikzset{join/.code=\tikzset{after node path={%
\ifx\tikzchainprevious\pgfutil@empty\else(\tikzchainprevious)%
edge[every join]#1(\tikzchaincurrent)\fi}}}
\makeatother

\tikzset{>=stealth',every on chain/.append style={join},
         every join/.style={->}}
\tikzset{
    >=stealth',
    punkt/.style={
           rectangle,
           rounded corners,
           draw=black, very thick,
           text width=6.5em,
           minimum height=2em,
           text centered},
    pil/.style={
           ->,
           thick,
           shorten <=2pt,
           shorten >=2pt,}
}

\newcommand{\longsquiggly}{\xymatrix{{}\ar@{~>}[r]&{}}}

\newcommand{\bb}{\bullet}
\newcommand{\bee}{\begin{enumerate}}
\newcommand{\eee}{\end{enumerate}}
\newcommand{\benn}{\begin{equation*}}
\newcommand{\eenn}{\end{equation*}}
\newcommand{\be}{\begin{equation}}
\newcommand{\ee}{\end{equation}}
\newcommand{\bean}{\begin{eqnarray}}
\newcommand{\eean}{\end{eqnarray}}
\newcommand{\bea}{\begin{eqnarray*}}
\newcommand{\eea}{\end{eqnarray*}}

\newcommand{\ff}{\mathfrak{f}}
\newcommand{\fF}{\mathfrak{F}}

\newcommand{\Z}{\mathbb{Z}}
\newcommand{\R}{\mathbb{R}}

\newcommand{\cF}{{\cal F}}

\newcommand{\op}[1]{\!\!\mathop{\rm ~#1}\nolimits}
\newcommand{\id}{\op{id}}

\newcommand{\zig}{\approx }
\newcommand{\s}{{\scriptscriptstyle\to}}
\mathchardef\za="710B  
\mathchardef\zb="710C  
\mathchardef\zg="710D  
\mathchardef\zd="710E  
\mathchardef\zve="710F 
\mathchardef\zz="7110  
\mathchardef\zh="7111  
\mathchardef\zy="7112 

\mathchardef\zi="7113  
\mathchardef\zk="7114  
\mathchardef\zl="7115  
\mathchardef\zm="7116  
\mathchardef\zn="7117  
\mathchardef\zx="7118  
\mathchardef\zp="7119  
\mathchardef\zr="711A  
\mathchardef\zs="711B  
\mathchardef\zt="711C  
\mathchardef\zu="711D  
\mathchardef\zf="711E 
\mathchardef\zq="711F  
\mathchardef\zc="7120  
\mathchardef\zw="7121  
\mathchardef\ze="7122  
\mathchardef\zvy="7123  
\mathchardef\zvw="7124  
\mathchardef\zvr="7125 
\mathchardef\zvs="7126 
\mathchardef\zvf="7127  
\mathchardef\zG="7000  
\mathchardef\zD="7001  
\mathchardef\zY="7002  
\mathchardef\zL="7003  
\mathchardef\zX="7004  
\mathchardef\zP="7005  
\mathchardef\zS="7006  
\mathchardef\zU="7007  
\mathchardef\zF="7008  
\mathchardef\zW="700A  

 \newcommand{\cK}{{\cal K}}

 \newcommand{\cP}{{\cal P}}
 \newcommand{\cC}{{\cal C}}
 \newcommand{\cA}{{\cal A}}

 \newcommand{\cI}{{\cal I}}
 \newcommand{\cX}{{\cal X}}
 
 \newcommand{\cG}{{\cal G}}
 
 \newcommand{\mfa}{\mathfrak{A}}

\newtheorem{rem}{Remark}
\newtheorem{theo}{Theorem}
\newtheorem{prop}{Proposition}
\newtheorem{lem}{Lemma}
\newtheorem{cor}{Corollary}

\newtheorem{defi}{Definition}

\newcommand{\ul}{\underline}

\newcommand{\0}{\otimes}

\definecolor{cof}{RGB}{219,144,71}
\definecolor{pur}{RGB}{186,146,162}
\definecolor{greeo}{RGB}{91,173,69}
\definecolor{greet}{RGB}{52,111,72}

\begin{document}

\title{A NEW APPROACH TO MODEL CATEGORICAL HOMOTOPY FIBER
SEQUENCES}
\author{Alisa Govzmann, Damjan Pi\v{s}talo, and Norbert Poncin}
\maketitle

\begin{abstract} We propose a simplified definition of Quillen's fibration sequences in a pointed model category that fully captures the theory, although it is completely independent of the concept of action. This advantage arises from the understanding that the homotopy theory of the model category's arrow category contains all homotopical information about its long fibration sequences.\end{abstract}

\small{\vspace{2mm} \noindent {\bf MSC 2020}: 18E35, 18N40, 14A30 \medskip

}

\tableofcontents

\thispagestyle{empty}

\section{Introduction}

Homotopy fiber sequences, also called fibration sequences, have been studied in the category of topological spaces, the category of chain complexes of modules, in general model categories and in homotopical categories. For instance, if $f:A\to B$ is a chain map, its shifted mapping cone $\op{Mc}(f)[-1]$ is the homotopy fiber of $f$ and the associated distinguished triangle $\op{Mc}(f)[-1]\to A\to B$ is a homotopy fiber sequence.\medskip

Quillen \cite{Quill} defined fibration sequences in the homotopy category of any pointed model category $(\tt M,0)\,.$ For this he first defined a loop space functor $\zW^Q$ from a path space functor and observed that the loop space $\zW^Q\cF$ of a fibrant object $\cF$ is a group object in the homotopy category $\tt Ho(M)$ of $\tt M\,.$ Up to  $\tt Ho(M)$--isomorphism a fibration sequence is then a $\tt Ho(M)$--sequence $K\to F\to \cF$ that is implemented by the kernel $K$ of a fibration $F\to \cF$ between fibrant objects $F$ and $\cF\,,$ together with an action that is up to isomorphism some action of the group object $\zW^Q\cF$ on $K\,.$ This action induces a connecting $\tt Ho(M)$--morphism $\zW^Q\cF\to K$ and the sequence $\zW^Q\cF\to K\to F$ is again a fibration sequence.\medskip

In this paper, we also work in a general pointed model category $({\tt M},0)\,,$ we define a loop space functor $\zW$ from any `dual cone functor' and define homotopy fiber sequences as commutative $\tt M$--squares $(A,B,C,D)$ such that $A$ is a specific type of generalized representative of the homotopy pullback of $C\to D\leftarrow B$ and the map $C\to 0$ is a weak equivalence. Further, for every morphism $f:F\to\cF$ we define its homotopy fiber $K_f$ such that $K_f\to F\to\cF$ is a homotopy fiber sequence. We get a universal connecting morphism $\zW\cF\to K_f$ such that $\zW\cF\to K_f\to F$ is also a homotopy fiber sequence.\medskip

It turns out that Quillen's loop space functor $\zW^Q$ is a loop space functor $\zW$ in the sense of the present paper. Further, an objectwise fibrant homotopy fiber sequence is a fibration sequence and our universal connecting morphism is the same as Quillen's connecting morphism induced by the action.\medskip

Although, as the previous descriptions show, the two theories are closely related, our approach to homotopy fiber sequences or fibration sequences, does not rely on the additional structure of an action. The point is that we use the homotopy theory of the category ${\tt M}^\s$ of $\tt M$--morphisms, which contains all relevant information about homotopy fiber sequences of $\tt M\,.$\medskip

More precisely, the paper is organized as follows:\medskip

Understanding homotopy fiber sequences requires a good understanding of homotopy fiber squares, homotopy pullbacks, derived functors, and category localizations. However, there are a number of variants and indeterminacies for each of these concepts. A structured approach in a unifying context is suggested in \cite{CompTheo} and \cite{Models}. In order to ensure an independent readability of our text, we recall in Section \ref{ModHoPull} the relevant results from \cite{CompTheo} and \cite{Models}, which we will need later. \medskip

In Section \ref{LoHoFibSeq} we give the precise definition of the category $\tt h(M)$ (resp., the category $\ell({\tt M})$) of homotopy fiber sequences (resp., of long homotopy fiber sequences) in a pointed model category $(\tt M,0)\,.$ We define the homotopy category of these new categories and choose a model structure on the category $\tt M^\s$ of $\tt M$--morphisms. We prove that the localization $$\op{Ho}(R_1):\tt Ho(\ell(M))\to Ho(M^\s)$$ of the restriction $R_1$ of long homotopy fiber sequences to their first two terms yields an equivalence of categories. This is arguably the deepest result of the present paper. To prove the equivalence theorem we construct the inverse up to natural isomorphisms and give an explicit description of the inverse of $$\op{Ho}(R_1)_{a_\bullet,b_\bullet}:\op{Hom}_{\tt Ho(\ell(M))}(a_\bullet,b_\bullet)\to\op{Hom}_{\tt Ho(M^\s)}(R_1(a_\bullet),R_1(b_\bullet))\;$$ (see pages 14--26).\medskip

Mimicking definitions from Algebraic Topology, we define in Section \ref{SectionPuppe} the notion of `based path space functor' -- which is dual to the `cone functor' -- and the associated `loop space functor' $\zW$ -- which is well-behaved with respect to fibrant objects, weak equivalences and homotopy fiber sequences, and whose derived functor is independent of the chosen based path space functor. We continue our abstraction of the topological situation and define concepts of homotopy fiber and universal connecting homomorphism that allow us to make sense of Puppe's sequence in the general setting considered. This leads to Puppe's functor $\cP:\tt M^\s\to\ell(M)$ whose value $\cP_f$ at $f\in{\tt M^\s}$ is the unique $\ell({\tt M})$--extension of $f$ up to a canonical $\tt Ho(\ell(M))$--isomorphism. The theory is valid in every pointed model category (in this general case it is natural to consider morphisms $f$ with fibrant source and target) and in right proper pointed model categories (in this case we do not need the property that the source and target are fibrant). Just as Puppe's sequence of a fibration gives the fibration's long exact homotopy sequence in Algebraic Topology, we associate in any pointed model category $\tt M$ a family of long exact sequences of $\,\tt Ho(M)$--Hom-sets to Puppe's $\ell(\tt M)$--extension of a morphism.\bigskip

Section \ref{ComparisonQuillen} contains a detailed account of the parallelism described above between Quillen's theory of fibration sequences and the theory of homotopy fiber sequences in the present work.\bigskip

One of the advantages of the new theory is that it is easy to use. In Section \ref{Applications} we apply it to the category of chain complexes of modules (both unbounded and non-negatively graded)  and recover the long exact sequence in homology induced by a short exact sequence of chain complexes and chain maps as a special case of the long exact sequence of $\tt Ho(M)$--Hom-sets of these chain maps. To this end, we consider the category of complexes first only as a pointed model category, then as a pointed monoidal model category. Although the two approaches naturally lead to different based path space functors, the Puppe extensions of a chain morphism coincide.\bigskip

In the final section \ref{Conclusions}, we briefly describe follow up questions and expected applications. In addition, the theory we develop in this text should lead to advances in homotopical algebraic geometry \cite{TV05, TV08, Schwarz, Linear, HAC} and higher supergeometry \cite{Berezinian, Z2nManifolds, LocalForms,Integration}, which are the contexts from which the need arose to examine the subjects of this paper.\bigskip\bigskip

{\bf Conventions and notations}. We assume that the reader is familiar with model categories and adopt the definition of a model category that is used in \cite{Hir}. More precisely, a model category is a category $\tt M$ that is equipped with three classes of morphisms called weak equivalences, fibrations and cofibrations. The category $\tt M$ has all small limits and colimits and the 2-out-of-3 axiom, the retract axiom and the lifting axiom are satisfied. Moreover $\tt M$ admits a functorial cofibration - trivial fibration factorization system and a functorial trivial cofibration - fibration factorization system. Further, we work with the Quillen homotopy category $\tt Ho(M)\,$, which is a strict localization of $\tt M$ at its weak equivalences $W$ with localization functor denoted $\zg_{\tt M}\,,$ and we use the Kan extension derived functor operations $\mathbb{L}^{\op{K}},\mathbb{R}^{\op{K}}$ and the strongly universal derived functor operations $\mathbb{L}^{\op{S}},\mathbb{R}^{\op{S}}$ in the sense of \cite{CompTheo}. We will consider different types of replacement, in particular local fibrant C-replacements, which means that for every $X\in\tt M$ we choose a fibrant replacement $\tilde{F}X\twoheadrightarrow\ast$ of $X$ such that the map $f_X:X\stackrel{\sim}{\to}\tilde{F}X$ is a cofibration and is identity if $X$ is already fibrant. If $f:X\to Y\,,$ there is a lifting $\tilde{F}f:\tilde{F}X\to \tilde{F}Y,$ which will play an important role:

\begin{equation}\label{FTilde} \begin{tikzpicture}
 \matrix (m) [matrix of math nodes, row sep=3em, column sep=3em]
   {  \stackrel{}{X}  & \stackrel{}{Y} & \tilde{F}Y  \\
      \tilde{F}X & & \stackrel{}{\ast}  \\ };
 \path[->]
 (m-1-1) edge node[auto] {\small{$f$}} (m-1-2)
 (m-1-2) edge [>->] node[auto] {\small{$\;\;{}_{\widetilde{}}\,\;f_Y$}}(m-1-3)
 (m-2-1) edge [->>] (m-2-3)
 (m-1-1) edge [>->] node[auto] {\small{$\;\;{}_{\widetilde{}}\,\;f_X$}} (m-2-1)
 (m-1-3) edge [->>] (m-2-3)
 (m-2-1) edge [->, dashed] node[auto] {\small{$\tilde{F}f$}} (m-1-3);
\end{tikzpicture}
\end{equation}

\section{Models of homotopy pullbacks}\label{ModHoPull}

We start recalling some results of \cite{CompTheo} and \cite{Models}.\medskip

The first theorem addresses the question of stability of a derived functor with respect to a change of definition (Kan extension versus strong universal property) and with respect to a change of the type of fibrant replacement used to compute it (local versus global).

\begin{theo}[\cite{CompTheo}]\label{Fundamental0}
If $G\in\tt Fun(M,N)$ is a functor between model categories that sends weak equivalences between fibrant objects to weak equivalences, its Kan extension right derived functor $$\mathbb{R}^{\op{K}}G\in\tt Fun(Ho(M),Ho(N))$$ and its strongly universal right derived functor $$\mathbb{R}^{\op{S}}G\in\tt Fun(Ho(M),Ho(N))$$ exist and we have \be\label{Fundamental1}\mathbb{R}^{\op{K}}G\doteq\op{Ho}(\zg_{\tt N}\circ G\circ\tilde{F})\doteq\mathbb{R}_R^{\op{S}}G:=\op{Ho}(\zg_{\tt N}\circ G\circ R)\stackrel{\cong}{\Rightarrow} \mathbb{R}^{\op{S}}G\;,\ee where $\tilde{F}$ is a local fibrant C-replacement, $R$ is a fibrant C-replacement functor and $\op{Ho}$ the unique on the nose factorization through $\tt Ho(M)\,$. This implies that \be\label{Fundamental2}\mathbb{R}^{\op{K}}G\circ\zg_{\tt M}\doteq\zg_{\tt N}\circ G\circ\tilde{F}\doteq\mathbb{R}_R^{\op{S}}G\circ\zg_{\tt M}=\zg_{\tt N}\circ G\circ R\stackrel{\cong}{\Rightarrow}\mathbb{R}^{\op{S}}G\circ\zg_{\tt M}\;,\ee where $\doteq$ denotes a canonical natural isomorphism and $\stackrel{\cong}{\Rightarrow}$ a not necessarily canonical natural isomorphism.
\end{theo}

The next corollary emphasizes that a derived or homotopy limit with respect to a suitable model structure $\zs$ on the diagram category under consideration belongs to a well-defined isomorphism class of the target homotopy category, regardless of the definition of a derived functor and the fibrant replacement we use. Considered as an object of the target {\it model} category, a homotopy limit is thus only well-defined up to a zigzag of weak equivalences. This indeterminacy is further increased by the ambiguity resulting from various choices for $\zs\,$.

\begin{cor}\label{IndeterminacyHoLimTheo}
Let $\,\tt S$ be a small category, let $\tt M$ be a model category and let $\zs$ be a model structure on the category $\tt Fun(S,M)$ of $\,\tt S$--shaped diagrams of $\,\tt M$ such that $\op{Lim}:\tt Fun(S,M)\to M$ is a right Quillen functor. If $X\in\tt Fun(S,M)$ its homotopy limit with respect to $\zs$ is given as an object of $\tt M$ by \be\label{IndeterminacyHoLim}\mathbb{R}_{\zs}\!\op{Lim}(X)\zig\op{Lim}(R_{\zs}X)\stackrel{\sim}{\rightleftarrows}\op{Lim}(\tilde{F}_{\zs}X)
\stackrel{\sim}{\to}\op{Lim}(F_{\zs}X)\;,\ee where $\mathbb{R}_{\zs}\!\op{Lim}(X)$ can be interpreted as Kan extension or strongly universal derived functor, where $\approx$ denotes a zigzag of weak equivalences and where $R_{\zs},\tilde{F}_{\zs},F_{\zs}$ are a fibrant C-replacement functor, a local fibrant C-replacement and any fibrant replacement in the model structure $\zs\,,$ respectively. The weak equivalence $\stackrel{\sim}{\to}$ between the last two representatives is the universal morphism \be\label{UniMorHoLim}\op{Lim}(\ell_\zs):\op{Lim}(\tilde{F}_\zs X)\stackrel{\sim}{\rightarrow}\op{Lim}(F_\zs X)\ee that is induced by a lifting
\begin{equation}\label{FibRep}
\begin{tikzpicture}
 \matrix (m) [matrix of math nodes, row sep=3em, column sep=3em]
   {X & & F_\zs X   \\
    \tilde{F}_\zs X & & \stackrel{}{\ast} \\};
 \path[->]
 (m-1-1) edge [->] node[above]{\small{$\sim$}} node[below]{\small{$f_X$}} (m-1-3)
 (m-1-1) edge [>->] node[left]{\small{$\sim$}} node[right]{\small{$\tilde{f}_X$}} (m-2-1)
 (m-2-1) edge [->>] (m-2-3)
 (m-1-3) edge [->>](m-2-3)
 (m-2-1) edge [->, dashed] node[below] {\small{$\ell_\zs$}} (m-1-3);
\end{tikzpicture}
\end{equation}
and its image $\zg_{\tt M}(\op{Lim}(\ell_\zs))$ in homotopy is independent of the lifting considered. A similar remark holds for the weak equivalences $\stackrel{\sim}{\rightleftarrows}\,$.
\end{cor}

In the case ${\tt S}={\tt I}:=\{c\to d\leftarrow b\}$ the functors $X \in \tt Fun(I, M)$ are the cospan diagrams $C \to D \leftarrow B$ of $\,\tt M$ and the limit $\op{Lim}(X)$ is the pullback $ B \times_D C \,$. There are three Reedy model structures $ \zs_i $ ($ i \in \{1,2,3 \} $) on $ \tt Fun(I, M) \,$, for which the pullback is a right Quillen functor. The homotopy limits $ \R_{\zs_i} \! \op{Lim}(X) $ with respect to the $ \zs_i $ are called homotopy pullbacks and are denoted by $ B \times_D^{h_{\zs_i}} C\,$. It can be proven that common representatives exist. We define the full homotopy pullback $ B \times_D^h C $ such that its canonical representatives are exactly the representatives of all three homotopy pullbacks $ B \times_D^{h_{\zs_i}} C $ ($ i \in \{1, 2.3 \} $). This leads to the

\begin{theo}[\cite{Models}]\label{IndHoPull}
The full homotopy pullback of a cospan $C\to D\leftarrow B$ in a model category is independent of the type of derived functor and of the model structure $\zs_i$ $(i\in\{1,2,3\})$ on cospan diagrams considered. We get the canonical representatives of the full homotopy pullback from the standard pullback of the weakly equivalent cospans $C'\to D'\leftarrow B'$ with three fibrant objects and at least one morphism that is a fibration: if in the adjacent commutative squares
\begin{equation}
\begin{tikzpicture}
 \matrix (m) [matrix of math nodes, row sep=3em, column sep=2.5em]
   {C & & D & & B \\
    C' & & D' & & B' \\};
 \path[->]
 (m-1-1) edge [->] (m-1-3)
 (m-1-3) edge [<-] (m-1-5)
 (m-2-1) edge [->] (m-2-3)
 (m-2-3) edge [<-] (m-2-5)
 (m-1-1) edge [->] node[left] {\small{$\sim$}} (m-2-1)
 (m-1-3) edge [->] node[left] {\small{$\sim$}}(m-2-3)
 (m-1-5) edge [->] node[right] {\small{$\sim$}}(m-2-5);
\end{tikzpicture}
\end{equation}
all vertical arrows are weak equivalences, all bottom nods are fibrant objects and at least one of the bottom arrows is a fibration, we have \be\label{HoPull}B\times_D^{h} C\approx B'\times_{D'}C'\;.\ee
\end{theo}

We increase the flexibility of homotopy limits, homotopy pullbacks and full homotopy pullbacks by allowing generalized representatives. In the case of full homotopy pullbacks, we have the

\begin{theo}\label{IndHoPullMod}
The vertex $A$ of the span of a commutative square
\begin{equation}\label{ComSqu}
\begin{tikzpicture}
 \matrix (m) [matrix of math nodes, row sep=2.5em, column sep=1.5em]
   {A && B\\
    C && D\\};
 \path[->]
 (m-1-1) edge [->] (m-1-3)
 (m-2-1) edge [->] (m-2-3)
 (m-1-1) edge [->] (m-2-1)
 (m-1-3) edge [->] (m-2-3);
\end{tikzpicture}
\end{equation} in a model category is a model or \emph{generalized representative} of the full homotopy pullback $B\times_D^{h}C$ if the universal morphism from $A$ to a \emph{canonical representative} of $B\times_D^{h}C$ is a weak equivalence. In other words, there must exist a cospan $C'\to D'\leftarrow B'$ to which $C\to D\leftarrow B$ is weakly equivalent, whose three nodes are fibrant objects and at least one of whose morphisms is a fibration, such that the universal morphism $A\to B'\times_{D'}C'$ is a weak equivalence.
\end{theo}

If the condition of Theorem \ref{IndHoPullMod} is satisfied for one replacement its is satisfied for all replacements.\medskip

In right proper model categories, we can weaken the condition:

\begin{theo}[\cite{Models}]\label{RightProper}
The vertex $A$ of the span of a commutative square \eqref{ComSqu} in a \emph{right proper} model category is a model of the full homotopy pullback $B\times_D^{h}C$ if there exists a cospan $C'\to D'\leftarrow B'$ to which $C\to D\leftarrow B$ is weakly equivalent and at least one of whose morphisms is a fibration, such that the universal morphism $A\to B'\times_{D'}C'$ is a weak equivalence.
\end{theo}

Again, if the condition of Theorem \ref{RightProper} is satisfied for one replacement with one fibration it is satisfied for all replacements of this type.\medskip

The following corollary is stated without proof in \cite{JL}:

\begin{cor}[\cite{Models}]\label{CorLur}
In a model category the standard pullback $B\times_DC$ of a cospan $C\stackrel{g}{\to} D\stackrel{f}{\leftarrow} B$ is a model of the cospan's full homotopy pullback if at least one of the morphisms $f$ or $g$ is a fibration and either all three objects $B,C,D$ are fibrant or the model category is right proper.
\end{cor}

Further, the concept of model of a homotopy pullback captures the notion of homotopy fiber square defined in \cite{Hir} and puts it in the right context.

\begin{cor}[\cite{Models}]\label{HoFibSqCor}
In a right proper model category with a fixed functorial trivial cofibration - fibration factorization system, a commutative square \eqref{ComSqu} is a model square, i.e., its vertex $A$ is a model of the homotopy pullback $B\times_D^hC$ if and only if it is a homotopy fiber square in the sense of $\cite{Hir}$.
\end{cor}

There is a pasting law for model squares.

\begin{prop}[\cite{Models}]\label{PasLaw}
Let
\begin{equation}\label{ComDia}
\begin{tikzpicture}
 \matrix (m) [matrix of math nodes, row sep=3em, column sep=2.5em]
   {A & & B & & C \\
    D & & E & & F \\};
 \path[->]
 (m-1-1) edge [->] (m-1-3)
 (m-1-3) edge [->] (m-1-5)
 (m-2-1) edge [->] (m-2-3)
 (m-2-3) edge [->] (m-2-5)
 (m-1-1) edge [->] (m-2-1)
 (m-1-3) edge [->] (m-2-3)
 (m-1-5) edge [->] (m-2-5);
\end{tikzpicture}
\end{equation}
be a commutative diagram in a model category. If the right square is a model square, then the left square is a model square if and only if the total square is a model square.
\end{prop}

The next result generalizes a property of homotopy fiber squares in a right proper model category with a fixed functorial trivial cofibration - fibration factorization system to model squares in an arbitrary model category.

\begin{prop}[\cite{Models}]\label{ComPara}
Let $ABCD$ and $A'B'C'D'$ be two commutative squares in a model category $\tt M\,$. If there exist four $\tt M$--morphisms from the vertices of the first square to the corresponding vertex of the second such that the four resulting squares commute and if these $\tt M$--morphisms are weak equivalences, then the first square is a model square if and only if the second is.
\end{prop}

\begin{equation}\label{ComCub}
\begin{tikzcd}[back line/.style={densely dotted}, row sep=1em, column sep=1em]
A \ar[dr, "\sim"] \ar{rr} \ar{dd} &  & B \ar{dd} \ar[dr, "\sim"] & \\
& A' \ar{rr} \ar{dd} & & B' \\
C \ar{rr} \ar[dr, "\sim"] & & D \ar[dr, "\sim"]  & \\
& C' \ar{rr} & & D'\ar[crossing over, leftarrow]{uu} \\
\end{tikzcd}
\end{equation}

Later we need

\begin{prop}\label{weq} Any commutative square
$(A,B,C,D)$ in a model category whose vertical (or horizontal) arrows are weak equivalences is a model square.
\end{prop}

\begin{proof} We apply a fibrant C-replacement functor $R$ to the commutative square $(A,B,C,D)$ and factor the morphism $$R(B\stackrel{\zk}{\to} D)=RB\stackrel{R\zk}{\to}RD=RB\stackrel{\sim}{\to}F(R\zk)\twoheadrightarrow RD$$ into a weak equivalence followed by a fibration. Moreover, we set $P:=F(R\zk)\times_{RD}RC$ and thus get the following commutative diagram:
\begin{equation}\label{AdjComCub}
\begin{tikzcd}[back line/.style={densely dotted}, row sep=1em, column sep=1em]
A \ar[dr, "\sim"] \ar{rrr} \ar[ddd, "\sim"] \ar[ddrr, bend right, dashed]
  & & & B \ar[ddd, "\sim"] \ar[dr, "\sim"] & & \\
& RA \ar{rrr} \ar[ddd, "\sim"] \ar[dr, dashed]
  & & & RB \ar[ddd, "\sim"]\ar[dr, "\sim"] & \\
& & P \ar{rrr} \ar{ddd} & & & F(R\zk)\ar[ddd, twoheadrightarrow, "\sim"] \\
C \ar{rrr} \ar[dr, "\sim"] & & & D \ar[dr, "\sim"] & & \\
& RC \ar{rrr} \ar[dr, equal] & & & RD \ar[dr, equal] \ar[crossing over, leftarrow]{uuu} &\\
& & RC \ar{rrr} & & & RD
\end{tikzcd}
\end{equation}
Since trivial fibrations are closed under pullbacks in any model category, the arrow $P\to RC$ is a trivial fibration, hence a weak equivalence. It follows that $A\to P$ is a weak equivalence, so that $ABCD$ is a model square.
\end{proof}

\begin{prop}\label{QuiPresMod}
Let $G:\tt M\to N$ be a right Quillen functor and let $ABCD$ be a commutative square of $\,\tt M\,$ with \emph{fibrant vertices}. Then, if $ABCD$ is a model square of $\,\tt M$ its image $G(ABCD)$ is a model square of $\,\tt N\,$.
\end{prop}

\begin{proof}
We factor $B\stackrel{\zk}{\to} D=B\stackrel{\sim}{\to}F(\zk)\twoheadrightarrow D$ into a weak equivalence followed by a fibration. Further, we set $P:=F(\zk)\times_D C$ and get the commutative cube
\begin{equation}\label{ComCubF}
\begin{tikzcd}[back line/.style={densely dotted}, row sep=1em, column sep=1em]
A \ar[rr]\ar[dd]\ar[dr, dashed] &&  B \ar[dd]\ar[dr, "\sim"] &\\
& P \ar[rr]\ar[dd] && F(\zk) \ar[dd, twoheadrightarrow] \\
C \ar[rr]\ar[dr, equal] && D \ar[dr, equal] & \\
& C \ar[rr] && D
\end{tikzcd}
\end{equation}
Since $ABCD$ is a model square, the universal arrow $A\dashrightarrow P$ is a weak equivalence. As fibrations are closed under pullbacks, the arrow $P\to C$ is a fibration and all vertices of \eqref{ComCubF} are fibrant objects.\medskip

The image of \eqref{ComCubF} by $G$ is the commutative cube
\begin{equation}\label{ComCubFG}
\begin{tikzcd}[back line/.style={densely dotted}, row sep=1em, column sep=1em]
GA \ar[rr]\ar[dd]\ar[dr] &&  GB \ar[dd]\ar[dr] &\\
& GP \ar[rr]\ar[dd] && G(F(\zk)) \ar[dd] \\
GC \ar[rr]\ar[dr, equal] && GD \ar[dr, equal] & \\
& GC \ar[rr] && GD
\end{tikzcd}
\end{equation}
Since right adjoint functors preserve limits, we have $GP=G(F(\zk))\times_{GD}GC$ and the arrow $GA\to GP$ is the universal arrow from $GA$ to the pullback of the front cospan. As $G$ is a right Quillen functor it preserves weak equivalences between fibrant objects (in view of Brown's lemma), so that the universal arrow $GA\to GP$ and the arrow $GB\to G(F(\zk))$ are weak equivalences ($\star_1$). As $G$ preserves fibrations (by definition of a right Quillen functor), the arrow $G(F(\zk))\to GD$ is a fibration ($\star_2$). Furthermore, the terminal object $\ast_{\tt M}\,$ of $\tt M$ is the limit $\op{Lim}\varnothing_{\tt M}$ of the unique functor $\varnothing_{\tt M}\in\tt Fun(\emptyset,M)$ from the empty category $\emptyset$ to $\tt M\,$. Hence, if $\textsf{F}\in\tt M$ is fibrant, we get \be\label{PresTerm}G(\textsf{F}\twoheadrightarrow\ast_{\tt M})=G(\textsf{F}\twoheadrightarrow\op{Lim}\varnothing_{\tt M})=G\textsf{F}\twoheadrightarrow\op{Lim}G(\varnothing_{\tt M})=G\textsf{F}\twoheadrightarrow\op{Lim}\varnothing_{\tt N}=G\textsf{F}\twoheadrightarrow\ast_{\tt N}\,,\ee so that $G$ preserves fibrant objects and all vertices of \eqref{ComCubFG} are fibrant ($\star_3$). From ($\star_1$), ($\star_2$) and ($\star_3$) it follows that the back square $G(ABCD)$ of \eqref{ComCubFG} is a model square of $\tt N\,.$
\end{proof}

\section{Long homotopy fiber sequences}\label{LoHoFibSeq}

\subsection{Definitions}

Let $\tt M$ be a {\bf pointed model category}, i.e., a model category with a zero object $0$ (a model category whose initial and terminal objects coincide).\medskip

It is natural to refer to the pullback $A:=B\times_D 0$ of an $\tt M$--morphism $B\to D$ over the point $0\to D$ as the fiber of $B\to D$ and to call $A\to B\to D$ a fiber sequence. The following generalization is crucial:

\begin{defi}
A {\bf homotopy fiber sequence} $A\to B\to D$ in a pointed model category $\tt M$ is a model square $ABCD$ in $\tt M$ whose left lower vertex $C$ is acyclic.
\end{defi}

Since an object $C\in\tt M$ is acyclic if the unique morphism $0\to C$ is a weak equivalence (or equivalently if the unique morphism $C\to 0$ is a weak equivalence), a homotopy fiber sequence is a commutative square whose left upper vertex is a model of the homotopy pullback of the square's cospan and whose left lower vertex is weakly equivalent to zero. We stress that, although the left lower vertex $C$ is implicit in the notation $A\to B\to D$ of the homotopy fiber sequence, it is an integral part of it. \medskip

Morphisms of homotopy fiber sequences are therefore defined as morphisms of commutative squares, i.e., as commutative cubes. Explicitly a morphism of homotopy fiber sequences from $A_2\to A_1\to A_0$ with implicit vertex $C_A$ to $B_2\to B_1\to B_0$ with implicit vertex $C_B$ is a quadruplet $$\Phi=(\phi_0,\phi_1,\phi_2,\zvf)$$ of $\tt M$--morphisms $$\phi_i:A_i\to B_i\;(i\in\{0,1,2\})\quad\text{and}\quad \zvf:C_A\to C_B\;,$$ such that the resulting cube commutes. So composition of morphisms of homotopy fiber sequences is induced by the composition of $\tt M\,.$ We denote $\tt h(M)$ the category of homotopy fiber sequences of $\tt M\,$. Further, we denote $$a=(a_1,a_2)$$ a homotopy fiber sequence $$A_2\stackrel{a_2}{\longrightarrow} A_1\stackrel{a_1}{\longrightarrow} A_0\;.$$

\begin{defi}
In a pointed model category a {\bf long homotopy fiber sequence} \be\label{LHFS}a_\bullet=(a_1,a_2,a_3, \cdots)\ee or more explicitly $$(A_\bullet,a_\bullet): \cdots\longrightarrow A_3\stackrel{a_3}{\longrightarrow}A_2\stackrel{a_2}{\longrightarrow} A_1\stackrel{a_1}{\longrightarrow} A_0\;$$ is a sequence of homotopy fiber sequences \be\label{SofSObj}A_{n+1}\stackrel{a_{n+1}}{\longrightarrow}A_{n}\stackrel{a_n}{\longrightarrow}A_{n-1}\quad(n\in\{1,2,3\cdots\})\;.\ee
\end{defi}

It is natural to define a morphism of long homotopy fiber sequences as a sequence of morphisms of homotopy fiber sequences. Explicitly a morphism $\Phi_\bullet$ of long homotopy fiber sequences from $(A_\bullet,a_\bullet)$ to $(B_\bullet,b_\bullet)$ is a sequence \be\label{MLHFS}\Phi_\bullet=(\zf_0,\zf_1,\zf_2,\zf_3,\cdots,\zvf_1,\zvf_2,\cdots)\ee of morphisms \be\label{SofSMor}\Phi_n=(\zf_{n-1},\zf_n,\zf_{n+1},\zvf_n)\quad(n\in\{1,2,3\cdots\})\ee of homotopy fiber sequences from $$A_{n+1}\stackrel{a_{n+1}}{\longrightarrow}A_{n}\stackrel{a_n}{\longrightarrow}A_{n-1}\quad\text{to}\quad B_{n+1}\stackrel{b_{n+1}}{\longrightarrow}B_{n}\stackrel{b_n}{\longrightarrow}B_{n-1}\;.$$ Composition of morphisms of long homotopy fiber sequences is again induced by the composition in $\tt M\,.$ We denote $\tt \ell(M)$ the category of long homotopy fiber sequences of $\tt M\,.$ Moreover, we say that a homotopy fiber sequence (resp., a long homotopy fiber sequence) of $\tt M$ is {\bf objectwise fibrant} if its four vertices are fibrant objects of $\tt M$ (resp., if all homotopy fiber sequences \eqref{SofSObj} are objectwise fibrant). We also say that a morphism of homotopy fiber sequences (resp., of long homotopy fiber sequences) of $\tt M$ is an {\bf objectwise weak equivalence} if its four component morphisms are weak equivalences of $\tt M$ (resp., if all morphisms of homotopy fiber sequences \eqref{SofSMor} are objectwise weak equivalences).\medskip

We close this subsection with the following corollary of Proposition \ref{QuiPresMod}:

\begin{cor}
Let $G:\tt M\to N$ be a right Quillen functor between pointed model categories. The image under $G$ of an objectwise fibrant (long) homotopy fiber sequence of $\,\tt M$ is a (long) homotopy fiber sequence of $\,\tt N\,.$
\end{cor}

\begin{proof}
If suffices to remember that a right Quillen functor $G$ preserves weak equivalences between fibrant objects and preserves the (fibrant) zero object (see \eqref{PresTerm}).
\end{proof}

\begin{rem}
\emph{We will show below that in a pointed model category it is often possible to extend a morphism to a long homotopy fiber sequence. In the next subsection we address the question of uniqueness of such an extension if it exists.}
\end{rem}

\subsection{Homotopy theory of long homotopy fiber sequences}

Let $\tt C$ be a category and $W$ a family of $\tt C$--morphisms. The {\bf Gabriel-Zisman localization} or {\bf zigzag localization} of $\tt C$ at $W$ is (at least in a higher universe) a pair $({\tt C}[[W^{-1}]], \zg)$ that consists of a category ${\tt C}[[W^{-1}]]$ and a functor $\zg: {\tt C} \to {\tt C}[[W^{-1}]]$ which sends all morphisms in $W$ to isomorphisms. Further, every functor out of $\tt C$ with this property factors uniquely and on the nose through ${\tt C}[[W^{-1}]]\,.$ Because of its universal property, the zigzag localization is unique up to a unique isomorphism. It is the strong localization of $\tt C$ at $W$ \cite{CompTheo} and is constructed by free inversion of the morphisms of $W$. More precisely, the description of $({\tt C}[[W^{-1}]],\zg)$ is exactly the one given in \cite{CompTheo} in the case of the homotopy category of a model category. As already mentioned, the localized category does not have to be locally small, so we implicitly move to a higher universe in order to get a genuine category. If the definition of $W$ is clear, we usually refer to ${\tt C}[[W^{-1}]]$ as the homotopy category of $\tt C$ and denote it with $\tt Ho(C)\,.$\medskip

In the case of the category $\ell({\tt M})$ of long homotopy fiber sequences of a pointed model category $\tt M\,$, we choose the objectwise weak equivalences for $W$ and consider the homotopy category $\tt Ho(\ell(M))$ in the previous sense.\medskip

Since we are interested in possible extensions of an $\tt M$--morphism to a long homotopy fiber sequence of $\tt M\,$, we need not only the categories $\ell({\tt M})$ and $\tt Ho(\ell(M)),$ but also the category ${\tt M}^{\scriptscriptstyle\to}\!$ of $\tt M$--morphisms \be A_1\stackrel{a_1}{\longrightarrow} A_0\;\,\ee and commutative squares
\begin{equation}\label{CS}
\begin{tikzpicture}
 \matrix (m) [matrix of math nodes, row sep=2.5em, column sep=1.5em]
   {A_1 && A_0\\
    B_1 && B_0\\};
 \path[->]
 (m-1-1) edge [->] node[auto] {\small{$a_1$}} (m-1-3)
 (m-2-1) edge [->] node[auto] {\small{$b_1$}} (m-2-3)
 (m-1-1) edge [->] node[auto] {\small{$\psi_1$}} (m-2-1)
 (m-1-3) edge [->] node[auto] {\small{$\psi_0$}} (m-2-3);
\end{tikzpicture}
\end{equation}
and its homotopy category $\tt Ho(M^\s\!)\,$. To give meaning to the latter, we endow the category $\tt M^\s\!$ with a model structure, so that its zigzag localization at its weak equivalences, which is its Quillen homotopy category, is a genuine category without us passing into a larger universe (see for example \cite[Theorem 2]{CompTheo}). To find a model structure, notice that $\tt M^\s$ is the functor category $\tt Fun(I,M)\,,$ where $\tt I$ is the inverse category ${\tt I}=\{1\to 0\}\,.$ In the case of such simple Reedy categories, the corresponding Reedy model structure is the injective model structure with objectwise weak equivalences and cofibrations (details can be found for instance in \cite{Models}). We equip $\tt M^\s$ with this model structure.\medskip

To study the mentioned extension problem, we introduce the {\bf restriction functor} $$R_1:\tt \ell(M)\to M^\s\;,$$ which we define on objects as $R_1a_\bullet=a_1$ (see \eqref{LHFS}) and on morphisms as $R_1\Phi_\bullet=(\zf_1,\zf_0)$ (see \eqref{MLHFS}). The functor $\zg_{\tt M^\s\!}\circ R_1$ sends objectwise weak equivalences of $\ell(M)$ to isomorphisms of $\tt Ho(M^\s)$ and therefore factors uniquely through $\tt Ho(\ell(M))\,:$ there is a unique functor $$\op{Ho}(R_1):{\tt Ho(\ell(M))\to Ho(M^\s)}\,,\;\text{such that}\; \zg_{\tt M^\s\!}\circ R_1=\op{Ho}(R_1)\circ\zg_{\ell({\tt M})}\;.$$

\begin{theo}\label{EquivHoCat}
Let $\tt M$ be a pointed model category. The localization $$\op{Ho}(R_1):{\tt Ho(\ell(M))\to Ho(M^\s)}$$ of the restriction functor of long homotopy fiber sequences yields an equivalence of categories between the homotopy category of the category of long homotopy fiber sequences of $\tt M$ and the homotopy category of the category of morphisms of $\tt M\,.$
\end{theo}

\begin{rem}\label{Lemff}
\emph{The categorical equivalence means that the localized functor $\op{Ho}(R_1)$ is essentially surjective and fully faithful. In other words, every $\tt M$--morphism is up to an isomorphism the restriction of a long homotopy fiber sequence and, for every $a_\bullet, b_\bullet\in \ell({\tt M})\,,$ the map \be\label{FullFaith}\op{Ho}(R_1)_{a_\bullet b_\bullet}:\op{Hom}_{\tt Ho(\ell(M))}(a_\bullet,b_\bullet)\to \op{Hom}_{\tt Ho(M^\s\!)}(a_1,b_1)\ee is a 1:1 correspondence. Alternatively, the equivalence means that $\op{Ho}(R_1)$ has an inverse up to natural isomorphisms. To prove the theorem we construct this inverse.}
\end{rem}

\begin{lem}\label{FunE}
Let $\tt M$ be a pointed model category. There exists a functor $E:\tt M^\s\to \ell(M)$ that preserves weak equivalences and sends morphisms $a\in\tt M^\s$ to long homotopy fiber sequences $f^a_\bullet$ made of fibrations $f^a_n:F^a_n\to F^a_{n-1}$ between fibrant objects $(n\in\{1,2,3\cdots\})\,.$
\end{lem}

\begin{rem}\label{Fixed}\emph{In the following we use a fixed functorial trivial cofibration - fibration factorization system $(\za,\zb)$ and the induced fibrant C-replacement functor $R\,.$}\end{rem}

\begin{proof}
Let $a_1,b_1$ be objects of $\tt M^\s$ and $\psi=(\psi_0, \psi_1)$ an $\tt M^\s$--morphism between them, see \eqref{CS}. We will construct $E$ simultaneously and inductively on these objects and this morphism, i.e., we construct inductively long homotopy fiber sequences $E(a_1)=:f^{a_1}_\bullet$ and $E(b_1)=:f^{b_1}_\bullet$ and a morphism $E(\psi)=:\Phi_\bullet$ of long homotopy fiber sequences between them. We start from the commutative diagram $\psi$ and use the chosen replacement functor and factorization system to get the commutative diagram
\be\label{Start}
\begin{tikzcd}
RA_1\arrow[r, "\za(Ra_1)"]\arrow[r,swap,"\sim"]\arrow[d, swap, "R\psi_1"]&F(Ra_1)\arrow[r,two heads,"\zb(Ra_1)"]\arrow[d]&RA_0\arrow[d,"R\psi_0"]\\
 RB_1\arrow[r, "\za(Rb_1)"]\arrow[r,swap,"\sim"]&F(Rb_1)\arrow[r,two heads,"\zb(Rb_1)"]&RB_0
\end{tikzcd}
\ee
We denote the upper and lower fibrations between fibrant objects by $f^{a_1}_1:F^{a_1}_1\to F^{a_1}_0$ and $f^{b_1}_1:F^{b_1}_1\to F^{b_1}_0,$ respectively. It is clear that if $\psi$ is a weak equivalence, all the vertical arrows in \eqref{Start} are weak equivalences, in particular the central arrow $\phi_1:F^{a_1}_1\to F^{b_1}_1$ and the right arrow $\phi_0:F^{a_1}_0\to F^{b_1}_0\,$. Assume now that the long homotopy fiber sequences $f^{a_1}_\bullet$, $f^{b_1}_\bullet$ (resp., the morphism $\Phi_\bullet$ between them) have (resp., has) been constructed together with their implicit vertices (resp., its implicit arrows) and with all the required properties, up to order $n\ge 1\,$. If we apply the functorial factorization $(\za,\zb)$ to $0\to F^{a_1}_{n-1}$ and $0\to F^{b_1}_{n-1}$, we get the commutative diagram
\be
\begin{tikzcd}[row sep=1em, column sep=1em]
&&F^{a_1}_{n+1}\arrow[rr,"f^{a_1}_{n+1}"] \arrow[dr,dashed,"\phi_{n+1}"] \arrow[dd]\arrow[dr, phantom, very near start, color=black] && F^{a_1}_{n} \arrow[dd, near start, "f^{a_1}_n", two heads] \arrow[dr,"\phi_n"] \\
&&& F^{b_1}_{n+1}\arrow[rr,crossing over, near start, "f^{b_1}_{n+1}"]\arrow[dr, phantom , very near start, color=black] &&  F^{b_1}_{n}\arrow[dd,crossing over, "f^{b_1}_n", two heads] \\
0\arrow[rr,tail,"\sim"]\arrow[dr, swap, "\sim"]\;&&C_{F^{a_1}_{n}} \arrow[rr,two heads] \arrow[dr,"\zvf_n"] && F^{a_1}_{n-1} \arrow[dr,"\phi_{n-1}"] \\
&0\arrow[rr,tail,"\sim"]&& C_{F^{b_1}_{n}} \arrow[rr,two heads] \arrow[uu,<-]&& F^{b_1}_{n-1}
\end{tikzcd}
\ee
in which $F^{a_1}_{n+1}$ and $F^{b_1}_{n+1}$ are pullbacks and $\phi_{n+1}$ is the universal morphism. Both pullback diagrams are canonically homotopy fiber sequences and, as already mentioned, the cube is commutative. So we built two long homotopy fiber sequences $E(a_1)=f^{a_1}_\bullet$ and $E(b_1)=f^{b_1}_\bullet$ and a morphism $E(\psi)=\Phi_\bullet$ between them. Since $(\za,\zb)$ is functorial (and the induced $R$ is a functor), the assignment $\psi\mapsto E(\psi)=\Phi_\bullet$ respects compositions and identities: $E$ is a functor $E:{\tt M^\s}\to\ell({\tt M})\,.$ Since fibrations are closed under pullbacks, the morphisms $f^{a_1}_{n+1},f^{b_1}_{n+1}$ are fibrations and the objects $F^{a_1}_{n+1},F^{b_1}_{n+1}$ are fibrant. Finally, the cospans of the back and the front square are fibrant in the injective model structure, see for instance \cite{Models}. Now, if $\zf_{n-1}$ and $\zf_n$ are weak equivalences, these fibrant cospans are weakly equivalent, as $\zvf_n$ is obviously a weak equivalence. However, the pullback functor is a right Quillen functor if the category of cospans is equipped with its injective model structure, so that it sends weak equivalences between fibrant objects to weak equivalences: $\phi_{n+1}$ is a weak equivalence. This means that $E(\psi)=\Phi_\bullet$ is an objectwise weak equivalence if $\psi$ is a weak equivalence.
\end{proof}\smallskip

\begin{lem}\label{FunI}
Let $\tt M$ be a pointed model category. There exists a functor $I:\ell({\tt M})\to\ell({\tt M})$ that preserves weak equivalences and sends long homotopy fiber sequences $a_{\bullet}$ to long homotopy fiber sequences $\ff^a_\bullet$ made of fibrations $\ff^{a}_n:\fF^a_n\to \fF^a_{n-1}$ between fibrant objects $(n\in\{1,2,3\cdots\})\,.$
\end{lem}\smallskip

\begin{proof}
Let $a_\bullet,b_\bullet$ be objects of $\ell({\tt M})$ and $\Psi_\bullet$ an $\ell({\tt M})$--morphism between them. We will construct $I$ simultaneously and inductively on these objects and this morphism. In other words, we will construct long homotopy fiber sequences $I(a_\bullet)=:\ff^a_\bullet$ and $I(b_\bullet)=:\ff^b_\bullet$ and a morphism $I(\Psi_\bullet)=:\zY_\bullet$ of long homotopy fiber sequences between them. At the same time we build step by step a natural weak equivalence $\zw:\id_{\ell({\tt M})}\stackrel{\sim}{\Rightarrow} I\,,$ i.e., we build step by step a commutative diagram
\be
\begin{tikzcd}\label{Goal}
a_\bullet\arrow[r,"\zw^a_\bullet"]\ar[r, swap, "\sim"]\arrow[d,"\Psi_\bullet"]& \ff^a_\bullet\arrow[d,"\zY_\bullet"]\\
b_\bullet\arrow[r, swap, "\sim"] \ar[r,"\zw^b_\bullet"]& \ff^b_\bullet
\end{tikzcd}
\ee
If we apply $R_1$ to $\Psi_\bullet:a_\bullet\to b_\bullet\,,$ we get $\psi=(\psi_0,\psi_1):a_1\to b_1$ and if we apply $E$ to the latter, we get $\Phi_\bullet:f^{a_1}_\bullet\to f^{b_1}_\bullet\,.$ We choose the first two terms $\phi=(\phi_0,\phi_1):f^{a_1}_1\to f^{b_1}_1$ of $\Phi_\bullet$ as the first two terms $\zy=(\zy_0,\zy_1):\ff^{a}_1\to \ff^{b}_1$ of $\zY_\bullet$ (this includes choosing the first two terms $(F^{a_1}_0,F^{a_1}_1)$ as the first two terms $(\fF^a_0,\fF^a_1)$ and similarly for $b$). Let us remember that above we defined $\zy$ as follows (see Diagram \eqref{Start}):
\pagebreak
\be\label{Start2}
\begin{tikzcd}[row sep=1.2em, column sep=1.5em]
A_1\ar[rrr,"a_1"]\ar[dd,"\psi_1"]\ar[dr,"\sim", tail] & & & A_0\ar[dd, near start, "\psi_0"]\ar[dr,"\sim", tail] &\\
& RA_1\arrow[r, "\za(Ra_1)"]\arrow[r,swap,"\sim"]\arrow[dd, swap, near start, "R\psi_1"]&\fF^{a}_1\arrow[rr,two heads, near start, "\ff^{a}_1"]\arrow[dd, near start,"\zy_1"]& &\fF^{a}_0\arrow[dd, "\zy_0"]\\
B_1\ar[rrr,"b_1"]\ar[dr,"\sim", tail] & & & B_0\ar[dr,"\sim", tail] &\\
& RB_1\arrow[r, "\za(Rb_1)"]\arrow[r,swap,"\sim"]&\fF^{b}_1\arrow[rr,two heads,"\ff^{b}_1"]& &\fF^{b}_0
\end{tikzcd}
\ee

Diagram \eqref{Goal} is a diagram in $\ell({\tt M})\,,$ so that every arrow is a sequence of commutative cubes. Its commutativity means that the $n$th cube of the down-right composite (which is the composite of a cube $A\rightsquigarrow B$ and a cube $B\rightsquigarrow\fF^b\,,$ where $\rightsquigarrow$ denotes a morphism between squares) coincides with the $n$th cube of the right-down composite (which is the composite of a cube $A\rightsquigarrow\fF^a$ and a cube $\fF^a\rightsquigarrow \fF^b$), for every $n\in\{1,2,3,\cdots\}$. The commutative diagram \eqref{Start2} means that half of this condition is satisfied for the $1$st cubes and it shows that $\zw^a=(\zw^a_0,\zw^a_1):a_1\to \ff^a_1$ and $\zw^b$ consist of two weak equivalences and that $\zy$ is made of weak equivalences if $\psi$ is. Assume now that the sequences $\fF^a_\bullet\,$ $\fF^b_\bullet\,,$ $\ff^a_\bullet\,,$ $\ff^b_\bullet\,,$ $\zw^a_\bullet\,,$ $\zw^b_\bullet$ and $\zY_\bullet$ have been constructed with all the required properties and implicit vertices or arrows up to order $n\ge 1$ and that the commutation condition of Diagram \eqref{Goal} is fulfilled up to half the condition for the $n$th cubes.
\be\label{TouDia}
\begin{tikzcd}[row sep=0.5em, column sep=0.75em]
&&&&A_{n+1}\ar[ddllll]\arrow[dr,dashed] \arrow[rr] \arrow[dddd] && A_n \ar[ddllll]\arrow[dddd] \arrow[dr,"\sim"] \\
&&&&& \fF^a_{n+1}\ar[ddllll,dashed]\arrow[dddd]\arrow[rr,crossing over,two heads] &&  \fF^a_n \ar[ddllll,crossing over]\arrow[dddd,two heads] \\
B_{n+1}\arrow[dr,dashed] \arrow[rr] \arrow[dddd] && B_n \arrow[dddd] \arrow[dr,"\sim"] \\
& \fF^b_{n+1}\arrow[rr,crossing over,two heads] &&  \fF^b_n \arrow[dddd,two heads] \\
&&&&C_{A_n} \ar[ddllll]\arrow[rr] \arrow[dr,"\sim"] && A_{n-1} \ar[ddllll]\arrow[dr,"\sim"] \\
&&&&& C_{\fF^a_n} \ar[ddllll]\arrow[rr, two heads] && \fF^a_{n-1}\ar[ddllll]\\
C_{B_n} \arrow[rr] \arrow[dr,"\sim"] && B_{n-1} \arrow[dr,"\sim"] \\
& C_{\fF^b_n} \arrow[from=uuuu, crossing over]\arrow[rr, two heads] && \fF^b_{n-1}
\end{tikzcd}
\ee

In \eqref{TouDia}, the back commutative square of the right cube is a homotopy fiber sequence of $a_\bullet\,.$ The right commutative square of this cube is given by $\zw^a_n,a_n,\ff^a_n$ and $\zw^a_{n-1}\,.$ To get the lower commutative square, we decompose $C_{A_n}\to \fF^a_{n-1}$ into a weak equivalence followed by a fibration using our fixed functorial factorization (see Remark \ref{Fixed}). Now we take the pullback $$\fF^a_{n+1}:=\fF^a_n\times_{\fF^a_{n-1}} C_{\fF^a_{n}}\,,$$ use the fact that fibrations are closed under pullbacks, and complete the right commutative cube by the universal arrow $A_{n+1}\dashrightarrow \fF^a_{n+1}\,$. Since the back square is in particular a model square this universal arrow is a weak equivalence. The left commutative cube is constructed similarly. As a result of Corollary \ref{CorLur} the square $\fF^a$ (the front square of the right commutative cube) is a homotopy fiber sequence and extends $\ff^a_\bullet$ or $(\fF^a_\bullet,\ff^a_\bullet)$ which is made of fibrations between fibrant objects to order $n+1\,$. Analogously the square $\fF^b$ extends $(\fF^b_\bullet, \ff^b_\bullet)\,.$ The commutative cube $A\rightsquigarrow \fF^a$ (resp., $B\rightsquigarrow \fF^b$) is an objectwise weak equivalence of homotopy fiber sequences that extends the $\ell({\tt M})$--morphism $\zw^a_\bullet$ (resp., $\zw^b_\bullet$) to order $n+1$.\medskip

We now describe the six cubes of Diagram \eqref{TouDia} that contain sloping arrows to the left. In fact, the top and bottom of these cubes are fully described as soon as the two back and two front cubes are. The commutative back cube $A\rightsquigarrow B$ is a cube of the sequence $\Psi_\bullet:a_\bullet\to b_\bullet\,.$ The front cube $A\fF^a\rightsquigarrow B\fF^b$ is counterpart in order $n$ to the commutative cube \eqref{Start2} and it is commutative in view of the induction assumption that the commutation condition of Diagram \eqref{Goal} is satisfied up to half the condition for the $n$th cubes. Since the trivial cofibration - fibration factorization system used is functorial, the commutative rectangle
\be\label{FunFac}
\begin{tikzcd}
C_{A_n}\arrow[r,tail,"\sim"]\arrow[d]&C_{\fF^a_n}\arrow[r,two heads]\arrow[d]&\fF^a_{n-1}\arrow[d]\\
C_{B_n}\arrow[r,tail,"\sim"]&C_{\fF^b_n}\arrow[r,two heads]&\fF^b_{n-1}
\end{tikzcd}
\ee
induces a central vertical arrow that makes the left and right squares (which are two of the six bottom squares of \eqref{TouDia}) commutative. The universal arrow $\fF^a_{n+1}\dashrightarrow \fF^b_{n+1}$ renders the upper and left square of the front cube $\fF^a\rightsquigarrow \fF^b$ commutative. Finally, in the back cube $A\fF^aC\rightsquigarrow B\fF^bC$, the top square is commutative because of the uniqueness of the universal arrow. This completes the description of the fully commutative diagram \eqref{TouDia}. Notice that the commutative cube $\fF^a\rightsquigarrow\fF^b$ extends $\zY_\bullet$ to order $n+1$ and remember that we still have to show that the complete commutation condition for the $n$th cubes is now fulfilled, that is, the composite cubes $A\rightsquigarrow B\rightsquigarrow \fF^b$ and $A\rightsquigarrow \fF^a\rightsquigarrow \fF^b$ in \eqref{TouDia} coincide. Since the diagram \eqref{TouDia} is fully commutative, this requirement is met. As the full commutation condition for the $n$th cubes includes the first half of the condition for the $(n+1)$th cubes, induction works. Eventually, if $\Psi_\bullet$ is an objectwise weak equivalence, it follows from the 2-out-of-3 axiom that $\zY_\bullet$ is an objectwise weak equivalence. This completes the proof.
\end{proof}

From the previous proof it follows that:

\begin{cor}\label{NatWeq}
There is a natural weak equivalence $$\zw:\id_{\ell({\tt M})}\stackrel{\sim}{\Rightarrow} I\,.$$
\end{cor}

The next lemma will allow us to prove Theorem \ref{EquivHoCat}.

\begin{lem}\label{NatWeqs}
There are natural weak equivalences $$\zvw:\id_{\tt M^\s\!}\stackrel{\sim}{\Rightarrow}R_1\circ E$$ and $$\upsilon:E\circ R_1\stackrel{\sim}{\Rightarrow} I\;.$$
\end{lem}

\begin{proof}
If we go back to the original notation in Diagram \eqref{Start2}, we have the commutative cube

\be\label{Start3}
\begin{tikzcd}[row sep=1.2em, column sep=1.5em]
A_1\ar[rrr,"a_1"]\ar[dd,"\psi_1"]\ar[dr,"\sim", tail] & & & A_0\ar[dd, near start, "\psi_0"]\ar[dr,"\sim", tail] &\\
& RA_1\arrow[r, "\za(Ra_1)"]\arrow[r,swap,"\sim"]\arrow[dd, swap, near start, "R\psi_1"]&F^{a_1}_1\arrow[rr,two heads, near start, "f^{a_1}_1"]\arrow[dd, near start,"\zf_1"]& &F^{a_1}_0\arrow[dd, "\zf_0"]\\
B_1\ar[rrr,"b_1"]\ar[dr,"\sim", tail] & & & B_0\ar[dr,"\sim", tail] &\\
& RB_1\arrow[r, "\za(Rb_1)"]\arrow[r,swap,"\sim"]&F^{b_1}_1\arrow[rr,two heads,"f^{b_1}_1"]& &F^{b_1}_0
\end{tikzcd}
\ee
for every $\tt M^\s$--morphism $\psi=(\psi_0,\psi_1):a_1\to b_1\,.$ The $\tt M^\s$--morphism $\zvw^{a_1}:a_1\to f^{a_1}_1$ is the upper commutative square which is a weak equivalence as it should be. The naturality of $\zvw$ precisely means that the total left square and the right square commute, which is the case. \medskip

Next we construct $\zu$ by proceeding similarly to the proof of Lemma \ref{FunI}. For every $a_\bullet\in\ell({\tt M})$ we must define an $\ell({\tt M})$--morphism $\zu^a_\bullet:f^{a_1}_\bullet\stackrel{\sim}{\to} \ff^a_\bullet\,,$ i.e., we have to define a sequence of commutative cubes which are objectwise weak equivalences. Moreover, for every $\ell({\tt M})$--morphism $\Psi_\bullet:a_\bullet\to b_\bullet\,,$ we must show that the diagram
\be
\begin{tikzcd}\label{Goal2}
f^{a_1}_\bullet\arrow[r,swap, "\zu^a_\bullet"]\ar[r, "\sim"]\arrow[d,"\Phi_\bullet"]& \ff^a_\bullet\arrow[d,"\zY_\bullet"]\\
f^{b_1}_\bullet\arrow[r, "\sim"] \ar[r, swap, "\zu^b_\bullet"]& \ff^b_\bullet
\end{tikzcd}
\ee
commutes (we have used the notations introduced above). Since we set $\ff^a_1=f^{a_1}_1$ (see proof of Lemma \ref{FunI}), we choose the identity maps as first two components $\zu^a=(\zu^a_0,\zu^a_1):f^{a_1}_1\to \ff^a_1$ of $\zu^a_\bullet\,.$ Since we also set $(\zy_0,\zy_1)=(\zf_0,\zf_1)\,,$ the naturality condition is so far fulfilled. Assume now that the sequences $\zu^a_\bullet,\zu^b_\bullet$ of commutative cubes which are objectwise weak equivalences have been constructed with their implicit arrows up to order $n\ge 1$ and that the commutation condition of \eqref{Goal2} is fulfilled up to half the condition for $n$th cubes.\medskip

To extend the sequences and the commutativity to order $n+1\,,$ we start by describing the total parallelepiped in the diagram
\be
\begin{tikzcd}[row sep=0.5em, column sep=0.5em]\label{FFF}
&0\ar[rr,tail,"\sim"]\ar[dd,very near end,"\sim"]\ar[ld,swap,"\sim"]&&C_{F^{a_1}_n}\ar[rr,two heads]\ar[dd,very near start,"S_2"]\ar[ld,swap,"S_1"]&&F^{a_1}_{n-1}\ar[dd,"\sim"]\ar[ld] \\
0\ar[rr,tail,very near end,"\sim"]\ar[dd,"\sim"]&&C_{F^{b_1}_n}\ar[rr,two heads]\ar[dd,swap,near end,"S_3"]&&F^{b_1}_{n-1}\ar[dd,very near start,"\sim"]& \\
&C_{A_n}\ar[rr,tail,very near start,"\sim"]\ar[ld]&&C_{\fF^a_n}\ar[rr,two heads]\ar[ld,near start,"S_4"]&&\fF^a_{n-1}\ar[ld] \\
C_{B_n}\ar[rr,tail,"\sim"]&&C_{\fF^b_n}\ar[from=uuur,swap]\ar[rr,two heads]&&\fF^b_{n-1}&
\end{tikzcd}
\ee
The arrow $F^{a_1}_{n-1}\to F^{b_1}_{n-1}$ is $\zf_{n-1}\,,$ the arrow $F^{a_1}_{n-1}\to \fF^a_{n-1}$ is $\zu^a_{n-1}\,,$ the arrow $C_{A_n}\to \fF^a_{n-1}$ is the composite of $C_{A_n}\to A_{n-1}$ in $a_\bullet$ and $\zw^a_{n-1}:A_{n-1}\to \fF^a_{n-1}$ (see Corollary \ref{NatWeq}), the arrow $C_{A_n}\to C_{B_n}$ is the implicit arrow in the $n$th cube of $\Psi_\bullet\,,$ and $\fF^a_{n-1}\to \fF^b_{n-1}$ is $\zy_{n-1}\,.$ All four rectangles of the parallelepiped commute (for the commutativity of the lower rectangle, see commutative diagram \eqref{TouDia}). By definition, we get the vertex $C_{F^{a_1}_n}$ (resp., $C_{\fF^a_n}$), if we apply the fixed functorial trivial cofibration - fibration factorization system (see Remark \ref{Fixed}) to $0\to F^{a_1}_{n-1}$ (resp., $C_{A_n}\to\fF^a_{n-1}$). Since the system is functorial, we also get the five central arrows, $S_1,S_2,S_3,S_4$ and the diagonal arrow. A priori we even get two diagonal arrows, one, say $D_1\,,$ induced by the commutative square, which is made up of the top and front rectangle, and one, $D_2\,,$ from the commutative square, which is made up of the back and bottom rectangle. From the functoriality of the factorization system follows that $D_1=S_3\circ S_1$ and $D_2=S_4\circ S_2\,.$ However, these two commutative squares coincide as the right square of \eqref{FFF} commutes (since half of the commutation condition is satisfied for the $n$th cubes that correspond to \eqref{Goal2}), so that $D_1=D_2\,,$ i.e., so that the central square of \eqref{FFF} commutes. \medskip

In the right cube of Diagram \eqref{TouDia2}
\be\label{TouDia2}
\begin{tikzcd}[row sep=0.25em, column sep=0.75em]
&&&&F^{a_1}_{n+1}\ar[ddllll]\arrow[dr,dashed] \arrow[rr,two heads] \arrow[dddd] && F^{a_1}_n \ar[ddllll]\arrow[dddd,two heads] \arrow[dr,"\sim"] \\
&&&&& \fF^a_{n+1}\ar[ddllll,dashed]\arrow[dddd]\arrow[rr,crossing over,two heads] &&  \fF^a_n \ar[ddllll,crossing over]\arrow[dddd,two heads] \\
F^{b_1}_{n+1}\arrow[dr,dashed] \arrow[rr,two heads] \arrow[dddd] && F^{b_1}_n \arrow[dddd,two heads] \arrow[dr,"\sim"] \\
& \fF^b_{n+1}\arrow[rr,crossing over,two heads] &&  \fF^b_n \arrow[dddd,two heads] \\
&&&&C_{F^{a_1}_n} \ar[ddllll]\arrow[rr,two heads] \arrow[dr,"\sim"] && F^{a_1}_{n-1} \ar[ddllll]\arrow[dr,"\sim"] \\
&&&&& C_{\fF^a_n} \ar[ddllll]\arrow[rr, two heads] && \fF^a_{n-1}\ar[ddllll]\\
C_{F^{b_1}_n} \arrow[rr,two heads] \arrow[dr,"\sim"] && F^{b_1}_{n-1} \arrow[dr,"\sim"] \\
& C_{\fF^b_n} \arrow[from=uuuu, crossing over]\arrow[rr, two heads] && \fF^b_{n-1}
\end{tikzcd}
\ee
the right commutative square is given by $\zu^a_n\,,$ $f^{a_1}_n\,,$ $\ff^a_n$ and $\zu^a_{n-1}\,.$ The lower commutative square is the right back square of Diagram \eqref{FFF}. The vertices $F^{a_1}_{n+1}$ and $\fF^a_{n+1}$ have been defined as pullbacks and the resulting commutative squares identified as homotopy fiber sequences. Since the back pullback square is a homotopy fiber sequence, the universal dashed arrow between these vertices is a weak equivalence. Hence the right commutative cube in \eqref{TouDia2} extends the sequence $\zu^a_\bullet$ to order $n+1\,.$ The left commutative cube is built similarly and it extends $\zu^b_\bullet$ to order $n+1\,$. The objectwise weak equivalence condition is satisfied.\medskip

We now describe the six cubes in Diagram \eqref{TouDia2} that contain sloping arrows to the left. In fact, the top and bottom of these cubes are fully described as soon as the two back and two front cubes are. The back cube $F^{a_1}\rightsquigarrow F^{b_1}$ (resp., front cube $\fF^a\rightsquigarrow \fF^b$) is part of $\Phi_\bullet$ (resp., of $\zY_\bullet$) and therefore it commutes. The front cube $F^{a_1}\fF^a\rightsquigarrow F^{b_1}\fF^b$ is commutative because of the induction assumption that the commutation condition of Diagram \eqref{Goal2} is satisfied up to half the condition for the $n$th cubes. Finally, the vertical faces of the back cube $F^{a_1}\fF^aC\rightsquigarrow F^{b_1}\fF^bC$ are commutative since the are faces of other commutative cubes. Its lower face is the central square of \eqref{FFF}, which commutes. Its upper face commutes because of the uniqueness of the universal arrow (see also \eqref{TouDia}). We still have to show that the complete commutation condition for the $n$th cubes is now fulfilled, that is, the composite cubes $F^{a_1}\rightsquigarrow\fF^a\rightsquigarrow\fF^b$ and $F^{a_1}\rightsquigarrow F^{b_1}\rightsquigarrow \fF^b$ in \eqref{TouDia2} coincide. Since the diagram \eqref{TouDia2} is fully commutative, this requirement is met.
\end{proof}

We remind the reader of the following lemma \cite{Models}, as it simplifies the proof of Theorem \ref{EquivHoCat}.

\begin{lem}\label{NatTraMHoM}
Let $\tt C$ be a category which is equipped with a distinguished family $W\!$ of morphisms called weak equivalences, let $\tt E$ be any category and let $\cF,\cG\in\tt Fun(C,E)$ be functors which send weak equivalences to isomorphisms. A family $\zy_X:\cF(X)\to \cG(X)$ of $\,\tt E$--maps indexed by the objects $X$ of $\,\tt C$ is a natural transformation $\op{Ho}(\zy):\op{Ho}(\cF)\Rightarrow \op{Ho}(\cG)$ if and only if it is a natural transformation $\zy:\cF\Rightarrow \cG\,.$
\end{lem}

\begin{lem}\label{DescWeq}
Let $\tt C,D$ be categories with distinguished families $W_{\tt C},W_{\tt D}$ of morphisms and let $F,G\in\tt Fun(C,D)$ be functors which preserve these weak equivalences. A natural weak equivalence $\zh:F\stackrel{\sim}{\Rightarrow}G$ induces a natural isomorphism $$\op{Ho}(\zh):\op{Ho}(F)\stackrel{\cong}{\Rightarrow}\op{Ho(G)}\,.$$
\end{lem}

\begin{proof}[Proof of Lemma \ref{DescWeq}] If we whisker the natural transformation $\zh$ with the localization functor $\zg_{\tt D}$ we get the natural transformation $$\zg_{\tt D}\star\zh:\zg_{\tt D}\circ F\Rightarrow \zg_{\tt D}\circ G\;,$$ all whose components $(\zg_{\tt D}\star\zh)_X=\zg_{\tt D}(\zh_X)$ are isomorphisms. If we apply Lemma \ref{NatTraMHoM} to $\cF=\zg_{\tt D}\circ F\,,$ $\cG=\zg_{\tt D}\circ G$ and the family $(\zg_{\tt D}\star\zh)_X\,,$ and if we write as usual $\op{Ho}(-)$ instead of $\op{Ho}(\zg_{\tt D}\circ -)\,,$ we get the announced result.
\end{proof}

\begin{proof}[Proof of Theorem \ref{EquivHoCat}]
If we apply Lemma \ref{DescWeq} to the natural weak equivalence $$\zw:\id_{\ell(M)}\stackrel{\sim}\Rightarrow I$$ of Corollary \ref{NatWeq}, we get a natural isomorphism \be\label{NatIso1}\op{Ho}(\zw):\id_{{\tt Ho}(\ell({\tt M}))}\stackrel{\cong}{\Rightarrow}\op{Ho}(I)\;,\ee since $\op{Ho}(\id_{\ell({\tt M})})=\id_{{\tt Ho}(\ell({\tt M}))}\,.$\medskip

Analogously, the natural weak equivalences $$\zvw:\id_{\tt M^\s}\stackrel{\sim}{\Rightarrow}R_1\circ E\quad\text{and}\quad\zu:E\circ R_1\stackrel{\sim}{\Rightarrow}I\;$$ of Lemma \ref{NatWeqs} induce natural isomorphisms \be\label{NatIso2}\op{Ho}(\zvw):\id_{{\tt Ho}({\tt M}^\s)}\stackrel{\cong}{\Rightarrow}\op{Ho}(R_1)\circ\op{Ho}(E)\quad\text{and}\quad\op{Ho}(\zu):\op{Ho}(E)\circ\op{Ho}(R_1)\stackrel{\cong}{\Rightarrow}\op{Ho}(I)\;.\ee Indeed, as $\op{Ho}(R_1\circ E)$ is the unique endofunctor of ${\tt Ho}({\tt M}^\s)$ such that $$\op{Ho}(R_1\circ E)\circ\zg_{{\tt M}^\s}=\zg_{{\tt M}^\s}\circ R_1\circ E\;,$$ and as $\op{Ho(R_1)}$ (resp., $\op{Ho(E)}$) is the unique functor from ${\tt Ho}(\ell({\tt M}))$ to ${\tt Ho}({\tt M}^\s)$ (resp., from ${\tt Ho}({\tt M}^\s)$ to ${\tt Ho}(\ell({\tt M}))$) such that $$\op{Ho}(R_1)\circ\zg_{\ell({\tt M})}=\zg_{{\tt M}^\s}\circ R_1\quad (\,\text{resp., such that}\;\op{Ho}(E)\circ \zg_{{\tt M}^\s}=\zg_{\ell({\tt M})}\circ E)\;,$$ we have that $$\op{Ho}(R_1\circ E)=\op{Ho}(R_1)\circ\op{Ho}(E)$$ and similarly for $\op{Ho}(E\circ R_1)\,.$\medskip

If we combine \eqref{NatIso1} and \eqref{NatIso2} we see that $$\op{Ho}(R_1):{\tt Ho}(\ell({\tt M}))\rightleftarrows{\tt Ho}({\tt M}^\s):\op{Ho}(E)$$ is an equivalence of categories.
\end{proof}

As Theorem \ref{EquivHoCat} has now been proven, the map $$\op{Ho}(R_1)_{a_\bullet b_\bullet}:\op{Hom}_{{\tt Ho(\ell(M))}}(a_\bullet, b_\bullet)\to \op{Hom}_{\tt Ho(M^\s)}(a_1,b_1)\quad (a_\bullet,b_\bullet\in\ell({\tt M}))$$ in Equation \eqref{FullFaith} {\it is} a bijection. Its inverse can be described explicitly using the functor $E$ of Lemma \ref{FunE} and the natural weak equivalences $\zw$ and $\zu$ of Corollary \ref{NatWeq} and Lemma \ref{NatWeqs}:

\begin{prop}\label{Explicit}
Let $\tt M$ be a pointed model category, let $a_\bullet,b_\bullet\in\ell({\tt M})\,,$ let $a_1:=R_1(a_\bullet)\,,$ $b_1:=R_1(b_\bullet)$ and let $$\xi\in\op{Hom}_{\tt Ho(M^\s)}(a_1,b_1)\;.$$ The unique preimage $$\Xi_\bullet\in\op{Hom}_{{\tt Ho(\ell(M))}}(a_\bullet, b_\bullet)$$ of $\,\xi$ under the bijection $\op{Ho}(R_1)_{a_\bullet b_\bullet}$ is the composite of $\,\tt Ho(\ell(M))$--morphisms

\be\label{Claim1}
\begin{tikzcd}
\\
a_\bullet\arrow[r,"\zg(\zw^a_\bullet)"]\arrow[r,swap,"\cong"]& \ff^a_\bullet\ar[r,"\zg(\zu^a_\bullet)^{-1}"]\ar[r,swap,"\cong"] & f^{a_1}_\bullet \ar[r,"\op{Ho}(E)(\xi)"]& f^{b_1}_\bullet\ar[r,"\zg(\zu^b_\bullet)"]\ar[r,swap,"\cong"]&\ff^b_\bullet\ar[r,"\zg(\zw^b_\bullet)^{-1}"]\ar[r,swap,"\cong"]&b_\bullet\;,\\
\end{tikzcd}
\ee
where $\zg=\zg_{\ell({\tt M})}\,.$ If $\xi$ is the class of weak equivalences and formal reversals of weak equivalences, then $\op{Ho}(E)(\xi)$ is an isomorphism and so is $\Xi_\bullet\,$. In this case, we refer to $\Xi_\bullet$ as the {\bf canonical isomorphism} in the homotopy category that extends $\xi\,.$
\end{prop}

\begin{proof}
The unique preimage $\Xi_\bullet$ is a class $$\Xi_\bullet=[a_\bullet \stackrel{\Psi_\bullet}{\longrightarrow} c_\bullet \stackrel{\zW_\bullet^{-1}}{\longsquiggly} d_\bullet \stackrel{\Psi'_\bullet}{\longrightarrow}\cdots \stackrel{\zW_\bullet'^{-1}}{\longsquiggly} b_\bullet]\;$$ of morphisms $\to$ of $\ell({\tt M})$ and formal reversals $\rightsquigarrow$ of weak equivalences $\stackrel{\sim}{\leftarrow}$ of $\ell({\tt M})\,.$ If we construct for each one of these $\ell({\tt M})$--morphisms $\Psi_\bullet,\zW_\bullet,\Psi'_\bullet,\cdots,\zW'_\bullet$ the commutative $\ell({\tt M})$--squares \eqref{Goal} and \eqref{Goal2} which encode the naturality of $\zw$ and $\zu\,,$ we get the following amalgamation of commutative $\ell({\tt M})$--squares
\be\label{ComAmaSqu}
\begin{tikzcd}
a_\bullet\arrow[d,"\sim"]\arrow[d,swap,"\zw^a_\bullet"]\arrow[r,"\Psi_\bullet"]& c_\bullet\arrow[d,"\sim"]\ar[d,swap,"\zw^c_\bullet"]& d_\bullet\arrow[l,"\sim"]\arrow[l,swap,"\zW_\bullet"]\arrow[r,"\Psi'_\bullet"]
\arrow[d,"\sim"]\ar[d,swap,"\zw^d_\bullet"]&\cdots&\arrow[l,"\sim"]\arrow[l,swap,"\zW'_\bullet"]b_\bullet\arrow[d,"\sim"]\arrow[d,swap,"\zw^b_\bullet"]\\
\ff^a_\bullet\arrow[r,"I(\Psi_\bullet)"]&\ff^c_\bullet&\ff^d_\bullet\arrow[l,"\sim"]\ar[l,swap,"I(\zW_\bullet)"]\arrow[r,"I(\Psi'_\bullet)"]&\cdots& \ff^b_\bullet\arrow[l,"\sim"]\ar[l,swap,"I(\zW'_\bullet)"]\\
\arrow[u,swap,"\sim"] f^{a_1}_\bullet\arrow[r,swap,"\cI(\Psi_{\bullet})"]\ar[u,"\zu^a_\bullet"]&\arrow[u,swap,"\sim"]\ar[u,"\zu^c_\bullet"]f^{c_1}_\bullet&\arrow[u,swap,"\sim"]
\ar[u,"\zu^d_\bullet"]f^{d_1}_\bullet \arrow[l,swap,"\sim"]\ar[l,"\cI(\zW_{\bullet})"]\arrow[r,swap,"\cI(\Psi'_{\bullet})"]&\cdots&\arrow[u,swap,"\sim"]
\ar[u,"\zu^b_\bullet"] f^{b_1}_\bullet\arrow[l,swap,"\sim"]\ar[l,"\cI(\zW'_{\bullet})"]\\
\end{tikzcd}
\ee
where $\cI:=E\circ R_1\,.$ If we apply $\zg=\zg_{\ell({\tt M})}$ to \eqref{ComAmaSqu}, we get a commutative $\tt Ho(\ell(M))$--diagram in which the images of the weak equivalences are isomorphisms. It is straightforwardly seen that the composite \be\label{Comp1}\zg(\zW'_\bullet)^{-1}\circ\cdots\circ\zg(\Psi'_{\bullet})\circ\zg(\zW_\bullet)^{-1}\circ\zg(\Psi_\bullet)\ee of the first row is equal to the composite \be\label{Comp2}\zg(\zw^b_\bullet)^{-1}\circ\zg(\zu^b_\bullet)\circ\zg(\cI(\zW'_\bullet))^{-1}\circ\cdots\circ\zg(\cI(\Psi'_\bullet))\circ\zg(\cI(\zW_\bullet))^{-1}\circ
\zg(\cI(\Psi_\bullet))\circ\zg(\zu^a_\bullet)^{-1}\circ\zg(\zw^a_\bullet)\;\ee of the first column, the last row and the last column. Since for any morphism $\Phi_\bullet$ of $\ell({\tt M})$ we have $\zg(\Phi_\bullet)=[\Phi_\bb]\,,$ since for any weak equivalence $W_\bb$ of $\ell({\tt M})$ we have $[W_\bb]^{-1}=[W_\bb^{-1}]$ and since a composite of classes is the class of the corresponding concatenation, the composite \eqref{Comp1} can be written \be\label{Comp3}[\zW'_\bullet]^{-1}\circ\cdots\circ [\Psi'_\bullet]\circ[\zW_\bullet]^{-1}\circ[\Psi_\bullet]=[\zW'^{-1}_\bullet]\circ\cdots\circ [\Psi'_\bullet]\circ[\zW_\bullet^{-1}]\circ[\Psi_\bullet]=[\Psi_\bb\,\zW_\bb^{-1}\,\Psi'_\bb\,\cdots\,\zW'^{-1}_\bb]=\Xi_\bullet\;.\ee On the other hand, since for any morphism $\Phi_\bb$ any weak equivalence $W_\bb$ we have $$\op{Ho}(\cI)[\Phi_\bb]=\zg(\cI(\Phi_\bb))\quad\text{and}\quad\op{Ho}(\cI)[W_\bb^{-1}]=\zg(\cI(W_\bb))^{-1}\;,$$ the partial composite in \eqref{Comp2} of the factors that contain $\cI=E\circ R_1$ is equal to  \be\label{Comp4}\op{Ho}(\cI)[\zW'^{-1}_\bb]\circ\cdots\circ\op{Ho}(\cI)[\Psi'_\bb]\circ\op{Ho}(\cI)[\zW_\bullet^{-1}]\circ\op{Ho}(\cI)[\Psi_\bb]=\op{Ho}(E)(\op{Ho}(R_1)(\Xi_\bullet))=
\op{Ho}(E)(\xi)\;.\ee The statement \eqref{Claim1} now follows from \eqref{Comp1}, \eqref{Comp2}, \eqref{Comp3} and \eqref{Comp4}.\medskip

As for the second statement in Proposition \ref{Explicit}, it suffices to observe that if $$\xi=[w\,\zw^{-1}\,w'\,\cdots\,\zw'^{-1}]\,,$$ then $$\op{Ho}(E)(\xi)=\op{Ho}(E)[\zw'^{-1}]\circ\cdots\circ\op{Ho}(E)[w']\circ\op{Ho}(E)[\zw^{-1}]\circ\op{Ho}(E)[w]=$$ $$\zg(E(\zw'))^{-1}\circ\cdots\circ\zg(E(w'))\circ\zg(E(\zw))^{-1}\circ\zg(E(w))\;,$$ which is a composite of isomorphisms, as $E$ preserves weak equivalences.
\end{proof}

Theorem \ref{EquivHoCat} has two more corollaries that we will apply later.

\begin{cor}\label{EquivHoCatSingle}
Let $\tt M$ be a pointed model category. The localization $$\op{Ho}(R_1):{\tt Ho(h(M))\to Ho(M^\s)}$$ of the restriction functor of homotopy fiber sequences yields an equivalence of categories between the homotopy category of the category of homotopy fiber sequences of $\,\tt M$ and the homotopy category of the category of morphisms of $\,\tt M\,.$
\end{cor}

\begin{proof}
Corollary \ref{EquivHoCatSingle} is a consequence of the proofs of Lemma \ref{FunE}, Lemma \ref{FunI} and Lemma \ref{NatWeqs} in which we stop the iterative process after the first step.
\end{proof}

\begin{cor}\label{PresHoFibSeq}
Let $\tt M$ and $\tt N$ be pointed model categories and let $F\in\tt Fun(M,N)$ be a functor which preserves weak equivalences between fibrant objects and sends the zero object $0_{\tt M}$ of $\,\tt M$ to an acyclic object $F(0_{\tt M})$ of $\,\tt N\,$. In addition, let $$a: A_2\xrightarrow{a_2} A_1 \xrightarrow{a_1} A_0\quad\text{and}\quad b: B_2\xrightarrow{b_2} B_1 \xrightarrow{b_1} B_0$$ be objectwise fibrant homotopy fiber sequences of $\tt h(M)$ and assume that there exists a weak equivalence $$w=(w_0,w_1):a_1 \to b_1$$ in $\tt M^\s$. Then $F$ preserves the homotopy fiber sequence $a\in\tt h(M)\,,$ i.e., we have $F(a)\in\tt h(N)$ if and only if it preserves the homotopy fiber sequence $b\in{\tt h(M)}\,,$ i.e., if and only if we have $F(b)\in\tt h(N)\,.$ If $F$ preserves all weak equivalences, the result is true without $a$ and $b$ being objectwise fibrant.
\end{cor}

\begin{proof}
In view of Corollary \ref{EquivHoCatSingle}, Proposition \ref{Explicit} is also valid if we replace $\ell({\tt M})$ by $\tt h(M)\,.$ The unique preimage $\Xi\in\op{Hom}_{\tt Ho(h(M))}(a,b)$ of $$\xi:=\zg_{\tt M^\s}(w)=[w]\in\op{Hom}_{\tt Ho(M^\s)}(a_1,b_1)$$ under the bijection $\op{Ho}(R_1)_{ab}$ is
$$
\begin{tikzcd}
\Xi:\;\; a\arrow[r,"\zg(\zw^a)"]\arrow[r,swap,"\cong"]& \ff^a\ar[r,"\zg(\zu^a)^{-1}"]\ar[r,swap,"\cong"] & f^{a_1}\ar[r,"\zg(E(w))"]\ar[r,swap,"\cong"]& f^{b_1}\ar[r,"\zg(\zu^b)"]\ar[r,swap,"\cong"]&\ff^b\ar[r,"\zg(\zw^b)^{-1}"]\ar[r,swap,"\cong"]&b\;.
\end{tikzcd}
$$
This means that $\Xi$ is the class of
$$
\begin{tikzcd}
a\arrow[r,"\zw^a"]\arrow[r,swap,"\sim"]& \ff^a\ar[r,rightsquigarrow,"(\zu^a)^{-1}"] & f^{a_1}\ar[r,"E(w)"]\ar[r,swap,"\sim"]& f^{b_1}\ar[r,"\zu^b"]\ar[r,swap,"\sim"]&\ff^b\ar[r,rightsquigarrow,"(\zw^b)^{-1}"]&b\;,
\end{tikzcd}
$$
so that we have a zigzag
\be\label{SeqWeqHFS}
\begin{tikzcd}
a\arrow[r,"\zw^a"]\arrow[r,swap,"\sim"]& \ff^a & f^{a_1}\ar[l,swap,"\zu^a"]\ar[l,"\sim"]\ar[r,"E(w)"]\ar[r,swap,"\sim"]& f^{b_1}\ar[r,"\zu^b"]\ar[r,swap,"\sim"]&\ff^b&b\ar[l,swap,"\zw^b"]\ar[l,"\sim"]\;
\end{tikzcd}
\ee
of weak equivalences between objectwise fibrant homotopy fiber sequences. It now suffices to show that if there is a weak equivalence $W:d\stackrel{\sim}\to e$ between objectwise fibrant homotopy fiber sequences $d,e\in\tt h(M)\,$, then $F(d)\in\tt h(N)$ if and only if $F(e)\in\tt h(N)\,.$ Weak equivalence of homotopy fiber sequences means of course a morphism of homotopy fiber sequences which is objectwise a weak equivalence. Hence $W$ is a commutative $\tt M$--cube which is objectwise a weak equivalence of $\tt M$ (see \eqref{ComCub2} and omit $F$). If we apply $F$ to $W$ we get a commutative $\tt N$--cube which is objectwise a weak equivalence of $\tt N$ (see \eqref{ComCub2}) ($\star_1$). From Proposition \ref{ComPara} it follows that $F(d)$ is a model square if and only if $F(e)$ is one. Since $C_D$ and $F(0_{\tt M})$ are acyclic, i.e., since $0_{\tt M}\stackrel{\sim}{\to}C_D$ and $0_{\tt N}\stackrel{\sim}{\to}F(0_{\tt M})\,$, we get that $0_{\tt N}\stackrel{\sim}{\to}F(0_{\tt M})\stackrel{\sim}{\to}F(C_D)$ ($\star_2$), so that $F(C_D)$ is acyclic. The same is true for $F(C_E)\,.$ Hence $F(d)$ is a homotopy fiber sequence if and only if $F(e)$ is a homotopy fiber sequence.
\begin{equation}\label{ComCub2}
\begin{tikzcd}[back line/.style={densely dotted}, row sep=1em, column sep=1em]
F(D_2) \ar[dr, "\sim"] \ar{rr} \ar{dd} &  & F(D_1) \ar{dd} \ar[dr, "\sim"] & \\
& F(E_2) \ar{rr} \ar{dd} & & F(E_1) \\
F(C_D) \ar{rr} \ar[dr, "\sim"] & & F(D_0) \ar[dr, "\sim"]  & \\
& F(C_E) \ar{rr} & & F(E_0)\ar[crossing over, leftarrow]{uu} \\
\end{tikzcd}
\end{equation}
Notice that we used the assumption that the homotopy fiber sequences $a$ and $b$ are objectwise fibrant only in ($\star_1$) and $(\star_2)\,.$ Hence this assumption is not necessary if $F$ preserves all weak equivalences.
\end{proof}

\section{Puppe's long homotopy fiber sequence}\label{SectionPuppe}

The results of this section are based on a suitable notion of loop space functor.\medskip

It is well known that the {\bf path space} fibration of a pointed topological space $(X,x_0)$ is the fibration $$\zp_X:\op{Path}_0X\to X$$ whose total space $$\op{Path}_0X:=\{\za\in C^0([0,1],X):\za(0)=x_0\}$$ is the space of paths of $X$ with starting point $x_0\,$ and whose projection $\zp_X$ maps every path $\za$ to its end point $\za(1)\,$. The fiber $\op{Path}_0X\times_X0$ of $\zp_X$ over $x_0$ is the loop space of $X$ at $x_0\,.$\medskip

A similar concept exists in every pointed model category $\tt M\,.$ Indeed, any functorial factorization into a weak equivalence followed by a fibration
$$
\begin{tikzcd}[back line/.style={densely dotted}, row sep=1em, column sep=1em]
0 \ar[r]\ar[d,"\sim"]& 0\ar[d,"\sim"]\\
\op{Path}_0X\ar[r]\ar[d,two heads] & \op{Path}_0Y\ar[d, two heads]\\
X\ar[r] & Y
\end{tikzcd}
$$
leads to an endofunctor $\op{Path}_0:\tt M\to M$ and a natural transformation $\op{Path}_0\Rightarrow\id_{\tt M}\,.$ However, we prefer to work with a weaker notion of based path space functor $\op{Path}_0\,:$

\begin{defi}\label{BPSF}
A {\bf based path space functor} in a pointed model category $\tt M$ is an endofunctor $$\op{Path}_0:\tt M\to M$$ together with a natural transformation $$\zp:\op{Path}_0\Rightarrow\id_{\tt M}$$ whose components $\zp_X:\op{Path}_0X\twoheadrightarrow X$ at all \emph{fibrant} $X\in\tt M$ are fibrations with acyclic domain $\op{Path}_0X\stackrel{\sim}{\leftarrow}0\,.$\medskip
\end{defi}

\begin{rem}
\emph{An important special case occurs when the pointed model category under consideration is right proper and the condition that the components of the natural transformation $\zp$ are fibrations with acyclic domain is fulfilled for each object $X\,$, whether fibrant or not (just as when the based path space functor is induced by a functional factorization). In the following we refer to this case as the {\bf strongly proper case}.}
\end{rem}

Let $\tt I$ be the category $1\to 2\leftarrow 0\,.$ The functor category $\tt Fun(I,M)$ is then the category of cospans of $\tt M\,.$ To every based path space functor $\op{Path}_0$ we can associate the functor $$\cC_{\op{Path}_0}:\tt M\to Fun(I,M)\;$$ (if there is no possibility of confusion, we simply write $\cC$) that is defined on objects $X\in\tt M$ by the cospan
$$\cC X : \;\;0\to X\leftarrow\op{Path}_0X$$
and on morphisms $f:X\to Y$ of $\tt M$ by the commutative diagram
\vspace{0.1cm}
$$
\begin{tikzcd}[back line/.style={densely dotted}, row sep=1.5em, column sep=1em]
\cC X\ar[d,swap,"\cC f"]\ar[d,"\;\;\;:"]& &0\ar[r]\ar[d]& X\ar[d] & \op{Path}_0X \ar[l]\ar[d]\\
\cC Y & & 0\ar[r] & Y & \op{Path}_0Y\ar[l]
\end{tikzcd}
$$

\begin{defi}
In a pointed model category $\tt M\,$, the {\bf loop space functor} associated to a based path space functor $\op{Path}_0$ is the composite $$\zW_{\op{Path}_0}:=\op{Lim}\circ\;\cC_{\op{Path}_0}:\tt M\to M$$ (or just $\zW$ if no confusion is possible) of the cospan functor $\cC_{\op{Path}_0}:\tt M\to Fun(I,M)$ and the limit functor $\op{Lim}:\tt Fun(I,M)\to M\,.$ In particular the loop space of $X\in\tt M$ is the object $$\zW X = \op{Path}_0X\times_X0\in\tt M\;.$$
\end{defi}

\begin{rem}\label{FPP}
\emph{From here on we work in a fixed pointed model category $\tt M\,$, which is equipped with a fixed based path space functor $\op{Path}_0\,$, and consider the associated loop space functor $\zW\,.$}
\end{rem}

\begin{theo}\label{Preservation}
The loop space functor preserves all fibrant objects, the weak equivalences between fibrant objects and the objectwise fibrant homotopy fiber sequences. In the strongly proper case, the loop space of every object is fibrant and the loop space functor preserves all weak equivalences and all homotopy fiber sequences.
\end{theo}

\begin{proof}
We start with a few observations. In this proof we have $0\stackrel{\sim}{\to}\op{Path_0}X\twoheadrightarrow X$ for every object $X\in\tt M\,$ that we consider, since either this object is fibrant or we work in the strongly proper case. For every $\tt M$--morphism $f:X\to Y$, we get a commutative diagram
\be
\begin{tikzcd}[back line/.style={densely dotted}, row sep=1em, column sep=1em]
&&&0\ar[dd,"\sim"]\\
&&0\ar[dd,very near start,"\sim"]\ar[ru,"\sim"]&\\
&\zW Y\ar[dd]\ar[rr]&&\op{Path}_0Y\ar[dd,two heads]\\
\zW X\ar[rr]\ar[dd]\ar[ru,dashed]&&\op{Path}_0X\ar[dd,two heads]\ar[ru]&\\
&0\ar[rr]&&Y\\
0\ar[rr]\ar[ru,near end,"\sim"]&&X\ar[ru,near start,"f"]&\\
\end{tikzcd}
\ee
It contains three types of commutative squares, which will appear several times below. We refer to squares similar to the commutative right lower square as squares of the type $P\,,$ to squares similar to the commutative front square as squares of the type $L\,,$ and to squares similar to the commutative upper square as squares of the type $U\,.$ In every type $P$ square the arrow between the path spaces is a weak equivalence. Further, it follows from Corollary \ref{CorLur} that every type $L$ square is a model square. Finally, in every type $U$ square that is induced by a weak equivalence $f\,,$ the universal arrow is a weak equivalence because of Theorem \ref{IndHoPullMod} and Theorem \ref{RightProper}.\medskip

From the last observation and the closedness of fibrations under pullbacks it follows that the statements about weak equivalences and fibrant objects in Theorem \ref{Preservation} are true.\medskip

Let now $A\stackrel{f}{\to}\cA\stackrel{g}{\to} \mfa$ be a homotopy fiber sequence. If we are not in the strongly proper case, we assume that it is objectwise fibrant. If we factor $g$ into a weak equivalence followed by a fibration, we get the following commutative diagram
\be\label{C1}
\begin{tikzcd}[back line/.style={densely dotted}, row sep=1em, column sep=1em]
A\ar[rr]\ar[dd]&&\cA\ar[dd,very near start,swap,"g"]\ar[rd,"\sim"]&\\
&K:=\ker\bar g\ar[rr]\ar[dd, two heads]&&\bar\cA\ar[dd,two heads,"\bar g"]\\
C_A\ar[rr]&&\mfa\ar[rd,equal]&\\
&0\ar[rr]&&\mfa
\end{tikzcd}
\ee
From Corollary \ref{CorLur} it follows that the front square is a homotopy fiber sequence. Since $\zW$ is an endofunctor that preserves weak equivalences between fibrant objects or, in the strongly proper case, all weak equivalences, and since $\zW(0)=\op{Path}_0(0)$ is acyclic, we can apply Corollary \ref{PresHoFibSeq} to the images under $\zW$ of the homotopy fiber sequences $a:\;A\to \cA\to \mfa$ and $k:\;K\to \bar\cA\to\mfa\,:$ to show that $\zW(a)$ is a homotopy fiber sequence, it suffices to prove that $\zW(k)$ ($\zW$ applied to the front square of \eqref{C1}) is a homotopy fiber sequence.\medskip

Because of Proposition \ref{ComPara}, it is even enough to build a model square that is weakly equivalent to the commutative square $\zW(k)\,.$ We get this model square by constructing the following commutative diagram step by step. The diagram has an upper, a lower, a left, a right, a front, a back and two middle parts, the parallel (to the front) middle part and the orthogonal one. In the following description, the first adjective refers always to the part and the second to the square we are looking at in that part. For instance, the right front square is the front square of the right part, i.e., the square $\op{Path}_0\bar\cA\;\bar\cA\;\op{Path}_0\mfa\;\mfa\,,$ whereas the front right square is the right square of the front part, i.e., the square $K\;\bar\cA\;0\;\mfa\,.$
\be\label{C2}
\begin{tikzcd}[back line/.style={densely dotted}, row sep=1em, column sep=1em]
\zW K\ar[rrr]\ar[rd,dashed]\ar[dd,equal]&&&\op{Path}_0K\ar[dd]\ar[rd,dashed]\ar[drrr]&&&&\\
&\zW\bar\cA\ar[rrrrr, bend left=8]\ar[rrr,dashed]\ar[rdd]\ar[ddd,equal]&&&\cK\ar[rr]\ar[rdd,dashed]\ar[rdddd, two heads]\ar[ddd,dashed]&&\op{Path}_0\bar\cA\ar[rdd, two heads]\ar[rdddd, two heads]\ar[ddd]&\\
\zW K\ar[rrr,dashed]\ar[rdd,dashed]&&&\zW(0)=\op{Path}_00\ar[rrrdd]\ar[rdd,dashed]&&&&\\
&&0\ar[rrr]\ar[dd]&&&K\ar[rr]\ar[dd,two heads]&&\bar\cA\ar[dd, two heads]\\
&\zW\bar\cA\ar[rrr,dashed]\ar[rd]&&&\zW\mfa\ar[rr]\ar[rd]&&\op{Path}_0\mfa\ar[rd, two heads]&\\
&&0\ar[rrr]&&&0\ar[rr]&&\mfa
\end{tikzcd}
\ee
We start from the front right square (which is the front homotopy fiber sequence of \eqref{C1}). The right front square is a type $P$ square. Let now $\cK$ be the kernel of the morphism $\op{Path}_0\bar\cA\twoheadrightarrow\bar\cA\twoheadrightarrow\mfa\;:$ Corollary \ref{CorLur} implies that the resulting square (the `diagonal' square) is a homotopy fiber sequence. The universal arrow $\cK\dashrightarrow K$ makes the upper right square and the middle upper triangle commutative. The pasting law for model squares now implies that the upper right square is a model square. The lower right square is a type $L$ square. The universal arrow $\cK\dashrightarrow\zW\mfa$ renders the middle lower triangle, the middle front square and the middle right square commutative, so that the right cube is fully commutative. Moreover, Theorems \ref{IndHoPullMod} and \ref{RightProper} imply that this universal arrow is a weak equivalence.\medskip

We now describe the front left cube. Its front, left and lower squares obviously commute (and so does its right square). The total upper front square is a type $L$ square and therefore a model square. Since $\cK$ is a kernel, the universal arrow $\zW\bar\cA\dashrightarrow\cK$ makes the parallel bent triangle commutative. As $K$ is a kernel, the upper front left square commutes and because of the pasting law it is a model square. The parallel middle part of \ref{C2} is a type $U$ square. As $\zW\mfa$ is a pullback, the middle left square commutes.\medskip

It still remains to explain the back cube and the back $3D$ wedge. We start looking at the union of this cube and wedge. The lower (left) square is the commutative square $\zW(k)$ and the lower triangle can be viewed as a type $U$ square. The left square obviously commutes and the right square is the image under $\op{Path}_0$ of a commutative square and is therefore itself commutative. The back square can be interpreted as a type $U$ square and the total upper square is also a type $U$ square. To understand the middle square (the back square of the orthogonal middle part) and the resulting subdivision of the union, we have to look again at the overall diagram \eqref{C2}. As $\cK$ is a kernel, there is a universal arrow $\op{Path}_0K\dashrightarrow\cK$ that makes the upper triangle commutative. The total upper left square is a type $L$ and a model square and since $\cK$ is a kernel, the upper back square commutes and is, in view of the pasting law, a model square.\medskip

It suffices now to show that this model square is weakly equivalent to $\zW(k)\,.$ We know already that $\op{Path}_0K\to\op{Path}_00$ and $\cK\dashrightarrow\zW\mfa$ are weak equivalences and that the lower back square commutes. Hence all we have to do is to prove that the middle back square is commutative. However, this is the case as $\zW\mfa$ is a pullback.
\end{proof}

Let us recall that if $f:X\to Y$ is a base point preserving continuous map between pointed topological spaces $(X,x_0)$ and $(Y,y_0)\,,$ its homotopy fiber or {\bf homotopy kernel} $$K_f:=\op{Path}_0Y\times_YX=\{(\za,x)\in\op{Path}_0Y\times X:\za(1)=f(x)\}$$ $$=\{(\za,x)\in C^0([0,1],Y)\times X:\za(0)=y_0,\za(1)=f(x)\}\;$$ fits into the `homotopy fiber sequence' $K_f\to X\to Y\,.$ The latter can be extended to a long sequence. More precisely, the loop space $$\zW Y:=\{\za\in C^0([0,1],Y):\za(0)=\za(1)=y_0=f(x_0)\}$$ injects into the homotopy fiber $K_f$ thus providing a {\bf connecting morphism} $\zd:\zW Y\to K_f\,.$ The extending long sequence mentioned is then the sequence $$\cdots\to \zW^2Y\to\zW(K_f)\to\zW X\to\zW Y\to K_f\to X\to Y\;$$ which is referred to as {\bf Puppe's sequence}.\medskip

We will generalize Puppe's sequence to our context. We start with the following definition (see also Remark \ref{FPP}):

\begin{defi}\label{HoKerCon}
Let $f:X\to Y$ be an $\tt M$--morphism between \emph{fibrant objects}. We refer to the pullback $$K_f:=\op{Path}_0Y\times_YX$$
as the {\bf homotopy kernel} of $f\,,$ to the universal arrow $$\zd_f:\zW Y\dashrightarrow K_f$$ as the {\bf connecting morphism} associated to $f$
\be\label{HoKerConPic}
\begin{tikzcd}[back line/.style={densely dotted}, row sep=1em, column sep=1em]
\zW Y\ar[rrr]\ar[dd,two heads]\ar[rd,dashed,very near end,"\zd_f"]&&&\op{Path}_0Y\ar[dd,two heads]\ar[rd,equal]&&\\
&K_f\ar[rrr,near start,"p_f"]\ar[ddd,two heads,near end,"\zp_f"]&&&\op{Path}_0Y\ar[ddd,two heads]&\\
0\ar[rrr]\ar[rdd]&&&Y\ar[rdd,equal]&&\\
&&&&&0\ar[dd]\ar[luu,near start,"\sim"]\\
&X\ar[rd,equal]\ar[rrr,"f"]&&&Y\ar[rd,equal]&\\
&&X\ar[rrr,"f"]&&&Y\
\end{tikzcd}
\ee
and to the sequence
\be\label{PuppeSeq}
\cP_f:\;\;\cdots\longrightarrow \zW^2 Y\stackrel{\zW(\zd_f)}{\longrightarrow}\zW(K_f)\stackrel{\zW(\zp_f)}{\longrightarrow}\zW X\stackrel{\zW f}{\longrightarrow}\zW Y\stackrel{\zd_f}{\longrightarrow}K_f\stackrel{\zp_f}{\longrightarrow} X\stackrel{f}{\longrightarrow}Y\;
\ee
as {\bf Puppe's sequence} of $f\,.$
\end{defi}

Definition \ref{HoKerCon} makes also sense if $X$ and $Y$ are not fibrant. However, if they are, the homotopy fiber or homotopy kernel $K_f$ of $f$ is a fibrant object and it is in view of Theorem \ref{IndHoPull} a canonical representative of the homotopy pullback $0\times_Y^hX\,.$ This justifies the assumption that $X$ and $Y$ are fibrant. In the strongly proper case, the homotopy fiber $K_f$ of $f$ is isomorphic in the homotopy category to the homotopy fiber $\op{HFib}(f)$ of $f$ of \cite{Hir}. Therefore \emph{we use the terminology of Definition \ref{HoKerCon} also in the strongly proper case}.

\begin{prop}\label{Puppe}
Let $\,\tt M^\s_{\op{f}}$ be the full subcategory of $\,\tt M^\s$ consisting of all the morphisms between fibrant objects of $\,\tt M$ and let $\tt M^\s_\dagger$ be the category $\tt M^\s_{\op{f}}\,,$ except in the strongly proper case where it is the full category $\tt M^\s\,.$ There is an extension functor $$\cP:{\tt M}^\s_\dagger\ni f\mapsto \cP_f\in\ell({\tt M})\;,$$ whose values $\cP_f$ on objects $f\in\tt M^\s_{\op{f}}$ are objectwise fibrant long homotopy fiber sequences and to which we refer as {\bf Puppe's extension functor}.
\end{prop}

\begin{proof}
Let $f\in{\tt M}^\s_\dagger$ be an $\tt M$--morphism with source $X$ and target $Y\,.$\medskip

We will describe the following commutative diagram using the terminology introduced at the beginning of the proof of Theorem \ref{Preservation}.
\be\label{PuppeProofPic}
\begin{tikzcd}[back line/.style={densely dotted}, row sep=1em, column sep=1em]
&&&\\
\zW X\ar[dd,bend right,two heads]\ar[r]\ar[d,dashed]&\op{Path}_0X\ar[r,equal]\ar[d,dashed]\ar[dd,bend right,two heads]&\op{Path}_0X\ar[d]&0\ar[l,swap,"\sim"]\\
\zW Y\ar[r,dashed]\ar[d,two heads]\ar[rr,bend right,crossing over]&K_f\ar[r]\ar[d,two heads]&\op{Path}_0Y\ar[d,two heads]&0\ar[l,swap,"\sim"]\\
0\ar[r]\ar[rr,bend right,crossing over]&X\ar[r]&Y
\end{tikzcd}
\ee
The three lower squares of \eqref{PuppeProofPic} are nothing more that the commutative cube of \eqref{HoKerConPic} in Definition \ref{HoKerCon}. The \emph{lower right square} is a model square and even a homotopy fiber sequence because of Corollary \ref{CorLur} and the total lower square is a type $L$ square and therefore also a model square. From the pasting law for model squares it now follows that the \emph{lower left square} is a homotopy fiber sequence. The total left square is also a model square as it is of the type $L\,.$ The total upper square is of the type $U$ and the total right square is of the type $P\,.$ The arrow $\op{Path}_0X\dashrightarrow K_f$ is the universal morphism and makes the upper right square and the vertical middle bent triangle commutative. The upper left square commutes because of the uniqueness of the universal morphism from $\zW X$ to $K_f\,.$ Finally the pasting law implies that the \emph{upper left square} is a homotopy fiber sequence.\medskip

Since $\zW$ preserves objectwise fibrant homotopy fiber sequences and even all homotopy fiber sequences if we work in a strongly proper environment, the images $$\zW(K_f)\to\zW X\to \zW Y\,,\quad \zW^2 Y\to \zW(K_f)\to \zW X\,,\quad \zW^2X\to\zW^2Y\to\zW(K_f)\,,\quad \cdots\;$$ are all homotopy fiber sequences that are objectwise fibrant if $f\in{\tt M}^\s_{\op{f}}\,.$\medskip

Let now $f,g\in{\tt M}^\s_\dagger$ be $\tt M$--morphisms with source $X,U$ and target $Y,V\,,$ respectively, and let $\psi=(\psi_0,\psi_1):f\to g$ be an ${\tt M}^\s$--morphism between them. The next diagram, in which dashed arrows represent universal morphisms as usual, defines $\cP_\psi:\cP_f\to\cP_g$ and shows that $\cP_\psi$ is an $\ell({\tt M})$--morphism, i.e., a sequence of commutative cubes.
\be\label{PuppeProofPic2}
\begin{tikzcd}[back line/.style={densely dotted}, row sep=1em, column sep=1em]
&&&\zW X\ar[r]\ar[d,dashed]\ar[ddlll,dashed]&\op{Path}_0X\ar[dr]\ar[d,dashed]\ar[ddlll,crossing over]&\\
&&&\zW Y\ar[dddlll,bend left=17.5,dashed]\ar[r,dashed]\ar[dd]&K_f\ar[r]\ar[dd]\ar[dddlll,dashed]&\op{Path}_0Y\ar[dd]\\
\zW U\ar[r]\ar[dd,dashed]&\op{Path}_0U\ar[ddr,crossing over]\ar[dd,crossing over,dashed]&&&&\\
&&&0\ar[r]\ar[ddlll, bend right=5]&X\ar[r,"f"]\ar[ddlll,near start,"\psi_0"]&Y\ar[ddlll,"\psi_1"]\\
\zW V\ar[r,dashed]\ar[d,crossing over]&K_g\ar[r,crossing over]\ar[d,crossing over]&\op{Path}_0V\ar[d,crossing over]\ar[from=uuurrr,bend left=7.5, crossing over]&&&\\
0\ar[r]&U\ar[r,near end,"g"]&V&&&
\end{tikzcd}
\ee
To convince ourselves of this claim, we can first note that all arrows from the back part of Diagram \eqref{PuppeProofPic2} to the front part are easy to understand. The commutation of the three cubes of \eqref{PuppeProofPic2} is obvious because of the uniqueness of universal arrows. As $\zW$ is a functor, we get the desired sequence of commutative cubes. Moreover, the assignment $\cP:\psi\mapsto \cP_\psi$ clearly respects compositions and identities, which completes the proof.
\end{proof}

\begin{prop}
The extension $\cP_f\in\ell({\tt M})$ of a morphism $f\in{\tt M}^\s_\dagger$ to a long homotopy fiber sequence is unique up to a canonical isomorphism of $\,\tt Ho(\ell(M))$, i.e., a zigzag of weak equivalences of $\ell({\tt M})\,.$
\end{prop}

\begin{proof}
The statement is a consequence of Proposition \ref{Explicit}. Indeed, if $b_\bullet\in\ell({\tt M})$ is another extension of $f$ than $a_\bullet:=\cP_f\in\ell({\tt M})\,,$ there is a canonical isomorphism in the homotopy category of $\ell({\tt M})$ between $a_\bullet$ and $b_\bullet$ that extends the class $\xi:=[\id:f\to f]\,.$
\end{proof}

In Algebraic Topology the long exact sequence of homotopy groups of a fibration is a consequence of Puppe's sequence of the fibration: to get the long exact sequence it suffices to apply to Puppe's sequence the covariant Hom functor in the homotopy category associated to the 0-sphere. A similar result exists in our context of a pointed model category $\,\tt M$ equipped with a based path space functor $\op{Path}_0$ and the corresponding loop space functor $\zW\,.$ More precisely, if $$[A,-]:=\op{Hom}_{\tt Ho(M)}(A,-)$$ is the covariant Hom functor in the homotopy category associated to an object $A\in{\tt M}$ and if $\zg$ is the localization functor of the category $\tt M\,,$ we have the following

\begin{prop}\label{LESH}
Let $f\in{\tt M}^\s_\dagger$ and let $A\in{\tt M}\,.$ The sequence
$$
\cdots\longrightarrow [A,\zW^2 Y]\stackrel{[A,\zg(\zW(\zd_f))]}{\longrightarrow}[A,\zW(K_f)]\stackrel{[A,\zg(\zW(\zp_f))]}{\longrightarrow}[A,\zW X]\stackrel{[A,\zg(\zW f)]}{\longrightarrow}
$$
\be\label{PuppeSeq}
[A,\zW Y]\stackrel{[A,\zg(\zd_f)]}{\longrightarrow}[A,K_f]\stackrel{[A,\zg(\zp_f)]}{\longrightarrow} [A,X]\stackrel{[A,\zg f]}{\longrightarrow}[A,Y]\;
\ee
is a long exact sequence of $\op{Hom}_{\tt Ho(M)}$--sets.
\end{prop}

We will show later that when we apply this result to chain complexes of modules, we get the usual long exact sequence of homology modules (since our theory is an abstraction of aspects of pointed topological spaces, it is not surprising that in this case we obtain the usual long exact sequence of homotopy groups). Further we will explain what an exact sequence of $\op{Hom}_{\tt Ho(M)}$--sets is after giving the next proposition of which Proposition \ref{LESH} is obviously a corollary.

\begin{theo}\label{HFSES}
Let $X\stackrel{f}{\longrightarrow}\cX\stackrel{g}{\longrightarrow}\mathfrak{X}$ be a homotopy fiber sequence in $\,\tt M$ and let $A$ be an object of $\,\tt M\,.$ Then the sequence
\be\label{ESH}
[A,X]\stackrel{[A,\zg f]}{\longrightarrow}[A,\cX]\stackrel{[A,\zg g]}{\longrightarrow}[A,\mathfrak{X}]
\ee
is an exact sequence of $\,\op{Hom}_{\tt Ho(M)}$--sets.
\end{theo}

Now as to the meaning of exactness, let $f:X\to\cX$ be an $\tt M$--morphism and let $A$ be an object of $\tt M\,.$ Then $\zg f:X\to\cX$ is the $\tt Ho(M)$--morphism $$\zg f=[\tilde F\tilde C f]_{\simeq}\in[X,\cX]\;,$$ where $\tilde F$ is a local fibrant C-replacement, where $\tilde C$ is a local cofibrant F-replacement and where $[-]_{\simeq}$ denotes the homotopy class. Further $[A,\zg f]:[A,X]\to [A,\cX]$ is the set-theoretical morphism $$[A,\zg f]=[\tilde F\tilde C f]_{\simeq}\circ -\;\;.$$ We define $\op{im}[A,\zg f]$ as the set of images and $\ker[A,\zg f]$ as the set of elements that are sent to $0_{[A,\cX]}\,,$ where $0_{[A,\cX]}$ is the image under $\zg$ of the zero morphism $0_{A\cX}:A\stackrel{t_A}{\to}0\stackrel{i_\cX}{\to}\cX\;:$ $$0_{[A,\cX]}=\zg(0_{A\cX})=\zg(i_\cX)\circ\zg(t_A)=[0\stackrel{\tilde F\tilde Ci_\cX}{\longrightarrow} \tilde F\tilde C \cX]_\simeq\circ[\tilde F\tilde C A\stackrel{\tilde F\tilde C t_A}{\longrightarrow}0]_\simeq=[\tilde F\tilde C A\to 0\to \tilde F\tilde C\cX]_\simeq\;.$$ The sequence \eqref{ESH} is exact now means that $$\op{im}[A,\zg f]=\ker[A,\zg g]\;.$$

\begin{proof}
Let $\tilde F\cX\stackrel{\sim}{\to}\bar F\cX\twoheadrightarrow \tilde F\mathfrak{X}$ be a factorization of a lifting $\tilde F g:\tilde F \cX\to \tilde F\mathfrak{X}$ into a weak equivalence followed by a fibration and let $K\to \bar F\cX\twoheadrightarrow \tilde F\mathfrak{X}$ be the kernel of this fibration. Since $K=\bar F\cX\times_{\tilde F\mathfrak{X}}0\,,$ we have two homotopy fiber sequences
\be\label{}
\begin{tikzcd}[back line/.style={densely dotted}, row sep=1em, column sep=1em]
x:\;&X\ar[r,"f"]&\cX\ar[r,"g"]\ar[d,"\sim"]&\mathfrak{X}\ar[d,"\sim"]\\
k:\;&K\ar[r,"\zk"]&\bar F\cX\ar[r,two heads,"h"]&\tilde F\mathfrak{X}
\end{tikzcd}
\ee
whose restrictions $x_1$ and $k_1$ are related by a weak equivalence $w=(w_0,w_1)$ in ${\tt M}^\s$ and an isomorphism $\xi:=\zg_{{\tt M}^\s}(w)$ in ${\tt Ho}({\tt M}^\s)\,.$ For later use, we note that all of the nodes in the lower sequence are fibrant. \medskip

In view of Proposition \ref{Explicit}, the homotopy fiber sequences $x$ and $k$ are related by a zigzag \eqref{SeqWeqHFS} of weak equivalences of homotopy fiber sequences, i.e., by a zigzag of commutative cubes that are objectwise weak equivalences (hence the corresponding vertices of $x$ and $k$ are all related by a zigzag of weak equivalences of $\tt M$). If we apply $\zg$ and $[A,-]$ to these cubes, we get in particular the following commutative squares in the category $\tt Set$ of sets
\begin{equation}\label{CSq1}
\begin{tikzpicture}
 \matrix (m) [matrix of math nodes, row sep=3em, column sep=3em]
   {\left[A,X\right]  & \left[A,\cX\right] & \left[A,\mathfrak{X}\right] \\
   \left[A,K\right] & \left[A,\bar F\cX\right] & \left[A,\tilde F\mathfrak{X}\right]  \\ };
 \path[->]
 (m-1-1) edge node[auto] {\small{$[A,\zg f]$}} (m-1-2)
 (m-1-2) edge node[auto] {\small{$[A,\zg g]$}} (m-1-3)
 (m-1-1) edge node[auto] {\small{$b_1$}} (m-2-1)
 (m-1-2) edge node[auto] {\small{$b_2$}} (m-2-2)
 (m-1-3) edge node[auto] {\small{$b_3$}} (m-2-3)
 (m-2-1) edge node[auto] {\small{$[A,\zg\zk]$}} (m-2-2)
 (m-2-2) edge node[auto] {\small{$[A,\zg h]$}} (m-2-3);
\end{tikzpicture}
\end{equation}
whose vertical arrows are bijections. More precisely, as said above, the corresponding vertices $X$ and $K$, for example, are related by a zigzag of weak equivalences $\stackrel{\zw}{\longrightarrow}$ and weak equivalences $\stackrel{\zvw}{\longleftarrow}\,.$ When we apply $\zg$ we get a zigzag of isomorphisms $$\zg(\stackrel{\zw}{\longrightarrow})=[\tilde F\tilde C\stackrel{\zw}{\longrightarrow}]_\simeq\quad\text{and}\quad\zg(\stackrel{\zvw}{\longleftarrow})=[\tilde F\tilde C\stackrel{\zvw}{\longleftarrow}]_\simeq\;,$$
so that the composite $i_1$ of the $$[\tilde F\tilde C\stackrel{\zw}{\longrightarrow}]_\simeq\quad\text{and}\quad[\tilde F\tilde C\stackrel{\zvw}{\longleftarrow}]^{-1}_\simeq$$ is an isomorphism in $\tt Ho(M)$ from $X$ to $K\,.$ The bijection $b_1$ is now $b_1=[A,i_1]\,.$ As a map between fibrant-cofibrant objects is a weak equivalence if and only if it is a homotopy equivalence, the $\tilde F\tilde C\stackrel{\zvw}{\longleftarrow}$ are invertible up to homotopy by weak equivalences $\stackrel{W}{\longrightarrow}$ and $$[\tilde F\tilde C\stackrel{\zvw}{\longleftarrow}]^{-1}_\simeq=[\stackrel{W}{\longrightarrow}]_\simeq\;.$$ Therefore $i_1$ is a composite of homotopy classes of weak equivalences, so it is the homotopy class of a weak equivalence: $$i_1=[\tilde F\tilde C X\stackrel{\sim}{\to}\tilde F\tilde C K]_\simeq\;.$$ It follows that $$b_1 0_{[A,X]}=[\tilde F\tilde C X\stackrel{\sim}{\to}\tilde F\tilde C K]_\simeq\circ[\tilde F\tilde C A\to 0\to \tilde F\tilde C X]_\simeq=[\tilde F\tilde C A\to 0\to\tilde F\tilde C K]_\simeq=0_{[A,K]}\;,$$ i.e., that the bijections $b_1,b_2$ and $b_3$ preserve the zero elements. \medskip

Moreover, for every object $A\in\tt M\,,$ there is a cofibrant object $\cA\in{\tt M}_{\op{c}}$ and a weak equivalence $\zw:\cA\stackrel{\sim}{\to}A$ of $\tt M\,.$ If we apply $\zg$ we get an isomorphism $i:\cA\stackrel{\cong}{\to}A$ of $\tt Ho(M)$ given by $$i=\zg\zw=[\tilde F\tilde C\zw]_\simeq=[\tilde F\tilde C\cA\stackrel{\sim}{\to}\tilde F\tilde C A]_\simeq\;,$$ and for every object $B\in\tt M$ we get a bijection $$b_B:=[i,B]=-\circ i:[A,B]\stackrel{\cong}{\to}[\cA,B]$$ that sends $0_{[A,B]}$ to
$$
b_B0_{[A,B]}=[\tilde F\tilde C A\to 0\to\tilde F\tilde C B]_\simeq\circ[\tilde F\tilde C \cA\stackrel{\sim}{\to}\tilde F\tilde C A]_\simeq=0_{[\cA, B]}\;.
$$
Hence we have the following commutative $\tt Set$--squares
\begin{equation}\label{CSq2}
\begin{tikzpicture}
 \matrix (m) [matrix of math nodes, row sep=3em, column sep=3em]
   {\left[A,K\right]  & \left[A,\bar F\cX\right] & \left[A,\tilde F\mathfrak{X}\right] \\
   \left[\cA,K\right] & \left[\cA,\bar F\cX\right] & \left[\cA,\tilde F\mathfrak{X}\right] \\ };
 \path[->]
 (m-1-1) edge node[auto] {\small{$[A,\zg \zk]$}} (m-1-2)
 (m-1-2) edge node[auto] {\small{$[A,\zg h]$}} (m-1-3)
 (m-1-1) edge node[auto] {\small{$b_K$}} (m-2-1)
 (m-1-2) edge node[auto] {\small{$b_{\bar F \cX}$}} (m-2-2)
 (m-1-3) edge node[auto] {\small{$b_{\tilde F\mathfrak{X}}$}} (m-2-3)
 (m-2-1) edge node[auto] {\small{$[\cA,\zg\zk]$}} (m-2-2)
 (m-2-2) edge node[auto] {\small{$[\cA,\zg h]$}} (m-2-3);
\end{tikzpicture}
\end{equation}
As for the commutativity of these squares, note that if for instance $[\psi]_\simeq\,\in[A,K]\,,$ then $$[\cA,\zg\zk]\big(b_K[\psi]_\simeq\big)=\zg\zk\circ\big([\psi]_\simeq\circ i\big)=\big(\zg\zk\circ[\psi]_\simeq\big)\circ i=b_{\bar F\cX}\big([A,\zg\zk][\psi]_\simeq\big)\;.$$ If we combine \eqref{CSq1} and \eqref{CSq2}, we get commutative squares
\begin{equation}\label{CSq3}
\begin{tikzpicture}
 \matrix (m) [matrix of math nodes, row sep=3em, column sep=3em]
   {\left[A,X\right]  & \left[A,\cX\right] & \left[A,\mathfrak{X}\right] \\
   \left[\cA,K\right] & \left[\cA,\bar F\cX\right] & \left[\cA,\tilde F\mathfrak{X}\right] \\ };
 \path[->]
 (m-1-1) edge node[auto] {\small{$[A,\zg f]$}} (m-1-2)
 (m-1-2) edge node[auto] {\small{$[A,\zg g]$}} (m-1-3)
 (m-1-1) edge node[auto] {\small{$\flat_1$}} (m-2-1)
 (m-1-2) edge node[auto] {\small{$\flat_2$}} (m-2-2)
 (m-1-3) edge node[auto] {\small{$\flat_3$}} (m-2-3)
 (m-2-1) edge node[auto] {\small{$[\cA,\zg\zk]$}} (m-2-2)
 (m-2-2) edge node[auto] {\small{$[\cA,\zg h]$}} (m-2-3);
\end{tikzpicture}
\end{equation}
whose vertical arrows are bijections which respect the zero elements of their source and target $\op{Hom}_{\tt Ho(M)}$--sets. It is straightforward to check that the exactness of the upper sequence of \eqref{CSq3} is equivalent to the exactness of the lower sequence. For instance, if $$\op{im}[\cA,\zg\zk]\subset\ker[\cA,\zg h]\;$$ and if $[\phi]_\simeq\,\in[A,X]\,,$ we have $$[A,\zg g]\big([A,\zg f][\phi]_\simeq\big)=\flat_3^{-1}\left([\cA,\zg h]\big([\cA,\zg \zk](\flat_1[\phi]_\simeq)\big)\right)=\flat_3^{-1}0_{[\cA, \tilde F\mathfrak{X}]}=0_{[A,\mathfrak{X}]}\;.$$ Hence, it suffices to show that the lower sequence is exact.\medskip

Since $\cA\in\tt M_{\op{c}}$ and $K,\bar F\cX,\tilde F\mathfrak{X}\in\tt M_{\op{f}}\,,$ the description of the $\op{Hom}_{\tt Ho(M)}$--sets in the lower sequence of \eqref{CSq3} can be simplified. Indeed, the map $$\zg:\op{Hom}_{\tt M}(\cA, K)\ni \mathfrak{f}\mapsto \zg\mathfrak{f}\in[\cA,K]$$ is surjective and induces a 1:1 correspondence \be\label{1:1CF}\tilde\zg:\op{Hom}_{\tt M}(\cA, K)/\simeq\;\ni [\mathfrak{f}]_\simeq\mapsto \zg\mathfrak{f}\in[\cA,K]\;,\ee so that \be\label{GaCF}\zg\mathfrak{f}=[\mathfrak{f}]_\simeq\;\ee if we identify the homotopy classes with the morphisms in the homotopy category \cite[Proposition 5.11]{DS}.\medskip

Hence every element of $\op{im}[\cA,\zg\zk]$ reads $[\cA,\zg\zk](\zg\mathfrak{f})$ and, since
$$
[\cA,\zg h]\big([\cA,\zg \zk](\zg\mathfrak{f})\big)=\zg(h\circ\zk\circ\mathfrak{f})=\zg(0_{K,\tilde{F}\mathfrak{X}}\circ\mathfrak{f})=\zg(0_{\cA,\tilde{F}\mathfrak{X}})=0_{[\cA,\tilde{F}\mathfrak{X}]}\;,
$$
we have
$$\op{im}[\cA,\zg\zk]\subset\ker[\cA,\zg h]\;.$$
Conversely, if $\zg\mathfrak{g}=[\mathfrak{g}]_\simeq\in[\cA,\bar F\cX]$ is an element of $\ker[\cA,\zg h]\,,$ we have $$[\cA,\zg h](\zg\mathfrak{g})=\zg(h\circ\mathfrak{g})=[h\circ\mathfrak{g}]_\simeq=0_{[\cA,\tilde F\mathfrak{X}]}=\zg(0_{\cA,\tilde F\mathfrak{X}})=[0_{\cA,\tilde F\mathfrak{X}}]_\simeq\;,$$ so that $h\circ\mathfrak{g}$ and $0:=0_{\cA,\tilde{F}\mathfrak{X}}$ are left homotopic, i.e., so that $(h\circ\mathfrak{g})\amalg 0:\cA\amalg\cA\to\tilde F\mathfrak{X}$ factors through a cylinder object $\op{Cyl}\cA$ of $\cA\,$. The cylinder object is a factorization $$\cA\stackrel{\zf_1,\zf_2}{\rightrightarrows}\cA\amalg\cA\stackrel{i}{\rightarrowtail}\op{Cyl}\cA\stackrel{w}{\stackrel{\sim}{\to}}\cA$$ of the fold map $\id_\cA\amalg\id_\cA:\cA\amalg\cA\to\cA$ into a cofibration $i$ and a weak equivalence $w\,,$ which means that $$w\circ i_1:=w\circ i\circ \zf_1=\id_\cA\quad\text{and}\quad w\circ i_2:=w\circ i\circ \zf_2=\id_\cA\;.$$ Since $\cA$ is cofibrant and cofibrations are closed under pushouts, the morphisms $\zf_1$ and $\zf_2$ are cofibrations and so are the morphisms $i_1$ and $i_2\,,$ which are obviously also weak equivalences. The factorization of $(h\circ\mathfrak{g})\amalg 0$ mentioned above now means that there is a morphism $H:\op{Cyl}\cA\to \tilde{F}\mathfrak{X}$ such that $H\circ i_1=h\circ\mathfrak{g}$ and $H\circ i_2=0\,,$ so that we have the commutative squares
\begin{equation}
\begin{tikzcd}
\cA\arrow[r,"i_1",tail]\arrow[r,tail,swap,"\sim"]\arrow[d,"\mathfrak{g}"]&\op{Cyl}\cA\arrow[dl,dashed,swap,"\ell"]\arrow[d,"H"]&\cA\arrow[l,tail,"\sim"]\arrow[l,tail,swap,"i_2"]\arrow[d]\\
\bar F\cX\arrow[r,two heads,"h"]&\tilde{F}\mathfrak{X}&\arrow[l]0
\end{tikzcd}
\end{equation}
The dashed arrow $\ell$ exists in view of the lifting axiom. Since $h\circ\ell\circ i_2=0\,,$ the morphism $\ell\circ i_2$ factors through the kernel $(K,\zk)$ of $h\,,$ which means that there is a morphism $\mathfrak{f}:\cA\to K$ such that $\ell\circ i_2=\zk\circ\mathfrak{f}\,.$ Hence we have the commutative diagram
\begin{equation}
\begin{tikzcd}
\cA\arrow[r,"i_1",tail]\arrow[r,tail,swap,"\sim"]\arrow[d,"\mathfrak{g}"]&\op{Cyl}\cA\arrow[d,dashed,"\ell"]&\cA\arrow[l,tail,"\sim"]\arrow[l,tail,swap,"i_2"]\arrow[d,"\mathfrak{f}"]\\
\bar F\cX\arrow[r,equal]&\bar{F}\cX&\arrow[l,"\zk"]K
\end{tikzcd}
\end{equation}
This means that $\ell$ is a homotopy between $\mathfrak{g}$ and $\zk\circ\mathfrak{f}\,:$ $$\zg\mathfrak{g}=[\mathfrak{g}]_\simeq = [\zk\circ\mathfrak{f}]_\simeq=\zg\zk\circ\zg\mathfrak{f}=[\cA,\zg\zk](\zg\mathfrak{f})\;,$$ i.e., $$\ker[\cA,\zg h]\subset\op{im}[\cA,\zg\zk]\;.$$
\end{proof}

We denote $\tt h(M_{\op{f}})$ the full subcategory of $\tt h(M)$ made of the objectwise fibrant homotopy fiber sequences of $\tt M\,.$ The category $\tt h(M_\dagger)$ is the category $\tt h(M_{\op{f}})$ except in the strongly proper case where it is the category $\tt h(M)\,.$ Proposition \ref{LESH} shows that we can associate a long exact sequence to every $g\in{\tt M}^\s_\dagger\,,$ i.e., to every $k\in\tt h(M_\dagger)$ of the type $k: K_g\stackrel{\zp_g}{\longrightarrow}\cX\stackrel{g}{\longrightarrow}\mathfrak{X}\,.$ It is also possible to associate a long exact sequence to an arbitrary $x\in\tt h(M_\dagger)\,:$

\begin{prop}\label{LESHGen}
Let $x: X\stackrel{f}{\longrightarrow}\cX\stackrel{g}{\longrightarrow}\mathfrak{X}$ be a homotopy fiber sequence $x\in\tt h(M_\dagger)$ and let $A\in\tt M\,.$ Then there is a connecting morphism $\zD\in[\zW\mathfrak{X},X]$ such that
\be\label{PuppeSeq}
\cdots\longrightarrow [A,\zW^2 \mathfrak{X}]\stackrel{[A,\mathbb{R}\zW(\zD)]}{\longrightarrow}[A,\zW X]\stackrel{[A,\zg(\zW f)]}{\longrightarrow}[A,\zW \cX]\stackrel{[A,\zg(\zW g)]}{\longrightarrow}
\ee
\be
[A,\zW \mathfrak{X}]\stackrel{[A,\zD]}{\longrightarrow}[A,X]\stackrel{[A,\zg f]}{\longrightarrow} [A,\cX]\stackrel{[A,\zg g]}{\longrightarrow}[A,\mathfrak{X}]\;
\ee
is a long exact sequence of $\op{Hom}_{\tt Ho(M)}$--sets.
\end{prop}

\begin{rem}
\emph{If we do not work in a strongly proper environment and $x\in\tt h(M)$ is not necessarily objectwise fibrant, we can apply a fibrant replacement functor $R$ to $x$ and associate a long exact sequence to $Rx\in\tt h(M_{\op{f}})$ (see Proposition \ref{ComPara}).}

\end{rem}

\begin{proof}
Since $x\in\tt h(M)$ induces the above $k\in\tt h(M)$ and the restrictions $k_1=x_1=g\in\tt M^\s$ are related by the isomorphism $$\xi:=\id_{\tt Ho(M^\s)}\!g=\zG(\id_{\tt M^\s}\!g)=\zG(\id_{\tt M}{\cX},\id_{\tt M}\mathfrak{X})\in\op{Hom}_{\tt Ho(M^\s)}(k_1,x_1)\;,$$ where $\zG:=\zg_{\tt M^\s}\,,$ Proposition \ref{Explicit} implies that there is a canonical isomorphism $\Xi\in\op{Hom}_{\tt Ho(h(M))}($ $k,x)$ that extends $\xi$ and Equation \eqref{SeqWeqHFS} gives the $\tt h(M)$--zigzag
$$
\begin{tikzcd}
k\arrow[r,"\zw^k"]\arrow[r,swap,"\sim"]& \ff^k & f^{g}\ar[l,swap,"\zu^k"]\ar[l,"\sim"]\ar[r,"E(\id)"]\ar[r,swap,"\sim"]& f^{g}\ar[r,"\zu^x"]\ar[r,swap,"\sim"]&\ff^x&x\ar[l,swap,"\zw^x"]\ar[l,"\sim"]\;
\end{tikzcd}
$$
whose class is equal to $\Xi\,.$ All arrows of this zigzag are commutative cubes in $\tt M$ that are objectwise weak equivalences of $\tt M\,.$ In particular we have the following commutative $\tt M$--squares
$$
\begin{tikzcd}
&K_g\ar[r,"\zp_g"]\ar[d,"\sim"]\ar[d,swap,"\zw^k_2"]&\cX\ar[r,"g"]\ar[d,"\sim"]\ar[d,swap,"\zw^k_1"]&\mathfrak{X}\ar[d,"\sim"]\ar[d,swap,"\zw^k_0"] \\
&\mathfrak{F}^k_2\ar[r,"\mathfrak{f}^k_2"]&\mathfrak{F}^k_1\ar[r,"\mathfrak{f}^k_1"]&\mathfrak{F}^k_0\\
&F^g_2\ar[d,"\sim"]\ar[r,"f^g_2"]\ar[u,swap,"\sim"]\ar[u,"\zu^k_2"]&F^g_1\ar[d,"\sim"]\ar[u,swap,"\sim"]\ar[u,"\zu^k_1"]\ar[r,"f^g_1"]&F^g_0\ar[d,"\sim"]\ar[u,swap,"\sim"]
\ar[u,"\zu^k_0"] \\
&\vdots&\vdots&\vdots\\
&X\ar[r,"f"]\ar[u,swap,"\sim"]&\cX\ar[u,swap,"\sim"]\ar[r,"g"]&\mathfrak{X}\ar[u,swap,"\sim"]
\end{tikzcd}
$$
If we apply the functor $\zg=\zg_{\tt M}$ to them, we get commutative $\tt Ho(M)$--squares in which the weak equivalences have been transformed into isomorphisms of $\tt Ho(M)\,.$ By inverting the upward isomorphisms, we obtain the commutative $\tt Ho(M)$--diagram
\be\label{ReducedZigZag}
\begin{tikzcd}
&K_g\ar[r,"\zg(\zp_g)"]\ar[d,"\cong"]&\cX\ar[r,"\zg g"]\ar[d,"\cong"]&\mathfrak{X}\ar[d,"\cong"]\\
&X\ar[r,"\zg f"]&\cX\ar[r,"\zg g"]&\mathfrak{X}
\end{tikzcd}
\ee
in which the two last isomorphisms are equalities. Indeed, as $$\Xi=\digamma(\zw^x)^{-1}\circ\digamma(\zu^x)\circ\digamma(\zu^{k})^{-1}\circ\digamma(\zw^k)\;,$$ where $\digamma=\zg_{\tt h(M)}\,,$ we have
$$
\zG(\id_{\tt M}\cX,\id_{\tt M}\mathfrak{X})=\xi=\op{Ho}(R_1)(\Xi)=
$$
$$\op{Ho}(R_1)(\digamma(\zw^x)^{-1})\,\circ\,\op{Ho}(R_1)(\digamma(\zu^x))\,\circ\,\op{Ho}(R_1)(\digamma(\zu^{k})^{-1})\,\circ\,
\op{Ho}(R_1)(\digamma(\zw^k))=
$$
\be\label{IdentificationAsIds}
(\zG(R_1\zw^x))^{-1}\,\circ\,\zG(R_1\zu^x)\,\circ\,(\zG(R_1\zu^k))^{-1}\,\circ\,\zG(R_1\zw^k)\;.
\ee
As composition and weak equivalences of $\tt M^\s$ have been defined objectwise, applying $\zG$ means that we apply $\zg$ objectwise (a similar remark holds for $\tt h(M)$ and $\digamma$). Hence, the last row of \eqref{IdentificationAsIds} coincides with the two last columns of \eqref{ReducedZigZag}, which are therefore equalities as announced.\medskip

If we denote the left isomorphism in \eqref{ReducedZigZag} with $\zve$ and the composite morphism \be\label{ConMorzD}\zW\mathfrak{X}\stackrel{\zg(\zd_g)}{\longrightarrow}K_g\stackrel{\zve}{\longrightarrow}X\;,\ee with $\zD\,,$ we can add a corresponding square in Diagram \eqref{ReducedZigZag} on the left. If we apply the functor $[A,-]$ to this extended commutative $\tt Ho(M)$--diagram, we get the commutative $\tt Set$--diagram
\be\label{BasicExact}
\begin{tikzcd}
\left [A,\zW\mathfrak{X}\right ]\ar[r,"{[A,\zg(\zd_g)]}"]\ar[d,equal]&\left [A,K_g\right ]\ar[r,"{[A,\zg(\zp_g)]}"]\ar[d,swap,"{[A,\zve]}"]\ar[d,"\cong"]&\left [A,\cX\right ]\ar[r,"{[A,\zg g]}"]\ar[d,equal]&\left [A,\mathfrak{X}\right ]\ar[d,equal]\\
\left [A,\zW\mathfrak{X}\right ]\ar[r,"{[A,\zD]}"]&\left [A,X\right ]\ar[r,"{[A,\zg f]}"]&\left [A,\cX\right ]\ar[r,"{[A,\zg g]}"]&\left[A,\mathfrak{X}\right ]
\end{tikzcd}
\ee
On the other hand, since $\zW$ in the general case sends weak equivalences between fibrant objects to weak equivalences (resp., in the strongly proper case preserves all weak equivalences), the right derived functor $\mathbb{R}^{\op{K}}\zW\in\tt Fun(Ho(M),Ho(M))$ exists and satisfies
$$
\mathbb{R}^{\op{K}}\zW\circ\zg\doteq\zg\circ\zW\circ\tilde{F}\quad(\text{resp.,}\;\mathbb{R}^{\op{K}}\zW\circ\zg\doteq\zg\circ\zW)\;.
$$
In particular, if $h:Y\to Z$ is an $\tt M$--morphism between fibrant objects (resp., in the strongly proper case any $\tt M$--morphism), we get that
$$
\mathbb{R}\zW(\zg h):\mathbb{R}\zW(\zg Y)\to\mathbb{R}\zW(\zg Z)\quad\text{is given by}\quad\zg(\zW h):\zW Y\to\zW Z\;,
$$
where we omitted superscript $\op{K}.$ So if we apply first $\mathbb{R}\zW$ to the above-mentioned extended commutative $\tt Ho(M)$--diagram and then $[A,-]\,,$ we get the commutative $\tt Set$--diagram
\be\label{FollowUpExact}
\begin{tikzcd}
\left [A,\zW^2\mathfrak{X}\right ]\ar[rr,"{[A,\zg(\zW(\zd_g))]}"]\ar[d,equal]&&\left [A,\zW(K_g)\right ]\ar[rr,"{[A,\zg(\zW(\zp_g))]}"]\ar[d,swap,"{[A,\mathbb{R}\zW(\zve)]}"]\ar[d,"\cong"]&&\left [A,\zW\cX\right ]\ar[rr,"{[A,\zg(\zW g)]}"]\ar[d,equal]&&\left [A,\zW\mathfrak{X}\right ]\ar[d,equal]\\
\left [A,\zW^2 \mathfrak{X}\right ]\ar[rr,"{[A,\mathbb{R}\zW(\zD)]}"]&&\left [A,\zW X\right ]\ar[rr,"{[A,\zg(\zW f)]}"]&&\left [A,\zW\cX\right ]\ar[rr,"{[A,\zg(\zW g)]}"]&&\left[A,\zW\mathfrak{X}\right ]
\end{tikzcd}
\ee
We can of course iterate this approach. The upper rows (resp., the lower rows) of \eqref{BasicExact}, \eqref{FollowUpExact} and of the commutative diagrams obtained from the iteration are the long exact sequence (resp., the sequence) of Proposition \ref{LESH} (resp., of Proposition \ref{LESHGen}). From the proof of Theorem \ref{HFSES} we know that if the vertical bijections in \eqref{BasicExact}, \eqref{FollowUpExact}... respect the zero elements, then the sequence of Proposition \ref{LESHGen} is exact as well. We know from the same proof that the bijection $[A,\zve]$ respects the zero elements, as $\zve$ is the composite of images $\zg\zw$ of weak equivalences $\zw$ and inverses $\zg\zvw^{-1}$ of images of weak equivalences $\zvw\,.$ Since $\mathbb{R}\zW(\zve)$ is the composite of the $\zg(\zW\zw)$ and the $\zg(\zW\zvw)^{-1},$ so is the composite of images under $\zg$ of weak equivalences $\zW\zw$ and inverses of such images, the bijection $[A,\mathbb{R}\zW(\zve)]$ respects also the zero elements. This completes the proof.
\end{proof}

We close this section with the following comparison of different loop space functors.

\begin{prop}\label{Comparison}
Let $\tt M$ be a pointed model category, let $\op{Path}_0^a\,,$ $\op{Path}_0^b\,,$ $\op{Path}_0^c...$ be based path space functors in $\,\tt M$ and denote $\zW^a,$ $\zW^b,$ $\zW^c...$ the associated loop space functors. There exist canonical natural isomorphisms $\iota^{ba}:\mathbb{R}\zW^a\stackrel{\sim}{\Rightarrow}\mathbb{R}\zW^b$ which satisfy the cocycle condition $\iota^{cb}\circ\iota^{ba}=\iota^{ca}\,.$
\end{prop}

\begin{proof}
Let $g:\cX\to \mathfrak{X}$ be an $\tt M$--morphism, choose a lifting $\tilde F g:\tilde F \cX\to\tilde F \mathfrak{X}\,,$ consider Puppe's sequences $\cP_{\tilde F g}^a, \cP_{\tilde F g}^b\in\ell({\tt M})$ and observe that their restrictions in $\tt M^\s$ are related by the isomorphism $\xi:=\id_{\tt Ho(M^\s)}(\tilde F g)\,.$ If we proceed as in the proof of Proposition \ref{LESHGen} (but in the case of $\ell({\tt M})$ instead of $\tt h(M)$), we get the commutative $\tt Ho(M)$--diagram
$$
\begin{tikzcd}
\cdots\ar[r] &\zW^a(\tilde F\cX)\ar[r,"\zg(\zW^a(\tilde F g))"]\ar[d,"\cong"]&\zW^a(\tilde F\mathfrak{X})\ar[d,"\cong"]\ar[r,"\zg(\zd^a_{\tilde F g})"]&K^a_{\tilde F g}\ar[r,"\zg(\zp^a_{\tilde F g})"]\ar[d,"\cong"]&\tilde F\cX\ar[r,"\zg(\tilde F g)"]\ar[d,equal]&\tilde F\mathfrak{X}\ar[d,equal]\\
\cdots\ar[r] &\zW^b(\tilde F\cX)\ar[r,"\zg(\zW^b(\tilde F g))"]&\zW^b(\tilde F\mathfrak{X})\ar[r,"\zg(\zd^b_{\tilde F g})"]&K^b_{\tilde F g}\ar[r,"\zg(\zp^b_{\tilde F g})"]&\tilde F\cX\ar[r,"\zg(\tilde F g)"]&\tilde F\mathfrak{X}
\end{tikzcd}
$$
The degrees 3 and 4 part of this commutative diagram reads
$$
\begin{tikzcd}
\mathbb{R}\zW^a(\cX)\ar[r,"\mathbb{R}\zW^a(\zg g)"]\ar[d,"\cong"]\ar[d,swap,"\iota^{ba}_{\cX}"]&\mathbb{R}\zW^a(\mathfrak{X})\ar[d,"\cong"]\ar[d,swap,"\iota^{ba}_{\mathfrak{X}}"]\\
\mathbb{R}\zW^b(\cX)\ar[r,"\mathbb{R}\zW^b(\zg g)"]&\mathbb{R}\zW^b(\mathfrak{X})
\end{tikzcd}
$$
which proves the `canonical isomorphism' part of Proposition \ref{Comparison} (see Lemma \ref{NatTraMHoM}). The `cocycle condition' part is a direct consequence of Proposition \ref{Explicit}.
\end{proof}

\section{Comparison with Quillen's fibration sequences}\label{ComparisonQuillen}

Recall that in any model category $\tt M$ path objects are dual to cylinder objects: a path object of $X\in\tt M$ is an object $\op{Path}X\in\tt M$ together with a factorization $$X\stackrel{\sim}{\to}\op{Path}X\twoheadrightarrow X\times X$$ of the diagonal map $$\zD_X:=(\id_X,\id_X):X\to X\times X$$ into a weak equivalence followed by a fibration [we think of the first (resp., second) morphism of the factorization as the map which assigns to every point the constant path a this point (resp., to every path its start and end points)]. If {\it we fix a functorial factorization} $(\za,\zb)$ of the diagonal map and $f:X\to Y$ is a morphism, we get the commutative diagram
$$
\begin{tikzcd}
X\ar[r,"\za(\zD_X)"]\ar[r,swap,"\sim"]\ar[d,"f"]&\op{Path}X\ar[r,two heads,"\zb(\zD_X)"]\ar[d,dashed,"\op{Path}f"]& X\times X\ar[d,dashed,"f\times f"]\\
Y\ar[r,"\za(\zD_Y)"]\ar[r,swap,"\sim"]&\op{Path}Y\ar[r,two heads,"\zb(\zD_Y)"]&Y\times Y
\end{tikzcd}
$$
Indeed, if we denote $\zp_1$ and $\zp_2$ the projections out of $Y\times Y\,,$ there is a unique morphism $(f,f):X\to Y\times Y$ such that $$\zp_1\circ (f,f)=\zp_2\circ(f,f)=f\;.$$ Since $(f\times f)\circ\zD_X$ and $\zD_Y\circ f$ satisfy this condition, the total square commutes. It follows from the functoriality of the factorization that the arrow $\op{Path}f$ that makes the left and right squares commutative exists, and that $\op{Path}$ is an endofunctor of $\tt M\,.$ We refer to $\op{Path}$ as the {\bf path space functor} of $\,\tt M\,.$\medskip

Now let $\tt M$ be a pointed model category as in the preceding sections.\medskip

In Section \ref{SectionPuppe} we considered a {\bf based path space functor} $\op{Path}_0$ of $\,\tt M$ and the corresponding loop space functor $\zW$ of $\tt M\,.$ On fibrant objects $X\in\tt M_{\op{f}}$ the loop space $\zW X\in\tt M_{\op{f}}$ is the kernel of the fibration $\op{Path}_0X\twoheadrightarrow X\,$ and on $\tt M_{\op{f}}$--morphisms $f:X\to Y$ the $\tt M_{\op{f}}$--morphism $\zW f:\zW X\to \zW Y$ is the universal arrow
$$
\begin{tikzcd}
\zW X\ar[r,"k_X"]\ar[d,dashed,"\zW f"]&\op{Path}_0X\ar[r,two heads]\ar[d,"\op{Path}_0f"]&X\ar[d,"f"]\\
\zW Y\ar[r,"k_Y"]&\op{Path}_0Y\ar[r,two heads]&Y
\end{tikzcd}
$$

Quillen defines a loop space functor $\zW^Q$ of $\tt M$ from the path space functor $\op{Path}$ of $\tt M\,.$ On objects $X\in\tt M$ the loop space $\zW^QX\in\tt M_{\op{f}}$ is the kernel of the fibration $\op{Path}X\twoheadrightarrow X\times X$ and on $\tt M$--morphisms $f:X\to Y$ the $\tt M_{\op{f}}$--morphism $\zW^Qf:\zW^QX\to\zW^QY$ is the universal arrow
$$
\begin{tikzcd}
\zW^QX\ar[r,"\zk_X"]\ar[d,dashed,"\zW^Qf"]&\op{Path}X\ar[r,two heads]\ar[d,dashed,"\op{Path}f"]&X\times X\ar[d,dashed,"f\times f"]\\
\zW^QY\ar[r,"\zk_Y"]&\op{Path}Y\ar[r,two heads]&Y\times Y
\end{tikzcd}
$$
A non-obvious result is that for any $V\in\tt M_{\op{f}}$ the functor of points $$[-,\zW^QV]\in \tt Fun(Ho(M)^{\op{op}}, Set)$$ is valued in the category $\tt Grp$ of groups and that `accordingly' $\zW^QV$ is a group object of $\tt Ho(M)\,.$ Another non-trivial result is that if $K$ is the kernel of a fibration $U\twoheadrightarrow V$ between fibrant objects $U,V,$ there is an $\tt M$--morphism $\zr:\zW^QV\times_{\tt M} K\to K$ such that $\zg\zr:\zW^QV\times_{\tt Ho(M)} K\to K$ is an action of the group object $\zW^QV$ on $K\,.$

\begin{theo}\label{QLSF}
Let $\,\tt M$ be a pointed model category that is equipped with a path space functor $\op{Path}$ implemented by a fixed functorial factorization. Quillen's loop space functor $\zW^Q$ is a loop space functor in the sense of the present paper, i.e., a loop space functor associated to a based path space functor $\op{Path}^Q_0$ of $\,\tt M\,.$
\end{theo}

\begin{proof}
It is natural to define the based path space $\op{Path}^Q_0X$ of $X\in\tt M$ as the kernel of the composite $\op{Path}X\twoheadrightarrow X\times X\stackrel{\zp_1}{\to}X\,,$ where $\zp_1$ is the projection on the first factor of $X\times X\,.$ The projection on the second factor will be denoted $\zp_2\,.$ For $f:X\to Y$ and $i\in\{1,2\}\,,$ we have a commutative diagram
$$
\begin{tikzcd}
\op{Path}^Q_0X\ar[r,"\mathfrak{k}_X"]\ar[d,dashed,"\op{Path}^Q_0f"]&\op{Path}X\ar[r,two heads]\ar[d,dashed,"\op{Path}f"]&X\times X\ar[r,"\zp_i"]\ar[d,dashed,"f\times f"]&X\ar[d,"f"]\\
\op{Path}^Q_0Y\ar[r,"\mathfrak{k}_Y"]&\op{Path}Y\ar[r,two heads]&Y\times Y\ar[r]&Y
\end{tikzcd}
$$
where $\op{Path}^Q_0f$ is the universal arrow that we get for $i=1\,.$ Since $\op{Path}$ is an endofunctor, the same holds for $\op{Path}^Q_0\,.$ If $i=2$ the diagram gives a natural transformation $\op{Path}^Q_0\Rightarrow\id_{\tt M}\,.$\medskip

The functor $\op{Path}^Q_0$ is a based path space functor in the sense of Definition \ref{BPSF}, if for every fibrant $X\in\tt M$ the $\tt M$--morphism $\op{Path}^Q_0X\to X$ is a fibration with an acyclic domain.\medskip

Since fibrations are closed under pullbacks, the projection $\zp_i:X\times X\to X$ is a fibration if $X$ is fibrant, so that the composite $p_i:\op{Path}X\twoheadrightarrow X\times X\stackrel{\zp_i}{\to}X$ is also a fibration. As $X\stackrel{\sim}{\to}\op{Path}X\stackrel{p_i}{\to}X$ is identity by definition of $\op{Path}X\,,$ it follows from the 2-out-of-3 axiom that $p_i$ is a weak equivalence and therefore a trivial fibration. Since trivial fibrations are closed under pullbacks and $\op{Path}^Q_0X:=\ker p_1=\op{Path}X\times_X0\,,$ the morphism $\op{Path}^Q_0X\to 0$ is a trivial fibration, so that $\op{Path}^Q_0X$ is acyclic. Next we show that the $\tt M$-morphism $$\mathfrak{p}_2:\op{Path}^Q_0 X\stackrel{\mathfrak{k}_X}{\to}\op{Path}X\stackrel{p_2}{\to} X$$ is a fibration. Therefore, let $Y\stackrel{\sim}{\rightarrowtail}Z$ be a trivial cofibration such that the left square of the diagram
$$
\begin{tikzcd}
Y\ar[r]\ar[d,tail,"\sim"]& \op{Path}_{0}^Q X\ar[r,"\mathfrak{k}_X"]\arrow[d,very near start,"\mathfrak{p}_2"]&\op{Path}X\arrow[d,two heads]\\
Z\arrow[r]\ar[rru,dashed, crossing over,near start,"\ell"]\ar[ru,dashed,"\mathfrak{l}"]&X\ar[r,"{(0,\id)}"]&X\times X
\end{tikzcd}
$$
commutes. To see that the right square also commutes, observe that there is a unique morphism $m:\op{Path}^Q_0X\to X\times X$ such that $\zp_1\circ m=0$ and $\zp_2\circ m=\mathfrak{p}_2\,.$ Since both morphisms $\op{Path}^Q_0X\to X\times X$ in the right square fulfill these conditions, they coincide. As the total square now commutes, there exists a lifting $\ell:Z\dashrightarrow\op{Path}X\,.$ Since $\op{Path}^Q_0X$ is a kernel, there is a unique morphism $\mathfrak{l}:Z\to\op{Path}^Q_0X$ such that $\mathfrak{k}_X\circ\mathfrak{l}=\ell\,.$ As the total upper triangle commutes and $\mathfrak{k}_X$ is a monomorphism so left cancellable, the left upper triangle commutes. In order to conclude that the left lower triangle commutes and that $\mathfrak{p}_2$ is a fibration, it suffices to notice that $(0,\id)$ is left cancellable.\medskip

It remains to prove that the loop space functor $\zW$ associated to the chosen based path space functor $\op{Path}^Q_0$ is Quillen's loop space functor $\zW^Q\,.$ For any $f:X\to Y$ we have the following commutative diagram
\be\label{LoopQuillen}
\begin{tikzcd}
&\zW X\ar[ld,dashed,swap,"\mathfrak{l}_X"]\ar[dd,dashed,near end,"\zW f"]\ar[r,"k_X(\mathfrak{p}_2)"]&\op{Path}^Q_0X\ar[dd,dashed,"\op{Path}^Q_0f"]\ar[rrr,bend left=30,two heads,"\mathfrak{p}_2"]\ar[r,"\mathfrak{k}_X(p_1)"]&\op{Path}X\ar[dd,dashed,"\op{Path}f"]\ar[r,two heads,"\zf_X"]&X\times X\ar[dd,dashed,near end,"f\times f"]\ar[r,"\zp_2"]&X\ar[dd,near end,"f"]&\\
\zW^QX\ar[rru,dashed,near end,crossing over,"\ell_X"]\ar[dd,dashed,"\zW^Q f"]\ar[rrru,crossing over,very near end,swap,"\zk_X(\zf_X)"]&&&&&&X\ar[from=lllu,bend right=20,crossing over,"p_1"]\ar[dd,"f"]\ar[from=llu,crossing over,near end,"\zp_1"]\\
&\zW Y\ar[r,"k_Y"]&\op{Path}^Q_0Y\ar[r,"\mathfrak{k}_Y"]&\op{Path}Y\ar[r,two heads,"\zf_Y"]&Y\times Y\ar[r]&Y&\\
\zW^QY\ar[rrru,swap,"\zk_Y"]\ar[from=ru,dashed]\ar[rru,dashed]&&&&&&Y\ar[from=llu]
\end{tikzcd}
\ee
in which only the universal morphisms $\ell_X$ and $\mathfrak{l}_X$ and the associated commutative squares require explanation. Obviously there is a unique morphism $\ell_X:\zW^QX\to \op{Path}_0^QX$ such that \be\label{Q1}\mathfrak{k}_X(p_1)\circ\ell_X=\zk_X(\zf_X)\;.\ee Notice now that $$\zp_i\circ\zf_X\circ \mathfrak{k}_X(p_1)\circ k_X(\mathfrak{p}_2)=0\,,$$ as $\mathfrak{k}_X(p_1)$ (resp., $k_X(\mathfrak{p}_2)$) is the kernel of $p_1$ (resp., $\mathfrak{p}_2$). However, the zero morphism $\zW X\to 0\to X\times X$ is the unique morphism from $\zW X$ to $X\times X$ whose composite with $\zp_i$ is the zero morphism $\zW X\to 0\to X\,.$ Hence $$\zf_X\circ \mathfrak{k}_X(p_1)\circ k_X(\mathfrak{p}_2)=0$$ and there is a unique morphism $\mathfrak{l}_X:\zW X\to \zW^Q X$ such that \be\label{Q2}\zk_X(\zf_X)\circ\mathfrak{l}_X=\mathfrak{k}_X(p_1)\circ k_X(\mathfrak{p}_2)\;.\ee From \eqref{Q1} and \eqref{Q2} it follows that $$\mathfrak{k}_X(p_1)\circ\ell_X\circ\mathfrak{l}_X=\mathfrak{k}_X(p_1)\circ k_X(\mathfrak{p}_2)\quad\text{and}\quad \ell_X\circ\mathfrak{l}_X=k_X(\mathfrak{p}_2)\;,$$ as $\mathfrak{k}_X(p_1)$ is left cancellable. The commutativity of the square associated with $\ell_X,\ell_Y$ (resp., with $\mathfrak{l}_X,\mathfrak{l}_Y$) follows from the left cancellability of $\mathfrak{k}_Y$ (resp., of $\zk_Y$).\medskip

We are now prepared to show that the pair $(\zW^QX,\ell_X)$ is a kernel of $\mathfrak{p}_2$ so that $\zW^Q X\doteq\zW X\,,$ $\zW^Q f\doteq\zW f$ and $\zW^Q\doteq \zW\,,$ which then completes the proof (cf. Diagram \eqref{LoopQuillen}). To see that $(\zW^QX,\ell_X)$ is a kernel, notice first that $$\mathfrak{p}_2\circ\ell_X=\zp_2\circ\zf_X\circ\zk_X(\zf_X)=0\;.$$ Further, if $\zl:\zL\to \op{Path}^Q_0X$ satisfies $\mathfrak{p}_2\circ\zl=0\,,$ there is a unique morphism $\zm:\zL\to\zW X$ such that $k_X(\mathfrak{p}_2)\circ\zm=\zl\,.$ However, then $\mathfrak{l}_X\circ\zm:\zL\to\zW^QX$ satisfies $\ell_X\circ\mathfrak{l}_X\circ\zm=\zl$ and is the unique morphism with these properties. Indeed, if $\zn:\zL\to\zW^QX$ is such that $\ell_X\circ\zn=\zl\,,$ then $$\zk_X(\zf_X)\circ\zn=\mathfrak{k}_X(p_1)\circ\ell_X\circ\zn=\mathfrak{k}_X(p_1)\circ\zl=\mathfrak{k}_X(p_1)\circ k_X(\mathfrak{p}_2)\circ\zm=\zk_X(\zf_X)\circ\mathfrak{l}_X\circ\zm\,,$$ so that $\zn=\mathfrak{l}_X\circ\zm\,.$
\end{proof}

\begin{rem}
\emph{If we dualize the constructions of this paper, we get the reduced suspension functor $\zS$ instead of the loop space functor $\zW\,$. From Theorem \ref{QLSF} and its dualization it follows that Quillen's suspension and loop space functors $\zS^Q$ and $\zW^Q$ are suspension and loop space functors in our sense. Since $$\mathbb{L}\zS^Q:{\tt Ho(M)}\rightleftarrows{\tt Ho(M)}:\mathbb{R}\zW^Q$$ are adjoint functors, Proposition \ref{Comparison} and its dualization show that any derived suspension functor $\mathbb{L}\zS$ is left adjoint to any derived loop space functor $\mathbb{R}\zW\,.$}
\end{rem}

Because of Theorems \ref{QLSF} and \ref{Preservation} Quillen's loop space functor $\zW^Q\in\tt Fun(M,M)$ preserves weak equivalences between fibrant objects, so that Theorem \ref{Fundamental0} implies that its derived functor $\mathbb{R}^{\op{K}}\zW^Q$ exists and is given at $\mathfrak{X}\in\tt M$ by $$\mathbb{R}\zW^Q\,(\mathfrak{X})\doteq\zW^Q(\tilde{F}\mathfrak{X})\in\tt Ho(M)\;,$$ where we omitted superscript $\op{K}\,.$ Quillen now gives the following
\begin{defi}\label{DefFS}
Let $\tt M$ be a pointed model category. A {\bf fibration sequence} in $\,\tt Ho(M)$ is a sequence $$X\to\cX\to\mathfrak{X}$$ in $\,\tt Ho(M)$ together with a $\tt Ho(M)$--morphism $R:\mathbb{R}\zW^Q(\mathfrak{X})\times_{\tt Ho(M)} X\to X\,,$ such that the following holds:
\begin{enumerate}
\item the sequence is isomorphic in $\,\tt Ho(M)$ to a sequence $K\to U\to V$ that is implemented by the kernel $K$ of an $\tt M$--fibration $U\twoheadrightarrow V$ between fibrant objects $U,V\,,$ i.e., there is a commutative $\tt Ho(M)$--diagram
\be
\begin{tikzcd}\label{IsoFS}
X\ar[r]\ar[d,"\cong"]\ar[d,swap,"i"]&\cX\ar[r]\ar[d,"\cong"]\ar[d,swap,"j"]&\mathfrak{X}\ar[d,"\cong"]\ar[d,swap,"k"]\\
K\ar[r]&U\ar[r]&V
\end{tikzcd}
\ee
whose vertical arrows are isomorphisms;
\item the morphism $R$ coincides under this isomorphism with the morphism $\zg\zr\,,$ i.e., if $\,\tilde k$ is the isomorphism $$\mathbb{R}\zW^Q(\mathfrak{X})\stackrel{\cong}{\to}\zW^QV$$ that is induced by $k$ and if $\tilde{k}^{-1}\times i^{-1}$ is the isomorphism $$\zW^QV\times_{\tt Ho(M)}K\stackrel{\cong}{\to} \mathbb{R}\zW^Q(\mathfrak{X})\times_{\tt Ho(M)}X$$ that is induced by $\tilde{k}^{-1}$ and $i^{-1}\,,$ we have $$i\circ R\circ(\tilde{k}^{-1}\times i^{-1})=\zg\zr\;.$$
\end{enumerate}
\end{defi}
\noindent We will show that Quillen's fibration sequences are tightly related to our homotopy fiber sequences. An initial observation that confirms this claim is Quillen's result that if we take a fibration sequence $K\to U\to V$ with action $\zg\zr\,,$ we have a connecting $\tt Ho(M)$--morphism \be\label{ConMor}\zd^Q:\zW^Q V\to K\;,\ee namely $$\zW^Q V\stackrel{(\id,\zg(0))}{\longrightarrow}\zW^Q V\times_{\tt Ho(M)}K\stackrel{\zg\zr}{\longrightarrow}K\,,$$ such that $\zW^Q V\stackrel{\zd^Q}{\to} K\to U$ is also a fibration sequence. Moreover, Quillen gets a long exact sequence similar to the long exact sequence in Proposition \ref{LESH}. Finally, the connecting morphism \eqref{ConMor} and the similarly defined connecting morphism \be\label{ConMorzDQ}\zD^Q:\mathbb{R}\zW^Q(\mathfrak{X})\to X\ee render the left square of the diagram
\be
\begin{tikzcd}\label{ConMorzDQDiag}
\mathbb{R}\zW^Q(\mathfrak{X})\ar[r,"\zD^Q"]\ar[d,"\cong"]\ar[d,swap,"\tilde{k}"]&X\ar[r]\ar[d,"\cong"]\ar[d,swap,"i"]&\cX\ar[d,"\cong"]\ar[d,swap,"j"]\\
\zW^QV\ar[r,"\zd^Q"]&K\ar[r]&U
\end{tikzcd}
\ee
commutative.\medskip

The next theorem specifies the relationship between fibration sequences and homotopy fiber sequences.

\begin{theo}
A homotopy fiber sequence $X\stackrel{f}{\longrightarrow}\cX\stackrel{g}{\longrightarrow}\mathfrak{X}$ of $h({\tt M}_{\op{f}})$ is a fibration sequence $X\stackrel{\zg f}{\longrightarrow}\cX\stackrel{\zg g}{\longrightarrow}\mathfrak{X}$ in $\tt Ho(M)$ and the connecting morphism $\zD$ defined in \eqref{ConMorzD} coincides with the connecting morphism $\zD^Q$ considered in \eqref{ConMorzDQ}.
\end{theo}

\begin{proof}
Let $X\stackrel{f}{\to}\cX\stackrel{g}{\to}\mathfrak{X}$ be an objectwise fibrant homotopy fiber sequence of $\tt M\,,$ let $\cX\stackrel{\sim}{\to}\bar\cX\stackrel{\bar g}{\twoheadrightarrow}\mathfrak{X}$ be a factorization of $g$ into a weak equivalence followed by a fibration and let $(K,\zk)$ be the kernel of $\bar g\,.$ Then $K\stackrel{\zk}{\to}\bar\cX\stackrel{\bar g}{\twoheadrightarrow}\mathfrak{X}$ is also a homotopy fiber sequence of $\tt M$ and, if we proceed as at the beginning of the proof of Proposition \ref{LESHGen}, we get a commutative $\tt Ho(M)$--diagram
\be\label{Bottom}
\begin{tikzcd}
&X\ar[r,"\zg f"]\ar[d,"\cong"]\ar[d,swap,"i"]&\cX\ar[r,"\zg g"]\ar[d,"\cong"]\ar[d,swap,"j"]&\mathfrak{X}\ar[d,equal]\\
&K\ar[r,"\zg \zk"]&\bar\cX\ar[r,"\zg \bar g"]&\mathfrak{X}
\end{tikzcd}
\ee
The group action $\zg\zr:\zW^Q\mathfrak{X}\times_{\tt Ho(M)}K\to K$ mentioned just above Theorem \ref{QLSF} induces a $\tt Ho(M)$--morphism
$$
R:\mathbb{R}\zW^Q(\mathfrak{X})\times_{\tt Ho(M)}X\stackrel{\id\times i}{\to}\zW^Q\mathfrak{X}\times_{\tt Ho(M)}K\stackrel{\zg\zr}{\to}K\stackrel{i^{-1}}{\to}X\;
$$
which, together with the sequence $X\stackrel{\zg f}{\to}\cX\stackrel{\zg g}{\to}\mathfrak{X}\,,$ satisfies the requirements of Definition \ref{DefFS}, so that this sequence is a fibration sequence in $\tt Ho(M)$ as announced.\medskip

To prove that $\zD=\zD^Q$ we will describe the following commutative $\tt Ho(M)$--diagram:
\be\label{Delta}
\begin{tikzcd}
&&&&\zW^Q\mathfrak{X}\ar[r,"\zg(\zd_{\bar g})"]\ar[dd,equal]&K_{\bar g}\ar[r,"\zg(\zp_{\bar g})"]\ar[dd,near end,swap,"\zg\mathfrak{k}^{-1}"]\ar[from=dd,near start,swap,"\cong"]&\bar\cX\ar[r,"\zg\bar g"]\ar[dd,equal]&\mathfrak{X}\ar[dd,equal]\\
\zW^Q\mathfrak{X}\ar[r,swap,"\zg(\zd_g)"]\ar[rrrru,equal]&K_g\ar[r,swap,"\zg(\zp_g)"]\ar[rrrru,crossing over]&\cX\ar[r,swap,"\zg g"]\ar[rrrru,crossing over]&\mathfrak{X}\ar[rrrru,crossing over,equal]&&&&\\
&&&&\zW^Q\mathfrak{X}\ar[r,"\zd^Q"]&K\ar[r,"\zg\zk"]\ar[from=dllll]&\bar\cX\ar[r,"\zg\bar g"]\ar[from=dllll]&\mathfrak{X}\ar[dllll,equal]\\
\zW^Q\mathfrak{X}\ar[r,swap,"\zD"]\ar[uu,crossing over,equal]&X\ar[r,swap,"\zg f"]\ar[from=uu,crossing over,swap,"\zve"]\ar[from=uu,crossing over,"\cong"]&\cX\ar[r,swap,"\zg g"]\ar[from=uu,crossing over,equal]&\mathfrak{X}\ar[from=uu,crossing over,equal]&&&&
\end{tikzcd}
\ee

The commutative front of Diagram \eqref{Delta} comes from the commutative $\tt Ho(M)$--diagram \eqref{ReducedZigZag} and Equation \eqref{ConMorzD}. In particular
\be\label{ComLFS}
\zD=\zve\circ\zg(\zd_g)\;.
\ee

The commutative bottom is nothing but the commutative $\tt Ho(M)$--diagram \ref{Bottom}. In particular, its arrows that are not labelled are the isomorphisms $i$ and $j\,$.\smallskip

The upper row of the back of Diagram \eqref{Delta} consists of the image under $\zg$ of the terms of degrees $0-3$ of Puppe's long homotopy fiber sequence $\cP_{\bar g}\in\ell({\tt M}_{\op{f}})$ associated to $\bar g\in{\tt M}^\s_{\op{f}}\,.$ Notice that it follows from Diagram \eqref{PuppeProofPic} and the pasting law for pullbacks that $(\zW^Q\mathfrak{X},\zd_{\bar g})$ is the kernel of the fibration $\zp_{\bar g}:K_{\bar g}\twoheadrightarrow\bar\cX$ with fibrant source and target.\smallskip

The lower row of the back contains the fibration sequence that is implemented by the kernel $K$ of the fibration $\bar\cX\twoheadrightarrow\mathfrak{X}$ between the fibrant objects $\bar\cX$ and $\mathfrak{X}$ and the connecting morphism $$\zW^Q\mathfrak{X}\stackrel{\zd^Q}{\to}K$$ of Equation \eqref{ConMor} that makes \be\label{SFS}\zW^Q\mathfrak{X}\stackrel{\zd^Q}{\to}K\stackrel{\zg\zk}{\to}\bar\cX\ee a fibration sequence (see paragraph below Definition \ref{DefFS} and
\cite[Section I.3.5, Proposition 3]{Quill}).\medskip

As in the commutative $\tt M$--diagram
$$
\begin{tikzcd}
&&K_{\bar g}\ar[dd]&\bar\cX\ar[dd,two heads,"\bar g"]\ar[lld,equal]\ar[from=l,crossing over,"\zp_{\bar g}"]\\
K\ar[r,"\zk"]\ar[dd]\ar[rru,dashed,"\mathfrak{k}"]&\bar\cX&&\\
&&\op{Path}^Q_0\mathfrak{X}&\mathfrak{X}\ar[lld,equal]\ar[from=l,crossing over,two heads]\\
0\ar[rru,near start,"\sim"]\ar[r]&\mathfrak{X}\ar[from=uu,crossing over,two heads,"\bar g"]&&
\end{tikzcd}
$$
the front square is a model square and the back square is the pullback of a weakly equivalent fibrant cospan, the universal arrow $\mathfrak{k}:K\dashrightarrow K_{\bar g}$ is a weak equivalence. This explains the middle and right commutative squares of the back of Diagram \eqref{Delta}. In order to show that \eqref{SFS} satisfies the requirements of Definition \ref{DefFS}, Quillen had to construct a sequence that is implemented by the kernel of a fibration between fibrant objects and is isomorphic in $\tt Ho(M)$ to \eqref{SFS}. Actually he showed that the left square of the back of Diagram \eqref{Delta} commutes and uses the isomorphism given by the left and middle squares:
\be\label{ComLBS}
\zd^Q=\zg\mathfrak{k}^{-1}\circ\zg(\zd_{\bar g})\;.
\ee

In order to understand the top square of \eqref{Delta}, we consider the commutative $\tt M$--diagram
$$
\begin{tikzcd}
&&K_{\bar g}\ar[dd]&\bar\cX\ar[dd,two heads,"\bar g"]\ar[from=l,crossing over,"\zp_{\bar g}"]\\
K_g\ar[r,near end,"\zp_g"]\ar[dd]\ar[rru,dashed]&\cX\ar[rru,near start,"\sim"]&&\\
&&\op{Path}^Q_0\mathfrak{X}&\mathfrak{X}\ar[lld,equal]\ar[from=l,crossing over,two heads]\\
\op{Path}^Q_0\mathfrak{X}\ar[rru,equal]\ar[r,two heads]&\mathfrak{X}\ar[from=uu,crossing over,"g"]&&
\end{tikzcd}
$$
Once more, since the front square is a model square and the back square is the pullback of a weakly equivalent fibrant cospan, the universal arrow $K_g\dashrightarrow K_{\bar g}$ is a weak equivalence. Hence, the top middle and right squares of \eqref{Delta} commute and their arrows that are not labelled are isomorphisms; the one on the right is isomorphism $j$ and the one on the left is an isomorphism that we denote by $\iota\,$. From the commutativity of Diagram \eqref{PuppeProofPic2} follows that the top left square of Diagram \eqref{Delta} commutes:
\be\label{ComLTS}
\zg(\zd_{\bar g})=\iota\circ\zg(\zd_g)\;.
\ee

It remains to explain the commutativity of the squares that are parallel to the right face of Diagram \eqref{Delta}. Only the commutativity of the leftmost square is not entirely obvious. However, as $$K_g\stackrel{\zp_g}{\to}\cX\stackrel{g}{\to}\mathfrak{X}\quad\text{and}\quad K\stackrel{\zk}{\to}\bar\cX\stackrel{\bar g}{\twoheadrightarrow}\mathfrak{X}$$ are homotopy fiber sequences of $\tt M$ and the factorization $$\cX\stackrel{\sim}{\to}\bar\cX\stackrel{\bar g}{\twoheadrightarrow\mathfrak{X}}$$ of $g$ implements an isomorphism $\xi$ in $\tt Ho(M^\s)$ between their restrictions, there is a unique isomorphism $\Xi$ in $\tt Ho(h(M))$ that extends $\xi\,.$ Hence
\be\label{ComLPS}
i\circ\zve=\zg\mathfrak{k}^{-1}\circ\iota\;.
\ee

If we compare Diagrams \eqref{Bottom} and \eqref{IsoFS}, we see that in our case $k=\id\,,$ so that Diagram \eqref{ConMorzDQDiag} shows that
\be\label{zDQ}
\zD^Q=i^{-1}\circ\zd^Q\;.
\ee
On the other hand, it follows from Equations \eqref{ComLFS}, \eqref{ComLTS}, \eqref{ComLPS} and \eqref{ComLBS} that
\be\label{zD}
\zD=\zve\circ\zg(\zd_g)=\zve\circ\iota^{-1}\circ\zg(\zd_{\bar g})= i^{-1}\circ\zg\mathfrak{k}^{-1}\circ\zg(\zd_{\bar g})=i^{-1}\circ\zd^Q\;.
\ee
Finally, Equations \eqref{zDQ} and $\eqref{zD}$ allow us to conclude that $\zD=\zD^Q\,.$
\end{proof}

\section{Application to chain complexes}\label{Applications}

\subsection{Long homotopy fiber sequence and long exact homology sequence}

A particular advantage of our homotopy fiber sequence concept and related theory is that they are easy to apply. Let us summarize our construction. In each {\it pointed model category} that is equipped with a {\it based path space functor}, we consider the associated {\it loop space functor} and, for each morphism between fibrant objects, we take the associated {\it homotopy kernel} and {\it connecting morphism}. We then get {\it Puppe's long homotopy fiber sequence} and the corresponding {\it long exact sequences of sets}. In this section we apply the previous construction to chain complexes of modules.\medskip

Let $\tt A$ be an Abelian category and denote by ${\tt Ch(A)}$ the Abelian category of chain complexes and chain maps in $\tt A\,.$ If $R$ is a unital ring, the category $R-{\tt Mod}$ of left $R$--modules and $R$--linear maps is Abelian and ${\tt Ch}(R):={\tt Ch}(R-{\tt Mod})$ is the (Abelian) category of chain complexes of (left) $R$--modules and corresponding chain maps. We denote ${\tt Ch}_{\ge 0}(R)$ the full subcategory of non-negatively graded chain complexes of $R$-modules. Both categories, ${\tt Ch}(R)$ and ${\tt Ch}_{\ge 0}(R)\,,$ have a projective model structure in which weak equivalences are quasi-isomorphisms, while fibrations are degree-wise surjective chain maps in the unbounded case and chain maps that are surjective in positive degrees in the non-negatively graded case. In particular, in both cases all objects are fibrant. Moreover, both {\it model categories} are {\it pointed} with zero object the chain complex $(\{0\},0)\,.$\medskip

If $p\in\Z\,,$ the translation functor $[p]\in\tt Fun(Ch(A),Ch(A))$ is defined on objects $(A,d_A)\in\tt Ch(A)$ by \be\label{Translation}A[p]_n:=A_{n-p}\quad\text{and}\quad d_{A[p]}:=(-1)^pd_A\;,\ee and on $\tt Ch(A)$--morphisms $f:(A,d_A)\to (B,d_B)$ by $f[p]_n:=f_{n-p}\,.$ Further, the mapping cone $\op{Mc}(f)$ of the chain map $f$ is the chain complex given by
\be\label{Mc}\op{Mc}(f)_n:=A[1]_n\oplus B_n\quad\text{and}\quad d_{\op{Mc}(f)}:=\left(
                                      \begin{array}{cc}
                                        d_{A[1]} & 0 \\
                                        f & d_B \\
                                      \end{array}
                                    \right)\;\;.
\ee

It is well known that to every short exact sequence $0\to A\stackrel{f}{\to} B\stackrel{g}{\to} C\to 0$ of chain complexes and chain maps in $\tt A$ is associated a long exact sequence in homology. It is easy to see that the short sequences of chain complexes \be\label{MapLabels}0\to B\stackrel{i}{\to} \op{Mc}(f)\stackrel{p}{\to} A[1]\to 0\ee ($i$ and $p$ are the canonical injection and projection, respectively) and $0\to C\to \op{Mc}(g)\to B[1]\to 0$ are exact. The long exact sequences associated to the latter two short exact sequences and the long exact sequence associated to the former short exact sequence are known to coincide. Since $H_n(A[p])=H_{n-p}(A)\,,$ the long exact sequence in homology reads for instance \cite{Wei}
$$
\cdots \longrightarrow H_1(\op{Mc}(f))\stackrel{H_{1}(p)}{\longrightarrow} H_0(A)\stackrel{H_0(f)}{\longrightarrow} H_0(B)$$
\be\label{LESWei}\stackrel{H_0(i)}{\longrightarrow}H_0(\op{Mc}(f))\stackrel{H_0(p)}{\longrightarrow} H_{-1}(A)\stackrel{H_{-1}(f)}{\longrightarrow} H_{-1}(B)\longrightarrow\cdots\;\;.
\ee

As mentioned above, in the case ${\tt A}=R-{\tt Mod}\,,$ the underlying category ${\tt Ch}(R)$ or ${\tt Ch}_{\ge 0}(R)$ is a pointed model category. Hence it is natural to ask whether we can find a based path space functor of the category of chain complexes considered, such that the preceding long exact sequence in homology associated to $f:A\to B$ or to the sequence
\be\label{SeqT}
\cdots\longrightarrow\op{Mc}(f)[-1]\stackrel{p[-1]}{\longrightarrow}A\stackrel{f}{\longrightarrow} B\stackrel{i}{\longrightarrow} \op{Mc}(f)\stackrel{p}{\longrightarrow} A[1]\stackrel{f[1]}{\longrightarrow}B[1]\longrightarrow\cdots
\ee
can be obtained as a long exact sequence of sets
$$
\cdots\longrightarrow [M,\zW^2B]\stackrel{[M,\zg(\zW(\zd_f))]}{\longrightarrow} [M,\zW(K_f)]\stackrel{[M,\zg(\zW(\zp_f))]}{\longrightarrow} [M,\zW A]\stackrel{[M,\zg(\zW f)]}{\longrightarrow}
$$
\be\label{LESH2}
[M,\zW B]\stackrel{[M,\zg(\zd_f)]}{\longrightarrow} [M,K_f]\stackrel{[M,\zg(\zp_f)]}{\longrightarrow} [M,A]\stackrel{[M,\zg f]}{\longrightarrow} [M,B]
\ee
corresponding to Puppe's long homotopy fiber sequence
\be\label{SeqP}
\cdots\longrightarrow \zW^2B\stackrel{\zW(\zd_f)}{\longrightarrow} \zW(K_f)\stackrel{\zW(\zp_f)}{\longrightarrow} \zW A\stackrel{\zW f}{\longrightarrow} \zW B\stackrel{\zd_f}{\longrightarrow} K_f\stackrel{\zp_f}{\longrightarrow} A\stackrel{f}{\longrightarrow} B
\ee
of $f\,.$

\subsection{Chain complexes as pointed model category}\label{ChCoI}

The comparison of \eqref{SeqT} and \eqref{SeqP} suggests that we find a based path space functor
$$
\op{Path}_0\in{\tt Fun}({\tt Ch}(R),{\tt Ch}(R))
$$
such that the homotopy kernel $K_f:=\op{Path}_0B\times_BA$ of a chain map $f:A\to B$ coincides with the shifted mapping cone $\op{Mc}(f)[-1]$ of $f$, so that we must define $\op{Path}_0B=K_{\id_B}$ by
$$
\op{Path}_0B:=\op{Mc}(\id_B)[-1]\in{\tt Ch}(R)\;,
$$
for any $B\in{\tt Ch}(R)\,.$ In view of \eqref{Translation} and \eqref{Mc} we have
$$
(\op{Path}_0B)_n=B_n\oplus B_{n+1}\quad\text{and}\quad \op{d}_B:=d_{\op{Path}_0B}=-d_{\op{Mc}(\id_B)}=\left(
                                      \begin{array}{cc}
                                        d_{B} & 0 \\
                                        -\id_B & -d_B \\
                                      \end{array}
                                    \right)\;\;.
$$
If $B\in{\tt Ch}_{\ge 0}(R)\subset{\tt Ch}(R)\,,$ then its based path space
$$
\op{Path}_0B:\;\,\cdots \stackrel{\op{d}_{B,2}}{\longrightarrow} B_1\oplus B_2\stackrel{\op{d}_{B,1}}{\longrightarrow}B_0\oplus B_1\stackrel{\op{d}_{B,0}}{\longrightarrow}B_0\stackrel{\op{d}_{B,-1}}{\longrightarrow}0
$$
in ${\tt Ch}(R)$ has a term in degree $-1\,,$ so that we define its based path space $\op{Trath}_0B$ in ${\tt Ch}_{\ge 0}(R)$ by truncation as the sub-complex
$$
\op{Trath}_0B:\;\,\cdots \stackrel{\op{d}_{B,2}}{\longrightarrow} B_1\oplus B_2\stackrel{\op{d}_{B,1}}{\longrightarrow}\ker\op{d}_{B,0}\stackrel{\op{d}_{B,0}}{\longrightarrow}0\;.
$$
If $f:A\to B$ is a chain map, then $f\oplus f[-1]$ is a chain map from $A\oplus A[-1]$ to $B\oplus B[-1]\,,$ i.e., it is a degree 0 $R$-linear map that commutes with the differentials $d\oplus(-d)\,,$ so also with the differentials $\op{d}$ since the additional terms are both equal to $-f\,.$ Hence $\op{Path}_0f:=f\oplus f[-1]$ is a chain map from $\op{Path}_0A$ to $\op{Path}_0B$ and its restriction $\op{Trath}_0f:=\op{Path}_0f|_{\op{Trath}_0A}$ is valued in $\op{Trath}_0B$ and is therefore a chain map from $\op{Trath}_0A$ to $\op{Trath}_0B\,.$ Since $\oplus$ is the coproduct functor $$\amalg:{\tt Ch}(R)\times{\tt Ch}(R)\to {\tt Ch}(R)\;$$ on the product category, we have
$$
\op{Path}_0(g\circ f)=\amalg\big(g\circ f, g[-1]\circ f[-1]\big)=\amalg\big((g,g[-1])\circ(f,f[-1])\big)=
$$
$$
\amalg\big(g,g[-1]\big)\circ\amalg\big(f,f[-1]\big)=\op{Path}_0g\circ\op{Path}_0f\;,
$$
and the same result obviously holds for the restriction $\op{Trath}_0(g\circ f)$ of $\op{Path}_0(g\circ f)\,.$ Since $\op{Path}_0$ and $\op{Trath}_0$ clearly preserve identities, they are endofunctors of ${\tt Ch}(R)$ and ${\tt Ch}_{\ge 0}(R)\,,$ respectively. Further, the projection $\zp_B:\op{Path}_0B\to B$ onto the first term of $\op{Path}_0B$ is visibly a degree-wise surjective chain map, i.e., a fibration of ${\tt Ch}(R)\,,$ and the projection $\zt_B:\op{Trath}_0B\to B$ onto the first term is a chain map that is surjective in positive degrees, i.e., is a fibration of ${\tt Ch}_{\ge 0}(R)\,.$ Since $\id_B$ is a quasi-isomorphism its mapping cone $\op{Mc}(\id_B)$ has vanishing homology and so has its shift $\op{Path}_0B$ and the sub-complex $\op{Trath}_0B\,.$ This means that the morphisms $0\to \op{Path}_0B$ and $0\to \op{Trath}_0B$ are quasi-isomorphisms, so that the based path spaces in ${\tt Ch}(R)$ and ${\tt Ch}_{\ge 0}(R)$ are acyclic. Finally, the transformations $\zp:\op{Path}_0\to \id_{{\tt Ch}(R)}$ and $\zt:\op{Trath}_0\to \id_{{\tt Ch}_{\ge 0}(R)}$ are clearly natural, so that $\op{Path}_0$ and $\op{Trath}_0$ are actually based path space functors in ${\tt Ch}(R)$ and ${\tt Ch}_{\ge 0}(R)\,,$ respectively. \medskip

To compute Puppe's sequence, we still need the loop space functor, the homotopy kernel and the connecting morphism.\medskip

By definition the loop space of $B$ is the kernel of the fibration $\zp_B$ or $\zt_B\,.$ It is easy to see that in the unbounded case, the loop space $\zW B$ is \be\label{Loop}\zW B = B[-1]\ee with differential $d_{\zW B}=d_{B[-1]}=-d_B\,,$ and that in the non-negatively graded case, the loop space $\zY B$ is the truncation sub-complex \be\label{LoopNNG}\zY B:\;\;\cdots\longrightarrow B_2\stackrel{-d_{B,2}}{\longrightarrow}\ker d_{B,1}\stackrel{-d_{B,1}}{\longrightarrow}0\;\ee of $\zW B\,.$ Moreover, the universal morphisms $\zW f$ and $\zY f$ associated to a chain map $f:A\to B$ are obviously the chain map \be\label{Loopf}\zW f=f[-1]\ee and its restriction \be\label{LoopfNNG}\zY f=f[-1]|_{\zY A}\;,\ee respectively. \medskip

Remember now that we chose the based path space functor so that the homotopy kernel $K_f$ of a chain map $f:A\to B$ should be the $(-1)$--shift of this map's mapping cone $\op{Mc}(f)\,.$ A direct computation shows that in the unbounded case the pullback $K_f:=\op{Path}_0B\times_BA$ is actually given by
\be\label{HoKer} K_f=\op{Mc}(f)[-1]\;,\ee i.e., that we have
$$
(K_f)_n=A_n\oplus B_{n+1}\quad\text{and}\quad \op{d}_f:=d_{K_f}=\left(
                                      \begin{array}{cc}
                                        d_A & 0 \\
                                        -f & -d_B \\
                                      \end{array}
                                    \right)\;\;.
$$
Indeed, it is straightforward to check that the arrows of the square in the diagram
\be\label{Univ}
\begin{tikzcd}
C_n\ar[dr,dashed,near end,"\zm"]\ar[rrd,"\zvf_1"]\ar[ddr,swap,"\zvf_2"]&&\\
&A_n\oplus B_{n+1}\ar[d,"\zp_f"]\ar[r,swap,"f\oplus\id\!{[-1]}"]&B_n\oplus B_{n+1}\ar[d,"\zp_B"]\\
&A_n\ar[r,swap,"f"]&B_n
\end{tikzcd}
\ee
are chain maps and that the square commutes. Further assume that $\zvf_1:C\to \op{Path}_0B$ and $\zvf_2:C\to A$ are chain maps such that $\zp_B\circ\zvf_1=f\circ\zvf_2\,.$ If we set $\zvf_1c=(\zvf_1c)_n+(\zvf_1c)_{n+1}\,,$ the preceding commutation information reads \be\label{Info}(\zvf_1c)_n=f(\zvf_2c)\;.\ee Now, if the universal arrow $\zm$ exists we have necessarily
\be\label{Nec}
(\zm c)_n=\zvf_2c\quad\text{and}\quad(\zm c)_{n+1}=(\zvf_1 c)_{n+1}\;,
\ee
so that it is unique. A short computation that uses \eqref{Info} shows that conversely the map $\zm$ defined by \eqref{Nec} is a chain map that makes the two triangles in \eqref{Univ} commute. In the non-negatively graded case, the homotopy kernel $\Im_f$ is again the truncation sub-complex, i.e., in positive degrees $\Im_f$ coincides with $K_f$ and in degree zero $\Im_f$ is given by
\be\label{HoKerNNG}
(\Im_f)_0=\ker\op{d}_{f,0}=\{(a_0,b_1)\in A_0\oplus B_1:d_B(b_1)=-f(a_0)\}\;.
\ee

Finally, we defined the connecting morphism $\zd_f$ as the universal map $\zW B\dashrightarrow K_f$ associated to the inclusion $\zvf_1$ and the zero morphism $\zvf_2\,.$ From \eqref{Nec} it follows that $\zd_f$ is in the unbounded case the inclusion \be\label{ConneMorph}\zd_f=i[-1]:B[-1]\to \op{Mc}(f)[-1]\;,\ee where $i:B\to \op{Mc}(f)$ is the injection of Equation \eqref{MapLabels}. In the non-negatively graded case, the universal connecting morphism $\zy_f:\zY B\dashrightarrow\Im_f$ coincides with $i[-1]$ in positive degrees and $\zy_{f,0}:\ker d_{B,1}\dashrightarrow \ker\op{d}_{f,0}$ is the inclusion \be\label{ConneMorphNNG}\zy_{f,0}=i[-1]_0|_{\ker d_{B,1}}\ee as $\zvf_ {1,0}$ is the inclusion from $\ker d_{B,1}$ to $$\ker \op{d}_{B,0}=\{(-d_Bb,b):b\in B_1\}\;.$$ Let us still mention that $\zp_f:K_f\to A$ is in the unbounded case the projection \be\label{Proj}\zp_f=p[-1]:\op{Mc}(f)[-1]\to A\,,\ee where $p:\op{Mc}(f)\to A[1]$ is the projection of Equation \eqref{MapLabels}, and that in the non-negatively graded case it coincides with $p[-1]$ in all positive degrees and with the restriction \be\label{ProjNNG}\zp_{f,0}=p[-1]_0|_{\ker \op{d}_{f,0}}\ee in degree 0.\medskip

We are now prepared to compute Puppe's long homotopy fiber sequence and the corresponding long exact sequences in sets. We already mentioned previously (see Equation \eqref{1:1CF}) that there is a 1:1 correspondence between the set $[A,B]$ of ${\tt Ho(M)}$--morphisms from a cofibrant object $A$ of a model category $\tt M$ to a fibrant object $B$ and the set $\op{Hom}_{\tt M}(A,B)/\simeq$ of homotopy classes of $\tt M$--morphisms from $A$ to $B\,.$ Let us remember that two morphisms from a cofibrant $A$ to a fibrant $B$ are homotopic if and only if they are right homotopic. For ${\tt M=Ch}$ with ${\tt Ch=Ch}(R)$ or ${\tt Ch=Ch}_{\ge 0}(R)\,,$ two chain maps from a cofibrant $A$ to any $B$ are homotopic if and only if they are chain homotopic \cite[Theorem 2.3.11]{Ho99}. Moreover, the chain complex $R$ concentrated in degree 0 (with zero differential) is cofibrant \cite[Lemma 2.3.6]{Ho99}, so that $$[R,A]\cong \op{Hom}_{\tt Ch}(R,A)/\simeq\;.$$ Since the ring $R$ with unit $1$ is a free $R$-module with basis $1\,,$ a degree zero $R$-linear map $\mathfrak{f}:R\to A$ is fully determined by the image $\mathfrak{f}(1)\in A_0$ and a chain map $\mathfrak{f}:R\to A$ can be identified with the image $\mathfrak{f}(1)\in\ker d_{A,0}\;:$ there is a 1:1 correspondence $$\flat:\op{Hom}_{\tt Ch}(R,A)\ni\mathfrak{f}\mapsto \mathfrak{f}(1)\in\ker d_{A,0}\;.$$ Further, two chain maps $\mathfrak{f},\mathfrak{g}:R\to A$ are homotopic if and only if there is an $R$-linear map $\mathfrak{h}:R\to A_1$ such that $\mathfrak{f}-\mathfrak{g}=d_{A,1}\circ \mathfrak{h}\,,$ or, equivalently, there is a $1$--chain $\mathfrak{h}(1)\in A_1$ such that $\mathfrak{f}(1)-\mathfrak{g}(1)=d_{A,1}(\mathfrak{h}(1))\,.$ This means that $$\mathfrak{f}\simeq \mathfrak{g}\quad\text{if and only if}\quad \flat(\mathfrak{f})-\flat(\mathfrak{g})\in\op{im}d_{A,1}\;.$$ Hence $\flat$ induces a 1:1 correspondence \be\label{IdObj}\flat_\sharp:[R,A]\cong \op{Hom}_{\tt Ch}(R,A)/\simeq\;\;\ni[\mathfrak{f}]_\simeq\mapsto [\mathfrak{f}(1)]_{\op{im}}\in H_0(A)\;.\ee Since $[\mathfrak{f}]_\simeq=\zg\mathfrak{f}$ (see \eqref{GaCF}), if $f\in\op{Hom}_{\tt Ch}(A,B)\,,$ then $\zg f\in[A,B]$ and $[R,\zg f]$ is the set-theoretical map $$[R,\zg f]:[R,A]\ni [\mathfrak{f}]_\simeq \mapsto [f\circ\mathfrak{f}]_\simeq\in [R,B]\;.$$ If we read this map through the correspondence \eqref{IdObj}, we get the map \be\label{IdMor} H_0(f): H_0(A)\ni[\mathfrak{f}(1)]_{\op{im}}\mapsto [f(\mathfrak{f}(1))]_{\op{im}}\in H_0(B)\;.\ee

In the case ${\tt M}={\tt Ch}(R)\,,$ if we apply \eqref{LESH2} with $M=R\,,$ take into account \eqref{Loop}, \eqref{Loopf}, \eqref{HoKer}, \eqref{ConneMorph} and \eqref{Proj}, and use the identifications \eqref{IdObj} and \eqref{IdMor}, we find the long exact sequence
$$
\cdots\longrightarrow H_2(B)\stackrel{H_2(i)}{\longrightarrow} H_2(\op{Mc}(f))\stackrel{H_2(p)}{\longrightarrow} H_1(A)\stackrel{H_1(f)}{\longrightarrow}
$$
$$
H_1(B)\stackrel{H_1(i)}{\longrightarrow} H_1(\op{Mc}(f))\stackrel{H_1(p)}{\longrightarrow} H_0(A)\stackrel{H_0(f)}{\longrightarrow} H_0(B)\;,
$$
which is the left hand side of the homology sequence \eqref{LESWei}. In the case ${\tt M}={\tt Ch}_{\ge 0}(R)\,,$ the additional equations \eqref{LoopNNG}, \eqref{LoopfNNG}, \eqref{HoKerNNG}, \eqref{ConneMorphNNG} and \eqref{ProjNNG} lead to the same homology sequence (since in degree zero the homologies of the sub-complexes coincide with those of the full complexes), which in this case is the complete sequence. In the unbounded case, we can also get the complete sequence, or, more precisely, we can extend the homology sequence obtained to the right up to any degree $-n$ ($n>0$). For this it suffices to replace in the construction above the chain map $f$ by the chain map $f[n]$ and to observe that $$H_k(K_{f[n]})=H_k(K_f[n])=H_{k-n}(K_f)\;.$$

\subsection{Monoidal model categories}

In this subsection we introduce a suitable based path space functor in a pointed monoidal model category. First we recall the

\begin{defi}
A {\bf closed symmetric monoidal category} is a symmetric monoidal category $({\tt C}, \otimes, \mathbb{I})$ such that for every $B\in\tt C$ the functor $-\0 B:\tt C\to C$ has a right adjoint $\ul{\op{Hom}}_{\tt C}(B,-)\,,$ i.e., there exists a functor $\ul{\op{Hom}}_{\tt C}(B,-)$ together with a family of bijections
\be\label{Closed}\op{Hom}_{\tt C}(A\otimes B,C)\cong\op{Hom}_{\tt C}(A,\underline{\op{Hom}}_{\tt C}(B,C))\;\ee
indexed by $A,C\in \tt C$ that is natural in $A$ and $C\,.$
\end{defi}

A closed symmetric monoidal category is a closed category. In particular, there is an {\bf internal Hom} functor $$\ul{\op{Hom}}_{\tt C}:{\tt C}^{\op{op}}\times{\tt C}\to {\tt C}$$ such that if we fix the first argument we get the right adjoint of the definition. Moreover, there is a natural isomorphism $$\ul{\op{Hom}}_{\tt C}(\mathbb{I},-)\cong\id_{\tt C}\;$$ that allows us to identify these endofunctors.\medskip

We will use the following definition of a monoidal model category.

\begin{defi}\label{MMC} A {\bf monoidal model category} is a closed symmetric monoidal category $({\tt M},\0,\mathbb{I},\ul{\op{Hom}}_{\tt M})$ equipped with a model structure such that the following two compatibility conditions are fulfilled:
\begin{enumerate}
  \item[(i)]  the monoidal unit $\mathbb{I}$ is a cofibrant object,
  \item[(ii)]  for every cofibration $i:A\rightarrowtail B$ and every fibration $p:K\twoheadrightarrow L\,,$ the universal map
\begin{equation}\label{homs}
\underline{\op{Hom}}_{\tt M}(B,K)\dashrightarrow \underline{\op{Hom}}_{\tt M}(B,L)\times_{\underline{\op{Hom}}_{\tt M}(A,L)} \underline{\op{Hom}}_{\tt M}(A,K)
\end{equation}
is a fibration which is a weak equivalence if $i$ or $p$ is.
\end{enumerate}
\end{defi}

\begin{rem} \emph{It follows from \cite[Lemma 4.2.2.]{Ho99} that Condition $(ii)$ is equivalent to the pushout-product axiom}\smallskip

\emph{{\small PPA}: If $i:A\rightarrowtail B$ and $j:K\rightarrowtail L$ are cofibrations, the universal morphism
$$A\otimes L\coprod_{A\otimes K} B\otimes K\dashrightarrow B\otimes L$$
is a cofibration which is a weak equivalence if $i$ or $j$ is.}\medskip

\emph{Further, Condition $(i)$ and the {\small PPA} imply the unit axiom}\smallskip

\emph{{\small UA}: for every cofibrant replacement $C\mathbb{I}\stackrel{q}{\rightarrow}\mathbb{I}$ of $\mathbb{I}$ and every cofibrant object $X$ the morphism $$C\mathbb{I}\otimes X\stackrel{q\0\id_X}\rightarrow X$$ is a weak equivalence. Indeed, the {\small PPA} implies that $-\0 X$ preserves trivial cofibrations, so that because of Brown's lemma $-\0 X$ sends weak equivalences between cofibrant objects to weak equivalences. Since $C\mathbb{I}\stackrel{q}{\rightarrow}\mathbb{I}$ is if $\mathbb{I}$ is cofibrant a weak equivalence between cofibrant objects, the conclusion follows.}\medskip

\emph{Hence Definition \ref{MMC} is a little stronger than the standard definition which requires that the axioms {\small UA} and {\small PPA} are fulfilled.}
\end{rem}

For example, if $R=k$ is a commutative unital ring, the category $$\big({\tt Ch}(k),\0_k,k,\ul{\op{Hom}}_{{\tt Ch}(k)}\big)$$ of chain complexes in the category ${\tt Mod}(k)$ of modules over $k$ with its projective model structure is a (pointed) monoidal model category in the sense of Definition \ref{MMC}. This follows from \cite[Proposition 4.2.13]{Ho99} and the observation in the previous subsection that $R=k$ is cofibrant. Let us also remind that the tensor product is defined by
$$(A\otimes_k B)_n=\bigoplus_{\zm+\zn=n}A_\zm\otimes_k B_\zn\,,\quad d_\0(a\otimes b)=d_Aa\otimes b+(-1)^{\zm}a\otimes d_Bb\;$$ and that the internal Hom is given by
\be\label{IntHom}{\underline{\op{Hom}}}_{{\tt Ch}(k)}(A,B)_n=\prod_{\zm\in \mathbb{Z}}\op{Hom}_{{\tt Mod}(k)}(A_\zm,B_{\zm+n})\,,\quad(\ul{d}f)_\zm=d_B\circ f_\zm+(-1)^{n+1}f_{\zm-1}\circ d_A\;.\ee

Let now $(\tt M,\0,\mathbb{I},\ul{\op{Hom}}_{\tt M})$ be any pointed monoidal model category and let \be\label{UnitCone}\mathbb{I}\rightarrowtail\op{Cone}(\mathbb{I})\stackrel{\sim}{\to}0\ee be a factorization of $\mathbb{I}\to 0$ into a cofibration followed by a weak equivalence. It is easy to check that if ${\tt M=Ch}(k)$ and thus $\mathbb{I}=k\,,$ the mapping cone
\be\label{Mck}
\op{Mc}(\id_k):\;\cdots\longrightarrow 0\longrightarrow \underbrace{k}_{(1)}\stackrel{\id_k}{\longrightarrow}\underbrace{k}_{(0)}\longrightarrow0\longrightarrow\cdots\;,
\ee
where the integers in parentheses indicate the degree, satisfies the condition $$k\rightarrowtail\op{Mc}(\id_k)\stackrel{\sim}{\to}0\;,$$ so that we can choose \be\label{Conek}\op{Cone}(k)=\op{Mc}(\id_k)\;.\ee Indeed, the morphism $k\to\op{Mc}(\id_k)$ is the morphism $S^0(k)\to D^1(k)$ from the 0-sphere at $k$ to the 1-disk at $k$ which is a generating cofibration and the morphism $\op{Mc}(\id_k)\to 0$ is a quasi-isomorphism as $$H(\op{Mc}(0:\op{Mc}(\id_k)\to 0))=H(\op{Mc}(\id_k)[1])=0$$ since $\id_k$ is a quasi-isomorphism. This justifies the notation $\op{Come}(\mathbb{I})$ and the name `cone of $\mathbb{I}$'. It is now natural to define the cone functor $\op{Cone}\in{\tt Fun}(\tt M, M)$ by $$\op{Cone}:=-\0\op{Cone}(\mathbb{I})$$ and the dual based path space functor $\textsf{Path}_0\in\tt Fun(M,M)$ by $$\textsf{Path}_0:=\ul{\op{Hom}}_{\tt M}(\op{Cone}(\mathbb{I}),-)\;.$$

\begin{prop}\label{BPSFMMC}
In a pointed monoidal model category $(\tt M,\0,\mathbb{I},\ul{\op{Hom}}_{\tt M})$ with unit cone $\op{Cone}(\mathbb{I})$ defined by \eqref{UnitCone}, the functor \emph{\be\label{BPSPMMCEq}\textsf{Path}_0:=\ul{\op{Hom}}_{\tt M}(\op{Cone}(\mathbb{I}),-)\;\ee} is a based path space functor in the sense of Definition \ref{BPSF}.
\end{prop}

\begin{proof} For every $A,B\in\tt M\,,$ we have $$\op{Hom}_{\tt M}(A,\ul{\op{Hom}}_{\tt M}(B,0))\cong\op{Hom}_{\tt M}(A\0 B,0)\;,$$ so that there is a unique morphism from every $A\in\tt M$ to $\ul{\op{Hom}}_{\tt M}(B,0)\in\tt M\,,$ which means that for every $B\in\tt M\,,$ we have $$\ul{\op{Hom}}_{\tt M}(B,0)\cong 0\;.$$ The same result holds for $\ul{\op{Hom}}_{\tt M}(0,B)\,.$ Indeed, since left adjoint functors preserve colimits, the functor $-\0 A$ preserves the initial object, so that $A\0 0\cong 0\0 A\cong 0\,.$ Hence, for every $A\in\tt M\,,$ we get $$\op{Hom}_{\tt M}(A,\ul{\op{Hom}}_{\tt M}(0,B))\cong \op{Hom}_{\tt M}(0,B)\;\,,$$ which implies that $$\ul{\op{Hom}}_{\tt M}(0,B)\cong 0\;.$$

If $i:\mathbb{I}\rightarrowtail\op{Cone}(\mathbb{I})\,,$ then for every $A\in\tt M\,,$ we have a morphism $$\ul{\op{Hom}}_{\tt M}(i,A):\ul{\op{Hom}}_{\tt M}(\op{Cone}(\mathbb{I}),A)\to\ul{\op{Hom}}_{\tt M}(\mathbb{I},A)\;$$ and these morphisms are the components of a natural transformation $$\zp:\textsf{Path}_0\Rightarrow\id_{\tt M}\;.$$ Indeed, if $f:A\to B$ is a morphism, then  $f\circ\zp_A=\zp_B\circ\textsf{Path}_0f\,,$ since the left hand side $\ul{\op{Hom}}_{\tt M}(\id_{\mathbb{I}},f)\circ\ul{\op{Hom}}_{\tt M}(i,\id_A)$ and the right hand side $\ul{\op{Hom}}_{\tt M}(i,\id_B)\circ\ul{\op{Hom}}_{\tt M}(\id_{\op{Cone}(\mathbb{I})},f)$ are both equal to $\ul{\op{Hom}}_{\tt M}(i,f)\,,$ as $\ul{\op{Hom}}_{\tt M}$ is a functor on the product category ${\tt M}^{\op{op}}\times{\tt M}$ with composition $$(g,g')\circ_\times(f,f')=(g\circ_{{\tt M}^{\op{op}}} f,g'\circ_{\tt M} f')=(f\circ_{\tt M} g,g'\circ_{\tt M} f')\;,$$ for all ${\tt M}^{\op{op}}$--morphisms $f:A\to B$ and $g:B\to C$ and all $\tt M$--morphisms $f':A'\to B'$ and $g':$ $B'\to C'\,.$\medskip

Let now $A\in{\tt M}$ be a fibrant object. If we apply the axiom $(ii)$ of Definition \ref{MMC} to the cofibration $i:\mathbb{I}\rightarrowtail\op{Cone}(\mathbb{I})$ and the fibration $p:A\twoheadrightarrow 0\,,$ we find that the universal map $$\textsf{Path}_0A=\ul{\op{Hom}}_{\tt M}(\op{Cone}(\mathbb{I}),A)\dashrightarrow \ul{\op{Hom}}_{\tt M}(\op{Cone}(\mathbb{I}),0)\times_{\ul{\op{Hom}}_{\tt M}(\mathbb{I},0)}\ul{\op{Hom}}_{\tt M}(\mathbb{I},A)= A$$ is a fibration. Further, since the morphism $\op{Cone}(\mathbb{I})\to 0$ is a weak equivalence, it follows from the 2-out-of-3 axiom that the morphism $0\to\op{Cone}(\mathbb{I})$ is also a weak equivalence. As $\mathbb{I}$ is cofibrant by definition, this morphism $\iota: 0\rightarrowtail\mathbb{I}\rightarrowtail\op{Cone}(\mathbb{I})$ is a trivial cofibration. If we apply now the axiom $(ii)$ to $\iota$ and $p\,,$ we get that the universal map $$\textsf{Path}_0A=\ul{\op{Hom}}_{\tt M}(\op{Cone}(\mathbb{I}),A)\dashrightarrow \ul{\op{Hom}}_{\tt M}(\op{Cone}(\mathbb{I}),0)\times_{\ul{\op{Hom}}_{\tt M}(0,0)}\ul{\op{Hom}}_{\tt M}(0,A)=0$$ is a trivial fibration, so that $\textsf{Path}_0A$ is acyclic.
\end{proof}

\subsection{Chain complexes as monoidal model category}\label{ChCoII}

In Subsection \ref{ChCoI} we have applied the machinery of this paper to the pointed model category ${\tt Ch}(R)$ of chain complexes in the category $R-\tt Mod$ of left modules over a unital ring $R\,.$ It seemed natural to define the based path space $\op{Path}_0A$ of a complex $A$ as the mapping cone $\op{Mc}(\id_A)[-1]$ of the identity of $A$ shifted by $-1\,.$ In the case where $R$ is a {\it commutative} unital ring $k\,$, we can also consider ${\tt Ch}(R)={\tt Ch}(k)$ as a pointed monoidal model category, apply Proposition \ref{BPSFMMC} and define the based path space functor as the dual $\textsf{Path}_0$ of the cone $\op{Cone}(k)$ of $k\,.$ In this subsection we compute the based path space $\textsf{Path}_0A\,,$ the corresponding loop space, homotopy kernel... and compare the results to those of Subsection \ref{ChCoI}.\medskip

From Equations \eqref{BPSPMMCEq}, \eqref{Conek}, \eqref{Mck} and \eqref{IntHom}, it follows that
$$(\textsf{Path}_0A)_n=\underline{\op{Hom}}_{{\tt Ch}(k)}(\op{Cone}(k),A)_n=\op{Hom}_{{\tt Mod}(k)}(\underbrace{k}_{(0)},A_n)\,\oplus\,\op{Hom}_{{\tt Mod}(k)}(\underbrace{k}_{(1)},A_{n+1})\;,
$$
so that we have the isomorphism of $k$--modules
\be\label{iso}\cI_n:(\textsf{Path}_0A)_n\ni (f_\zm)_\zm = (f_0,f_1)\mapsto a_n+a_{n+1}:=f_0(1)+f_1(1)\in A_n\oplus A_{n+1}\;.\ee
If we read the differential \eqref{IntHom} through the isomorphisms $\cI_n\,,$ we get the differential $$\textsf{d}_A=\cI_{n-1}\circ\ul{d}\circ\cI_n^{-1}$$ given by
$$
\textsf{d}_A\left(\begin{array}{c}
            a_n \\
            a_{n+1}
          \end{array}
\right)= \cI_{n-1}(\ul{d}(f_0,f_1))=\cI_{n-1}(d_A\circ f_0+d_A\circ f_1+(-1)^{n+1}f_0)=
$$
$$
d_Aa_n+d_Aa_{n+1}+(-1)^{n+1}a_n=\left(\begin{array}{cc}
                                                                                                                               d_A & 0 \\
                                                                                                                               (-1)^{n+1} & d_A
                                                                                                                             \end{array}
\right)\left(\begin{array}{c}
               a_n \\
               a_{n+1}
             \end{array}
\right)\;,
$$
since $d_{\op{Cone}(k)}$ vanishes except if $\zm=1$ where it is the identity. Finally,
\be\label{Path}
(\textsf{Path}_0A)_n=A_n\,\oplus\,A_{n+1}\quad\text{and}\quad \textsf{d}_{A,n}=\left(\begin{array}{cc}
                                                                                                                               d_A & 0 \\
                                                                                                                               (-1)^{n+1}\id_A & d_A
                                                                                                                             \end{array}
\right)\;.
\ee

The $A$-component of the natural transformation $\zp:\textsf{Path}_0\Rightarrow\id_{{\tt Ch}(k)}$ is $\zp_A=\ul{\op{Hom}}_{{\tt Ch}(k)}(i,A)$ where the cofibration $i:k\rightarrowtail\op{Cone}(k)$ vanishes except in degree 0 where it is the identity. If we read $(\zp_A)_n$ through the isomorphism $\cI_n$ we get the canonical projection $\zp_1:A_n\,\oplus\,A_{n+1}\to A_n\,.$ The loop space of $A$ is its kernel
\be\label{Omega}
(\textsf{O} A)_n=A_{n+1}\quad\text{and}\quad \textsf{d}_{\textsf{O} A,n}=d_{A,n+1}\;.
\ee
A direct verification shows that given a morphism $f:A\to B\,,$ the homotopy kernel $\textsf{K}_f=\textsf{Path}_0B\times_BA$ is
\be\label{Kaef}
(\textsf{K}_f)_n=A_{n}\,\oplus\,B_{n+1}\quad\text{and}\quad
\textsf{d}_{f,n}=
\begin{pmatrix}
d_A & 0 \\
(-1)^{n+1}f & d_B
\end{pmatrix}\;.
\ee

The mapping cone $\op{Mc}(f)$ of a chain map $f:A\to B$ is defined in \eqref{Mc}. An alternative definition is
\be\label{McAlt}\op{Mc}(f)_n:=A[1]_n\oplus B_n\quad\text{and}\quad \textsf{d}_{\op{Mc}(f),n}:=\left(
                                      \begin{array}{cc}
                                        d_{A} & 0 \\
                                        (-1)^nf & d_B \\
                                      \end{array}
                                    \right)\;\;.
\ee
Equations \eqref{Path}, \eqref{Omega} and \eqref{Kaef} show that the complexes obtained here and in Subsection \ref{ChCoI} are the same graded modules. As for their differentials, we found in \ref{ChCoI} that
$$
\op{d}_{A,n}:=d_{\op{Path}_0A,n}=-d_{\op{Mc}(\id_A),n+1}\;,\quad d_{\zW A,n}=-d_{A,n+1}\;\quad\text{and}\quad \op{d}_{f,n}:=d_{K_f,n}=-d_{\op{Mc}(f),n+1}\;.
$$
Here we find the same differentials, but without the sign change in the right hand side and with the standard differential $d_{\op{Mc}(f)}$ replaced by the previous alternative differential $\textsf{d}_{\op{Mc}(f)}\,:$
$$
\textsf{d}_{A,n}:=d_{\textsf{Path}_0A,n}=\textsf{d}_{\op{Mc}(\id_A),n+1}\;,\quad \textsf{d}_{\textsf{O}A,n}=d_{A,n+1}\;\quad\text{and}\quad \textsf{d}_{f,n}:=d_{\textsf{K}_f,n}=\textsf{d}_{\op{Mc}(f),n+1}\;.
$$
These slight differences are of course completely irrelevant.

\section{Follow up questions}\label{Conclusions}

In a triangulated category, every morphism $f:A\to B$ has a cone $B\to C(f)\to A[1]$ such that $A\to B\to C(f)\to A[1]$ is a distinguished triangle. However, the cone $C$ is not a functor. It has been mentioned in the literature \cite{GM} that this drawback is a sign that the axioms of a triangulated category are suboptimal. More precisely, if $C(f)[-1]\to A\to B$ and $C(f')[-1]\to A'\to B'$ are distinguished triangles together with a commutative square $S:=(A,B,A',B')\,,$ the is no unique induced map $C(S):C(f)[-1]\to C(f')[-1]$ that makes $C$ a functor. We expect that Theorem \ref{EquivHoCat} and Proposition \ref{Explicit} can be used to suggest a definition of triangulated categories with a functorial cone.\medskip

In \cite{KTRCR, HAC} and \cite{PP}, Di Brino and two of the authors of the present work have taken up ideas from \cite{BD04, CG1, P11, Paugam1, TV05, TV08, Vino} and have introduced homotopical algebraic geometry over the ring of differential operators as a suitable framework for investigating the solution space of partial differential equations modulo symmetries. The implementation of the associated research program requires that a certain quintuplet be a homotopical algebraic geometric context (HAGC) in the sense of \cite{TV08}. We are convinced that the theory of homotopy fiber sequences, which we have detailed in this paper, will enable us to prove the HAGC theorem and thus to take an important step towards fully working through the above program.

{}
\vfill
{\emph{email:} {\sf alisa.govzmann@uni.lu}; \emph{email:} {\sf damjan.pistalo@uni.lu}; \emph{email:} {\sf norbert.poncin@uni.lu}.}

\end{document}